\documentclass{article}
\usepackage{graphicx} 
\usepackage{amsmath,amsfonts,amssymb,amsthm,enumerate}
\usepackage{color}
\usepackage[margin=1in]{geometry}

\usepackage{multirow}
\usepackage{booktabs}
\usepackage{threeparttable}
\usepackage{algorithm, algpseudocode}

\usepackage{hyperref}
\newtheorem{theorem}{Theorem}

\begin{document}
\title{The Adaptive Solution of High-Frequency Helmholtz Equations via Multi-Grade Deep Learning}
\author{Peiyao Zhao\thanks{School of Mathematics, Jilin University, Changchun, 130012, P. R. China. E-mail address: {\it zhaopy23@mails.jlu.edu.cn}.}, Rui Wang\thanks{School of Mathematics, Jilin University, Changchun 130012, P. R. China. E-mail address: {\it rwang11@jlu.edu.cn}. All correspondence should be sent to this author.}, Tingting Wu\thanks{School of Mathematics and Statistics, Shandong Normal University, Jinan 250358, P. R. China. E-mail address: {\it tingtingwu@sdnu.edu.cn}.} \ and \ Yuesheng Xu\thanks{Department of Mathematics and Statistics, Old Dominion University, Norfolk, VA 23529, United States of America. E-mail address: {\it y1xu@odu.edu}. All correspondence should be sent to this author.}}

\date{}
\maketitle
\begin{abstract}
The Helmholtz equation is fundamental for modeling wave propagation in acoustics, electromagnetics, and geophysics; however, high-frequency regimes remain notoriously difficult due to the severe numerical ``pollution effect.'' We propose FD-MGDL, an adaptive framework that synergizes finite difference discretizations with Multi-Grade Deep Learning (MGDL) to efficiently resolve high-frequency wavefields. Unlike standard physics-informed neural networks (PINNs), which frequently suffer from spectral bias and heavy automatic differentiation overhead, FD-MGDL employs a progressive, grade-wise training strategy that incrementally incorporates shallow sub-networks to refine residual errors. By leveraging ReLU activations in refinement grades, the framework reformulates the highly non-convex global optimization problem into a sequence of tractable convex subproblems, dramatically improving training stability and convergence reliability. Extensive numerical experiments in two and three dimensions with wavenumbers up to $\kappa=200$ demonstrate that FD-MGDL significantly outperforms single-grade networks and standard benchmark neural solvers in both accuracy and computational efficiency. When applied to an inhomogeneous concave velocity model, the proposed method accurately captures wave focusing and caustic formations, markedly outperforming classical five-point finite difference schemes in resolving sharp phase transitions and peak amplitudes. These results establish FD-MGDL as a robust, scalable, and mathematically grounded paradigm for high-frequency wave simulation in complex media.
\end{abstract}


\section{Introduction}
\label{introduction}

Deep learning has rapidly expanded beyond traditional data-driven tasks into AI-driven scientific computing, particularly for solving partial differential equations (PDEs). Representative advances include Physics-Informed Neural Networks (PINNs) \cite{raissi2019physics}, which incorporate physical laws into neural network training for forward and inverse problems, and the Fourier Neural Operator (FNO) \cite{li2021fourier}, which learns discretization-invariant mappings between function spaces. Deep learning has also been applied to high-dimensional PDEs through backward stochastic differential equations (BSDEs) \cite{e2017deep,han2018solving}, mitigating the curse of dimensionality. The Deep Ritz Method \cite{e2018deep} and the Deep Galerkin Method (DGM) \cite{sirignano2018dgm} provide mesh-free neural approaches for variational and high-dimensional problems. More broadly, neural operators \cite{kovachki2023neural} generalize neural networks to learn solution operators independent of discretization. Multi-scale architectures such as MscaleDNN \cite{liu2020multi} incorporate frequency-domain techniques, while the Koopman Neural Operator \cite{xiong2024koopman} reformulates PDE dynamics as linear prediction problems to enhance long-term accuracy. Together, these developments demonstrate the transformative potential of deep learning in scientific computing.

Deep neural networks (DNNs) are increasingly utilized to solve the Helmholtz equation in acoustics, electromagnetics, and wave propagation. Recent developments encompass meshless ray-based DNNs \cite{yang2023novel,yang2022mesh} that bypass adaptive meshing in high-frequency regimes, ray-based frameworks \cite{yeung2022learning} that extract directional wave features to enhance discontinuous Galerkin methods, hybrid learning–multigrid strategies \cite{azulay2022multigrid} employing CNN preconditioners for heterogeneous media, and deep architectures for resonator inverse design \cite{dogra2023deep}. While Physics-Informed Neural Network (PINN) formulations \cite{cho2023separable,schoder2024feasibility,urban2025unveiling,wang2021understanding} enable real-time optimization, their scope remains largely restricted to low wavenumbers ($\kappa < 10$) due to severe spectral bias and optimization stiffness. Furthermore, most existing neural solvers focus on wavenumber-linear phase factors ($e^{i\kappa x}$) rather than pure trigonometric variations ($\sin(\kappa x)$), leaving high-frequency, strongly oscillatory regimes, particularly in three dimensions, insufficiently addressed.

High-wavenumber Helmholtz problems present severe computational bottlenecks due to the rapidly oscillatory nature of their solutions. Fundamentally, this challenge stems from the spectral bias (or frequency principle) of DNNs \cite{basri2019convergence,xu2020frequency}, wherein standard architectures inherently prioritize low-frequency components during training. As $\kappa$ increases, solution dynamics shift toward higher frequencies where conventional networks stagnate or diverge \cite{lerer2024multigrid,song2023simulating}. Although DNNs are universal approximators, resolving such rapid spatial variations without specialized architectural interventions or extreme network depth leads to severe scaling limitations as problem complexity grows \cite{lerer2024multigrid,liu2020multi}.

Beyond spectral bias, high-frequency Helmholtz problems exhibit rapid spatial variations, a challenge directly analogous to the ``pollution effect" in classical numerical analysis, which makes accurate wavefield resolution exceptionally difficult for standard Physics-Informed Neural Network (PINN). This issue is further compounded by the heavy computational burden of repeated automatic differentiation required to evaluate differential operators in the high-frequency regime. Consequently, DNN solvers have largely remained confined to two-dimensional problems at modest wavenumbers, leaving the high-$\kappa$ regime underexplored.

To address these limitations, we propose a finite difference–based multi-grade deep learning (FD-MGDL) framework. Building on the Multi-Grade Deep Learning (MGDL) methodology \cite{fang2024addressing,jiang2024deep,xu2025multi,xu2025sal}, which systematically mitigates spectral bias through grade-wise residual learning, we incorporate finite difference discretization to replace the computationally expensive automatic differentiation (AD) operators used in standard PINNs. We introduce an adaptive strategy that dynamically adjusts the number of grades based on problem difficulty. The initial grade employs a shallow network with two hidden layers to capture global wave structures, whereas each subsequent grade incrementally adds a single hidden layer to resolve residual errors. Unlike conventional SGDL models—which require redesigning the architecture and retraining the entire network from scratch if target accuracy is not reached—MGDL adaptively appends new grades until the precision requirement is satisfied. Guided by the theoretical foundation of MGDL, this work investigates the numerical performance of FD-MGDL on Helmholtz equations across high-wavenumber regimes ($50 \le \kappa \le 200$).

We evaluate the proposed FD-MGDL against AD-MGDL, FD-SGDL and several benchmark methods, including PINN \cite{raissi2019physics}, Mscale \cite{liu2020multi}, FBPINN \cite{moseley2023finite}, SIREN \cite{sitzmann2020implicit}, and Pre-PINN \cite{ryck2024operator}. Comprehensive numerical experiments demonstrate that FD-MGDL consistently outperforms all baseline methods in accuracy, computational efficiency, and numerical stability across all tested scenarios.

To evaluate robustness in heterogeneous media, we apply FD-MGDL to a concave velocity model characterized by spatially varying velocities that induce lensing, caustics, and multipathing phenomena. These features give rise to localized phase transitions and amplitude concentrations that present severe challenges for numerical solvers. While traditional finite difference and finite element methods require extremely fine meshes to resolve such dynamics—incurring heavy computational costs and risking numerical dispersion, our 5-point FD-MGDL framework cleanly captures focal regions and rapid oscillations, substantially outperforming the classical 5-point finite difference scheme in both spatial resolution and numerical stability.

The core novelties of the proposed FD-MGDL framework are summarized as follows:
\begin{itemize}
    \item \textbf{Multi-Grade Residual Training:} Overcomes spectral bias by replacing a single deep network with sequential, shallow residual updates, reducing the cumulative optimization workload by over 80\%. This harnesses the expressive capacity of DNNs while bypassing optimization stiffness and facilitates seamless architectural adaptivity, thereby resolving the structural rigidity of traditional deep learning frameworks.
    
    \item \textbf{FD--Neural Hybridization:} Replaces automatic differentiation with finite difference stencils during residual training. This eliminates vanishing second-derivative signals on non-smooth activation functions (such as ReLU) while preserving a continuous, mesh-free representation of the wavefield.
    
    \item \textbf{Dual-Activation Strategy:} Combines sinusoidal activations in Grade 1 to capture global, highly oscillatory wave patterns with ReLU activations in subsequent grades to resolve localized residual errors without introducing spurious high-frequency ringing.
    
    \item \textbf{Monotonicity and Convexity Guarantees:} Establishes a theoretical monotonicity guarantee ensuring non-increasing loss as new grades are added, while reformulating individual shallow subproblems into equivalent convex optimizations to guarantee stable convergence.
\end{itemize}

The remainder of this paper is organized as follows. Section~\ref{FD-MGDL} introduces the overarching FD-MGDL framework. Section~\ref{Adaptive algorithm} details the adaptive training scheme, demonstrating how it recasts the highly non-convex optimization objective into a sequence of tractable convex subproblems. Section~\ref{details} provides implementation specifics alongside an ablation study on grade-wise depth allocation. Sections~\ref{2d} and \ref{3d} evaluate the method on two- and three-dimensional Helmholtz benchmarks, respectively. Section~\ref{FDM} presents a direct comparative analysis against classical finite difference methods. Section~\ref{FD_extensions} investigates the integration of higher-order accuracy discretizations into the loss formulation. Section~\ref{concave} examines performance on a complex heterogeneous domain using a concave velocity model. 
Finally, Section~\ref{conclusion} concludes the paper, while the Appendix provides supplementary material detailing the computational setup, efficiency analysis, and statistical robustness across random initializations.

\section{FD-MGDL for the Helmholtz equation}
\label{FD-MGDL}
In this section, we introduce a finite difference-based MGDL approach for the numerical solution of the Helmholtz equation.

We begin by describing the Helmholtz equation. For each $d\in\mathbb{N}$, we define $\mathbb{N}_d:=\{1,2,\ldots,d\}$ and $\mathbb{Z}_d:=\{0,1,\ldots,d-1\}$. Let $\Delta:=\sum_{j\in\mathbb{N}_d}\partial^2/\partial x_j^2$ denote the Laplace operator, $\kappa$ represent the wavenumber and $f$ be the source term. The Helmholtz equation, expressed as the elliptic partial differential equation 
\begin{equation*}
    (\Delta + \kappa^2)u=f,
\end{equation*}
represents a time-independent wave equation, where $u$ is the unknown solution to be learned.
We consider the boundary value problem for the Helmholtz equation 
\begin{equation}\label{Helmholtz}
\begin{cases}
\left( \Delta+\kappa^2(\mathbf{x}) \right)u(\mathbf{x})=f(\mathbf{x}),&\mathbf{x}\in \Omega, \\
\mathcal{B} \left( u(\mathbf{x}) \right)= 0,&\mathbf{x}\in \Gamma,
\end{cases}
\end{equation}
where $\Omega \subset \mathbb{R}^d$ is an open domain with boundary $\Gamma$, and $\mathcal{B}$ is the nonlinear operator representing the boundary conditions. 

DNN methods have emerged as a prominent research area for the numerical solution of PDEs. These approaches utilize DNNs as flexible parametric representations of unknown solutions, leveraging their universal approximation properties to handle complex functional forms. A DNN of depth $D \in \mathbb{N}$ is structured as a sequence of layers: an input layer, $D-1$ hidden layers, and a final output layer. Let $d, s \in \mathbb{N}$ represent the dimensions of the input and output spaces, respectively. For each layer $j \in \mathbb{Z}_{D+1}$, we denote the number of neurons by $d_j$, with the boundary conditions $d_0 = d$ and $d_D = s$. The network parameters for each layer $j \in \mathbb{N}_{D}$ are defined as weight matrix $\mathbf{W}_j \in \mathbb{R}^{d_j \times d_{j-1}}$ and bias vector $\mathbf{b}_j \in \mathbb{R}^{d_j}$. Let $\sigma: \mathbb{R} \to \mathbb{R}$ be a fixed activation function, applied componentwise when acting on vectors. For an input vector $\mathbf{x} = (x_1, \dots, x_d)^\top \in \mathbb{R}^d$, the transformation performed by the first hidden layer is expressed as
\begin{equation*}
    \mathcal{H}_1\left(\left\{\mathbf{W}_1, \mathbf{b}_1\right\};\mathbf{x}\right) := \sigma\left(\mathbf{W}_1 \mathbf{x} + \mathbf{b}_1\right).
\end{equation*}

For a DNN with depth $D\geq3$, the outputs of the subsequent hidden layer are defined recursively for $j\in\mathbb{N}_{D-2}$ as
\begin{equation*}
    \mathcal{H}_{j+1}\left(\left\{\mathbf{W}_i, \mathbf{b}_i:i\in\mathbb{N}_{j+1}\right\};\mathbf{x}\right):= \sigma\left(\mathbf{W}_{j+1} \mathcal{H}_j\left(\left\{\mathbf{W}_i, \mathbf{b}_i:i\in\mathbb{N}_{j}\right\};\mathbf{x}\right) + \mathbf{b}_{j+1}\right).
\end{equation*}
The final output of the DNN is represented by a vector-valued function $\mathcal{N}_D: \mathbb{R}^d \to \mathbb{R}^s$. This map, parameterized by the collection of weights and biases $\{\mathbf{W}_j, \mathbf{b}_j: j\in\mathbb{N}_D\}$, is given by:
\begin{equation*}\label{DNN}
    \mathcal{N}_D\left(\left\{\mathbf{W}_j, \mathbf{b}_j:j\in\mathbb{N}_D\right\}; \mathbf{x} \right) := \mathbf{W}_D \mathcal{H}_{D-1}\left(\left\{\mathbf{W}_j, \mathbf{b}_j:j\in\mathbb{N}_{D-1}\right\}; \mathbf{x} \right) + \mathbf{b}_D,\ \ \mathbf{x}\in\mathbb{R}^d.
\end{equation*}

The PINN framework \cite{raissi2019physics} utilizes a DNN to approximate the solution of a PDE by incorporating the governing equations and boundary conditions directly into the learning process. For the Helmholtz equation \eqref{Helmholtz}, the PINN loss function is typically composed of two primary terms: the PDE residual loss and the boundary condition loss. Let $\Theta := \{\mathbf{W}_j, \mathbf{b}_j:j\in\mathbb{N}_D\}$ denote the network parameters. For interior collocation points $\{\mathbf{x}_j:j\in\mathbb{N}_{N_{int}}\}$ and boundary points $\{\tilde{\mathbf{x}}_j:j\in\mathbb{N}_{N_b}\}$, the components of the loss function are defined as follows:
\begin{itemize}
    \item {\bf PDE Residual Loss}: Quantifies the degree to which the network output $\mathcal{N}_D$ satisfies the Helmholtz operator:
\begin{equation*}
    \mathcal{L}_{PDE}(\Theta) := \frac{1}{N_{int}} \sum_{j\in\mathbb{N}_{N_{int}}}\left | (\Delta+\kappa^2(\mathbf{x}_j) )\mathcal{N}_D(\Theta;\mathbf{x}_j)-f(\mathbf{x}_j) \right |^2, 
\end{equation*}
\item {\bf Boundary Condition Loss}: Enforces the boundary conditions $\mathcal{B}$ at the specified points:
\begin{equation*}
    \mathcal{L}_B(\Theta):=\frac{1}{N_{b}}\sum_{j\in\mathbb{N}_{N_b}}\left | \mathcal{B}(\mathcal{N}_D(\Theta;\tilde{\mathbf{x}}_j)) \right | ^2.
\end{equation*}
\end{itemize}
The total training objective is generally the weighted sum of these components, minimized through gradient-based optimization.

The selection of training points significantly impacts the convergence and accuracy of the model. Current research highlights two primary strategies:
\begin{enumerate}
    \item {\bf Random Sampling:} Points are drawn from prescribed probability distributions over the domain \cite{cen2024deep,raissi2019physics}.
    \item {\bf Adaptive Sampling:} Points are iteratively selected in regions where the PDE residual is high to improve the model's local resolution \cite{chen2024adaptive,gao2023failure}.
\end{enumerate}

The training of a PINN is framed as a multi-objective optimization problem. The total loss function is defined as a weighted sum of the PDE residual loss and the boundary condition loss:
\begin{equation}\label{loss-pinn}
    \mathcal{L}_{PINN}(\Theta):= \lambda_{PDE}\mathcal{L}_{PDE}\left( \Theta\right) + \lambda_B\mathcal{L}_{B}\left( \Theta\right).
\end{equation}
Core components of the model include
\begin{itemize}
    \item {\bf Weighting Coefficients:} The positive hyperparameters $\lambda_{PDE}$ and $\lambda_{B}$ scale the relative contribution of each loss term.
    \item {\bf Optimization Dynamics:} Selecting appropriate weights is critical; imbalances between terms can lead to ``gradient pathologies'', where a single term dominates the gradient flow and hinders overall convergence.
    \item {\bf Performance Impact:} Research indicates that the precise tuning of these coefficients dictates the computational efficiency, numerical stability, and convergence accuracy of the model \cite{liu2021dual,wang2021understanding,xiang2022self}.
\end{itemize}

The objective is to identify the optimal network parameters $\Theta^*$ by minimizing the total loss over the combined training dataset $\mathbb{D} := \{\mathbf{x}_j:j\in\mathbb{N}_{N_{int}}\}\cup\{\tilde{\mathbf{x}}_k:k\in\mathbb{N}_{N_b}\}$:
\begin{equation*}
    \Theta^* = \arg\min_{\Theta} \mathcal{L}_{PINN}(\Theta).
\end{equation*}
This minimization is typically performed using gradient-based optimization algorithms, such as Adam or L-BFGS, with gradients computed via backpropagation.

Despite its versatility, the standard PINN framework faces two primary challenges:
\begin{enumerate}
    \item {\bf Weak Enforcement of Constraints:} Because boundary conditions are enforced as soft constraints through the loss function, the resulting solution may exhibit inconsistencies or residuals at the boundaries.
    \item {\bf Optimization Stiffness:} The optimization landscape can become excessively ``stiff'' due to the competing objectives of the PDE and boundary terms. This competition often leads to training instabilities that can impede or entirely prevent convergence \cite{wang2021understanding,wang2022when}.
\end{enumerate}

To mitigate these issues, we adopt an approach in which derivative terms, typically computed via Automatic Differentiation (AD), are replaced by Finite Difference (FD) approximations when constructing the PDE loss. This formulation allows the governing equation and boundary conditions to be naturally incorporated into a single discrete system, from which a unified loss function is defined, thereby alleviating issues associated with weak constraint enforcement and optimization stiffness. We emphasize that the MGDL framework is flexible and not restricted to this particular setting; depending on the application, it can accommodate both AD- and FD-based formulations \cite{zeng2023multi}.

We consider the Helmholtz equation with Dirichlet boundary conditions on a bounded domain $\Omega \subset \mathbb{R}^d$ with boundary $\Gamma := \partial \Omega$ as a representative model problem:
\begin{equation}\label{Helmholtz-Dirichlet}
    \left\{\begin{aligned}
        &\left( \Delta+\kappa^2(\mathbf{x}) \right) u(\mathbf{x}) = f(\mathbf{x}),&&\mathbf{x}\in \Omega,\\
        &u(\mathbf{x}) = g(\mathbf{x}),&&\mathbf{x}\in \Gamma.
    \end{aligned}\right.
\end{equation}
To ensure well-posedness, we assume that the corresponding homogeneous Dirichlet problem, obtained by setting $f(\mathbf{x})=0$ and $g(\mathbf{x})=0$ in \eqref{Helmholtz-Dirichlet}, admits only the trivial solution $u\equiv 0$. Equivalently, zero is not a Dirichlet eigenvalue of the operator $-\Delta-\kappa^2(\mathbf{x})$. Under this non-resonance condition, problem \eqref{Helmholtz-Dirichlet} is uniquely solvable. For a constant wavenumber $\kappa$, this requirement simplifies to $\kappa^2 \notin \rho(-\Delta_D)$, where $\rho(-\Delta_D)$ denotes the spectrum of the negative Laplacian with homogeneous Dirichlet boundary conditions. On a hypercubic domain $\Omega:=(a,b)^d$, the Dirichlet eigenvalues are explicitly given by
\begin{equation*}
    \lambda_{\mathbf{n}}:=\frac{\pi^2}{(b-a)^2}\sum_{i=1}^d n_i^2, \quad \mathbf{n}=(n_1,\dots,n_d) \in \mathbb{Z}_+^d.
\end{equation*}
All wavenumbers considered in our numerical experiments satisfy this continuous non-resonance condition, thereby guaranteeing unique solvability for all benchmark problems.

We approximate the exact solution $u$ by a DNN $\mathcal{N}_D(\Theta; \cdot)$ with $D$ layers, and approximate the Laplace operator $\Delta$ using a second order central difference scheme. 
Let $m\in\mathbb{N}$ and define the set of training points
\begin{equation*}
    \mathbf{x}_{\mathbf{j}}:=(x_{\mathbf{j},i}:i\in\mathbb{N}_d),\ \ \mbox{where}\ \  \mathbf{j}:=(j_i:i\in\mathbb{N}_d)\in\mathbb{Z}_{m+2}^d\ \ \mbox{and}\ \  x_{\mathbf{j},i}:=a+j_ih\ \ \mbox{with}\ \  h=(b-a)/(m+1).
\end{equation*}
At interior points $\mathbf{x}_{\mathbf{j}} \in \Omega$, the solution is approximated by
\begin{equation*}
    u(\mathbf{x}_{\mathbf{j}})\approx\mathcal{N}_D(\Theta; \mathbf{x}_{\mathbf{j}}).
\end{equation*}
A key modification is required when the finite difference stencil involves boundary nodes. Instead of evaluating the neural network at boundary points, we directly impose the prescribed boundary data by defining
\begin{equation*}\label{boundary-condition}
    \widetilde{\mathcal{N}}_D(\Theta;\mathbf{x}):=
    \left\{\begin{aligned}
        &\mathcal{N}_D(\Theta;\mathbf{x}),&\mathbf{x}\in \Omega,\\
        &g(\mathbf{x}),&\mathbf{x} \in \Gamma.
    \end{aligned}\right.
\end{equation*}
This construction ensures that the boundary conditions are enforced exactly at the discrete level. 

Using  $\widetilde{\mathcal{N}}_D$, for the interior points, the second-order central difference approximation of the second derivative at an interior point $\mathbf{x}_{\mathbf{j}}$ is given by
\begin{equation*}
    \frac{\partial^2u}{\partial x_i^2}(\mathbf{x}_{\mathbf{j}}) \approx \frac{\widetilde{\mathcal{N}}_D(\Theta;\mathbf{x}_{\mathbf{j}+\mathbf{e}_i}) - 2\widetilde{\mathcal{N}}_D(\Theta;\mathbf{x}_{\mathbf{j}}) + \widetilde{\mathcal{N}}_D(\Theta;\mathbf{x}_{\mathbf{j}-\mathbf{e}_i})}{h^2},\quad i\in \mathbb{N}_d,
\end{equation*}
where $\mathbf{e}_i$ denotes the $i$-th coordinate unit vector in $\mathbb{R}^d$. 

For each interior training point $\mathbf{x}_{\mathbf{j}}\in\Omega$, we define the local residual by
\begin{equation*}
    \mathcal{E}(\Theta;\mathbf{x}_\mathbf{j}):=f(\mathbf{x}_{\mathbf{j}}) - \left[ \sum_{i\in \mathbb{N}_d} \left(\frac{\widetilde{\mathcal{N}}_D(\Theta;\mathbf{x}_{\mathbf{j}+\mathbf{e}_i}) - 2\widetilde{\mathcal{N}}_D(\Theta;\mathbf{x}_{\mathbf{j}}) + \widetilde{\mathcal{N}}_D(\Theta;\mathbf{x}_{\mathbf{j}-\mathbf{e}_i})}{h^2} \right)+\kappa^2(\mathbf{x}_{\mathbf{j}}) \widetilde{\mathcal{N}}_D(\Theta;\mathbf{x}_{\mathbf{j}}) \right].
\end{equation*}
The overall loss function is then defined as
\begin{equation}\label{LOSS}
    \mathcal{L}(\widetilde{\mathcal{N}}_D,\Theta):=\frac{1}{m^d} \sum_{\mathbf{j}\in\mathbb{N}_m^d} \mathcal{E}^2(\Theta;\mathbf{x}_\mathbf{j}).
\end{equation}

The finite difference-based SGDL (FD-SGDL) model described above minimizes $\mathcal{L}(\widetilde{\mathcal{N}}_D,\Theta)$ with respect to $\Theta$, 
yielding the optimal parameters $\Theta^*:=\{ \mathbf{W}_j^*,\mathbf{b}_j^*:j\in\mathbb{N}_D\}$ and the corresponding DNN approximation $\mathcal{N}_D(\Theta^*;\cdot)$. This FD-enhanced construction embeds the boundary conditions into the discrete PDE operator, effectively converting 
boundary conditions from soft penalty terms into hard constraints. As a result, it eliminates the detrimental competition between $\mathcal{L}_{PDE}$ and $\mathcal{L}_B$ in \eqref{loss-pinn}, substantially stabilizes the optimization process, and improves solution accuracy in the vicinity of $\Gamma$. Moreover, by exploiting information from neighboring grid points, the finite difference formulation has been observed to offer enhanced numerical robustness compared with AD-based residual constructions.






\section*{Finite Difference-Based Multi-Grade Deep Learning (FD-MGDL)}

We introduce the \textbf{FD-MGDL framework}, designed to enhance the accuracy and stability of numerical solutions to the Helmholtz equation \eqref{Helmholtz-Dirichlet}. The core strategy involves decomposing a DNN of total depth $D$ into $L$ successive \textbf{grades}, where each grade consists of a Shallow Neural Network (SNN). Specifically, the $l$-th grade is associated with an SNN $\mathcal{N}_{D_l}$ of depth $D_l$, satisfying $1<D_l<D$ and $\sum_{l\in\mathbb{N}_{L}}D_l=D+L-1$. For each grade $l$, we denote by $\Theta_l:=\left\{ \mathbf{W}_{l,j},\mathbf{b}_{l,j}:j\in\mathbb{N}_{D_l} \right\}$ the $l$-th grade network parameters.

\section*{Grade-1 Initialization:}

For the first grade, we employ an SNN $\mathcal{N}_{D_1}$ with depth $D_1$, the initial output $s_1$ is defined as:
\begin{equation*}
    s_1(\Theta_1;\mathbf{x}):=
    \begin{cases} 
        \mathcal{N}_{D_1}(\Theta_1;\mathbf{x}), & \mathbf{x} \in \Omega, \\ 
        g(\mathbf{x}), & \mathbf{x} \in \Gamma. 
    \end{cases}
\end{equation*}
The optimal parameters $\Theta_1^*$ are determined by minimizing the loss function $\mathcal{L}(s_1, \Theta_1)$. To facilitate the hierarchical structure, we decompose this SNN into a \textbf{feature map} $\mathbf{h}_1(\mathbf{x})$ (comprising the first $D_1-1$ layers) and a \textbf{linear output layer}:
\begin{equation*}
    \mathbf{h}_{1}(\mathbf{x}) := \mathcal{H}_{D_1-1}\left(\{\mathbf{W}_{1,j}^*, \mathbf{b}_{1,j}^*:j\in\mathbb{N}_{D_1-1}\}; \mathbf{x}\right), \quad g_1^*(\mathbf{x}) = \mathbf{W}_{1,D_1}^* \mathbf{h}_1(\mathbf{x}) + \mathbf{b}_{1,D_1}^*.
\end{equation*}

\section*{Recursive Multi-Grade Correction:}
For subsequent grades $l \ge 2$, the framework constructs correction terms recursively. Each grade leverages the feature representations learned in all preceding grades. The $l$-th grade output $g_l$ is defined via the composition of the current SNN with the frozen feature maps of previous grades:
\begin{equation*}
    g_{l}(\Theta_{l};\mathbf{x}) := \mathcal{N}_{D_{l}}(\Theta_{l};\cdot) \circ \mathbf{h}_{l-1}(\mathbf{x}).
\end{equation*}

The \textbf{cumulative approximation} at grade $l$ is the sum of the current correction and all previous optimal increments:
\begin{equation}\label{s-l}
    s_{l}(\Theta_{l};\mathbf{x}) := 
    \begin{cases} 
        g_{l}(\Theta_{l};\mathbf{x}) + \displaystyle\sum_{i=1}^{l-1} g_i^*(\mathbf{x}), & \mathbf{x} \in \Omega, \\ 
        g(\mathbf{x}), & \mathbf{x} \in \Gamma. 
    \end{cases}
\end{equation}

During the optimization of $\mathcal{L}(s_l, \Theta_l)$, only the current parameters $\Theta_l$ are trainable, while all preceding parameters $\Theta_{1}^*,\dots,\Theta_{l-1}^*$ remain \textbf{fixed}. This prevents optimization pathologies like vanishing gradients by focusing the network's capacity on learning the residual error. After training, the $l$-th grade feature representation and solution increment are updated as:
\begin{equation*}
    \mathbf{h}_{l}(\mathbf{x}) := \mathcal{H}_{D_{l}-1}\left(\{\mathbf{W}_{l,j}^*, \mathbf{b}_{l,j}^*:j\in\mathbb{N}_{D_l-1}\}; \mathbf{h}_{l-1}(\mathbf{x})\right), \quad g_l^*(\mathbf{x}) = \mathbf{W}_{l,D_l}^* \mathbf{h}_l(\mathbf{x}) + \mathbf{b}_{l,D_l}^*.
\end{equation*}

\section*{Final Approximation:}
Upon completing $L$ grades, the final FD-MGDL numerical approximation is the aggregate of all learned solution increments:
\begin{equation*}
    s_{L}^*(\mathbf{x}) = \sum_{i=1}^{L} g_i^*(\mathbf{x}), \quad \mathbf{x} \in \Omega.
\end{equation*}

A key theoretical property of the proposed FD-MGDL framework is that the optimal loss sequence $\mathcal{L}(s_{l},\Theta_{l}^*)$, $l\in\mathbb{N}_L$, is non-increasing. This result extends previous findings in function approximation \cite{xu2025multi} and integral equation modeling \cite{jiang2025adaptive} to the domain of differential equations. 

For each $l\in\mathbb{N}_{L}$ and each $j\in\mathbb{N}_{D_l}$, let $d_{l,j}$ denote the width of the $j$-th layer of the neural network at grade $l$. To formalize this property, we partition the parameters $\Theta_{l}$ of each grade $l \in \mathbb{N}_L$ into two distinct subspaces:

\begin{itemize}
    \item Feature Parameters ($\mathcal{M}_{l,1}$): Contains the weights and biases for the first $D_l-1$ layers:
    \begin{equation*}
        \mathcal{M}_{l,1} := \left\{ \{ \mathbf{W}_{l,j}, \mathbf{b}_{l,j} \}_{j=1}^{D_l-1}:  \mathbf{W}_{l,j} \in \mathbb{R}^{d_{l,j} \times d_{l,j-1}}, \mathbf{b}_{l,j} \in \mathbb{R}^{d_{l,j}} \right\}.
    \end{equation*}
    \item Output Parameters ($\mathcal{M}_{l,2}$): Contains the parameters for the final linear layer:
    \begin{equation*}
        \mathcal{M}_{l,2} := \left\{ (\mathbf{W}, \mathbf{b}): \mathbf{W} \in \mathbb{R}^{d_{l,D_l} \times d_{l,D_l-1}}, \mathbf{b} \in \mathbb{R}^{d_{l,D_l}} \right\}.
    \end{equation*}
\end{itemize}
Consequently, any $\Theta_l$ can be uniquely represented as the pair $(\Theta_{l,1}, \Theta_{l,2})$  with
\begin{equation*}
    \Theta_{l,1}:=\left\{ \mathbf{W}_{l,j},\mathbf{b}_{l,j}:j\in\mathbb{N}_{D_{l}-1} \right\}\in \mathcal{M}_{l,1}\ \  \mbox{and}\ \  \Theta_{l,2}:=(\mathbf{W}_{l,D_{l}},\mathbf{b}_{l,D_{l}})\in\mathcal{M}_{l,2}.
\end{equation*}
With the feature function $\mathbf{h}_l$ fixed from the previous grade, we define the linear operator $\mathcal{Z}_{l}: \mathcal{M}_{l,2} \to C(\Omega)$ as:
\begin{equation*}
    \mathcal{Z}_{l}((\mathbf{W}, \mathbf{b})) := \mathcal{A}_h \left( \mathbf{W}\mathbf{h}_l(\cdot) + \mathbf{b} \right),
\end{equation*}
where $\mathcal{A}_h$ is the discrete Helmholtz operator. For any $v \in C(\Omega)$, $\mathcal{A}_h$ is defined by:
\begin{equation}\label{difference-operator}
    (\mathcal{A}_h v)(\mathbf{x}) := \sum_{i=1}^d \frac{v(\mathbf{x}+h\mathbf{e}_i) - 2v(\mathbf{x}) + v(\mathbf{x}-h\mathbf{e}_i)}{h^2} + \kappa^2(\mathbf{x})v(\mathbf{x}).
\end{equation}
Furthermore, let $\|\cdot\|_N$ denote the discrete semi-norm computed over the $m^d$ training points $\mathbf{x}_{\mathbf{j}}$:
\begin{equation}\label{semi-norm}
    \|v\|_N := \sqrt{\frac{1}{m^d} \sum_{\mathbf{j} \in \mathbb{N}_m^d} |v(\mathbf{x}_{\mathbf{j}})|^2 }, \quad v \in C(\Omega).
\end{equation}

The following theorem rigorously establishes that the FD-MGDL framework ensures progressive error reduction.

\begin{theorem}[Monotonicity of Loss]\label{non-increasing}
In the FD-MGDL framework, the optimal loss sequence satisfies:
\begin{equation*}\mathcal{L}(s_{l+1}, \Theta_{l+1}^*) \le \mathcal{L}(s_l, \Theta_l^*), \quad \mbox{for all}\ \  l \in \mathbb{N}_{L-1}.
\end{equation*}
Moreover, the loss remains strictly decreasing unless the $(l+1)$-th grade correction satisfies $\big\| \mathcal{A}_h g_{l+1}^* \big\|_N = 0$.
\end{theorem}

\begin{proof}
For each grade $l\in\mathbb{N}_{L-1}$, the loss function, originally defined in \eqref{LOSS} with $\widetilde{\mathcal{N}}_D:=s_l$ and $\Theta:=\Theta_l$, can be expressed using the discrete Helmholtz operator \eqref{difference-operator} and the discrete semi-norm \eqref{semi-norm} as
\begin{equation*}\label{L}
    \mathcal{L}(s_l,\Theta_l)=\left\| f-\mathcal{A}_hs_l(\Theta_l;\cdot) \right\|_N^2,
\end{equation*}
We define the corresponding residual function by
\begin{equation*}\label{e-l}
    e_l(\Theta_l;\mathbf{x}):=f(\mathbf{x})-\mathcal{A}_hs_l(\Theta_l;\mathbf{x}), \quad \mathbf{x}\in\Omega.
\end{equation*}
Upon obtaining the optimal parameters
$\Theta_l^*$, we denote the trained network and its residual by
\begin{equation*}\label{s-l*+e-l*}
    s_l^*(\mathbf{x}):=s_l(\Theta_l^*;\mathbf{x}),\quad e_l^*(\mathbf{x}):=e_l(\Theta_l^*;\mathbf{x})=f(\mathbf{x})-\mathcal{A}_hs_l^*(\mathbf{x}),\quad \mathbf{x}\in\Omega.
\end{equation*}
Thus, the optimal loss simplifies to $\mathcal{L}(s_l,\Theta_l^*)=\|e_l^*\|_N^2$. To establish the desired inequality $\mathcal{L}(s_{l+1},\Theta_{l+1}^*) \le \mathcal{L}(s_l,\Theta_l^*)$, it suffices to prove $\|e_{l+1}^*\|_N \le \|e_l^*\|_N$. 

From the grade-wise construction \eqref{s-l}, it follows that $s_l^*(\mathbf{x})=\sum_{i\in\mathbb{N}_{l}}g_i^*(\mathbf{x})$ for all $\mathbf{x}\in\Omega$. Consequently, at grade $l+1$ we have
\begin{equation}\label{s-l+1}
    s_{l+1}(\Theta_{l+1};\mathbf{x}) = s_l^*(\mathbf{x})+g_{l+1}(\Theta_{l+1};\mathbf{x}), \quad \mathbf{x}\in\Omega.
\end{equation}
Substituting equation \eqref{s-l+1} into the definition of the residual $e_{l+1}$ yields 
\begin{equation}\label{e-l+1}
    e_{l+1}(\Theta_{l+1};\mathbf{x})=e_l^*(\mathbf{x})-\mathcal{A}_hg_{l+1}(\Theta_{l+1};\mathbf{x}), \quad \mathbf{x}\in\Omega,
\end{equation}
with trainable parameters $\Theta_{l+1}$.
Let $\Theta_{l+1}^*:=(\Theta_{l+1,1}^*,\Theta_{l+1,2}^*)$ denote the optimal parameters obtained by minimizing $\mathcal{L}(s_{l+1},\Theta_{l+1})$. We then obtain from \eqref{e-l+1} that 
\begin{equation}\label{e_l+1*}
    e_{l+1}^*(\mathbf{x})=e_l^*(\mathbf{x})-\mathcal{A}_hg_{l+1}^*(\mathbf{x}),\quad \mathbf{x}\in\Omega.
\end{equation}
Observe that
$(\Theta_{l+1,1}^*,\Theta_{l+1,2}^*)\in \mathcal{M}_{l+1,1}\times \mathcal{M}_{l+1,2}$ is a local minimizer of 
\begin{equation*}
    F(\Theta_{l+1,1},\Theta_{l+1,2}):=\|e_{l}^*-\mathcal{A}_hg_{l+1}((\Theta_{l+1,1},\Theta_{l+1,2});\cdot)\|_N^2.
\end{equation*}
Consequently, with $\Theta_{l+1,1}^*$ fixed, $\Theta_{l+1,2}^*\in\mathcal{M}_{l+1,2}$ is a local minimizer of  $F\left(\Theta_{l+1,1}^*,\Theta_{l+1,2}\right)$.
By the definition of the linear operator $\mathcal{Z}_{l+1}$, the minimization of $F(\Theta_{l+1,1}^*,\Theta_{l+1,2})$ over $\Theta_{l+1,2}\in\mathcal{M}_{l+1,2}$ is equivalent to solving the convex optimization problem
\begin{equation*}\label{convex-optimization}
\min_{\Theta_{l+1,2}\in\mathcal{M}_{l+1,2}}\|e_{l}^*-\mathcal{Z}_{l+1}(\Theta_{l+1,2})\|_N^2.
\end{equation*}
Thus, $\Theta_{l+1,2}^*\in\mathcal{M}_{l+1,2}$, being a local minimizer of this convex problem, is also a global minimizer. This implies that $\mathcal{Z}_{l+1}(\Theta_{l+1,2}^*)=\mathcal{A}_h g_{l+1}^*$ is the best approximation
to $e_l^*$ from the linear space $\mathcal{Z}_{l+1}(\mathcal{M}_{l+1,2})$ with respect to the semi-norm $\|\cdot\|_N$. Clearly, the zero function belongs to $\mathcal{Z}_{l+1}(\mathcal{M}_{l+1,2})$. Therefore, by the best-approximation property and equation \eqref{e_l+1*}, we conclude that
\begin{equation*}
    \|e_{l+1}^*\|_N = \|e_l^* - \mathcal{A}_hg_{l+1}^*\|_N \le \|e_l^* - 0\|_N \le \|e_l^*\|_N,
\end{equation*}
which proves the first assertion.

We next prove the second assertion regarding the equality condition. Clearly, the equality 
$$
\mathcal{L}(s_{l+1},\Theta_{l+1}^*)= \mathcal{L}(s_l,\Theta_l^*)
$$ 
holds if and only if  
\begin{equation}\label{equality-equivalent}
    \|e_l^* - \mathcal{A}_hg_{l+1}^*\|_N = \|e_l^* - 0\|_N.
\end{equation}
Recall that $\mathcal{A}_h g_{l+1}^*$ is the best approximation
to $e_l^*$ from the linear space $\mathcal{Z}_{l+1}(\mathcal{M}_{l+1,2})$ with respect to the semi-norm $\|\cdot\|_N$.
Thus, equality \eqref{equality-equivalent} holds if and only if  the zero function is a best approximation to $e_l^*$ from that subspace. It remainsto show that the latter is equivalent to $\|\mathcal{A}_hg_{l+1}^*\|_N=0$. By similar arguments used
in the proof of Lemma 7 in \cite{jiang2025adaptive}, if the zero function is a best approximation to $e_l^*$, then there holds that $\|\mathcal{A}_hg_{l+1}^*-0\|_N=0$, that is $\|\mathcal{A}_hg_{l+1}^*\|_N=0$. Conversely, if  $\|\mathcal{A}_hg_{l+1}^*\|_N=0$, then by the triangle inequality of the semi-norm, we obtain that  
\begin{equation*}
    \|e_{l}^*-0\|_N \leq \|e_l^* - \mathcal{A}_hg_{l+1}^*\|_N+\|\mathcal{A}_hg_{l+1}^*\|_N=\|e_l^* - \mathcal{A}_hg_{l+1}^*\|_N,
\end{equation*}
which yields that the zero function is a best approximation to $e_l^*$.
\end{proof}

The parameters $\Theta_l^*$ in Theorem~\ref{non-increasing} denote exact theoretical grade-wise minimizers satisfying the optimization assumptions of the proof. In numerical implementations, however, iterative optimizers such as Adam are terminated after a finite epoch budget. Consequently, Theorem~\ref{non-increasing} does not strictly guarantee a monotonic decrease in the empirically observed loss values. Nevertheless, practical loss monotonicity is maintained whenever the grade-$(l+1)$ residual subnetwork performs no worse than the trivial zero-correction state. Serving as the theoretical foundation for the adaptive algorithm presented in the next section, Theorem~\ref{non-increasing} is verified numerically in Section~\ref{2d} and illustrated in Figure~\ref{Optimal-loss}.

\section{Adaptive FD-MGDL Algorithm}
\label{Adaptive algorithm}

In this section, we propose an adaptive FD-MGDL algorithm for the numerical solution of the Helmholtz equation. This framework incrementally introduces training grades, utilizing an error-driven stopping criterion to automatically determine the optimal network depth. By leveraging the inter-grade error behavior inherent to MGDL, the algorithm adaptively balances approximation accuracy and computational efficiency without requiring the number of grades to be prescribed in advance.

\medskip
\noindent{\bf (1) Error-Driven Adaptation and Stopping Criterion}

As established in Theorem \ref{non-increasing}, the MGDL framework produces a computable loss sequence that is monotonically non-increasing. This property guarantees that additional grades will not deteriorate the approximation quality, providing a rigorous foundation for adaptive refinement.

In the FD-MGDL setting, the network is trained progressively: each new grade consists of a shallow network composed with the existing, frozen model. Only the parameters of the newly added grade are trained to approximate the current residual. We incorporate an automatic stopping criterion by monitoring the absolute difference between the training losses of two consecutive grades. If the difference falls below a prescribed tolerance $\epsilon$, the process terminates; otherwise, a new grade is introduced.

This mechanism offers a distinct advantage over conventional DNN training: instead of redesigning and retraining an entire deep architecture when performance is unsatisfactory, FD-MGDL simply appends a low-complexity correction layer.

\medskip
\noindent{\bf (2) Hybrid Architectural Design}

The architecture of the shallow sub-network at each grade is tailored to the physics of wave propagation and residual decay. To avoid the vanishing gradients, spectral bias, and optimization pathologies of deep end-to-end networks, we employ a physics-informed hybrid strategy:

\paragraph{\bf Grade-1: Capturing Coherent Global Waves.}
For the initial grade ($l=1$), we employ a shallow architecture comprising two hidden layers with sinusoidal activation functions. This choice is directly motivated by the physical nature of the Helmholtz equation, whose primary solution is dominated by coherent, highly oscillatory wavefield components. Standard piecewise-linear architectures (e.g., ReLU networks) suffer from well-documented spectral bias, preferentially learning low-frequency components while requiring extreme network depth to resolve high-frequency modes. In contrast, sinusoidal activations establish an explicit, Fourier-like representation mechanism where each periodic neuron encodes a distinct wave mode, enabling efficient multi-scale spatial superposition in the spirit of SIREN \cite{sitzmann2020implicit}.

A second critical advantage lies in differential operator evaluation. Because the Helmholtz operator requires second-order derivatives, $C^\infty$-smooth analytic functions are essential for capturing continuous wave derivatives without numerical artifacting. Sinusoidal activations yield smooth higher-order derivatives that inherently preserve wavefield oscillations. Conversely, standard piecewise-linear activations have weak second derivatives that vanish almost everywhere and exhibit singular delta-like behavior at activation interfaces, rendering them poorly suited for direct continuous PDE evaluation at the base grade.

\paragraph{\bf Subsequent Grades: Localized Residual Refinement via Composite Activations.}
For all subsequent grades ($l \ge 2$), we transition to shallow sub-networks consisting of a single hidden
layer with ReLU activations. This architectural shift is justified by four key theoretical and practical factors:
\begin{itemize}
    \item \textbf{Target Characteristics and Localized Wavelets:} After Grade-1 isolates the dominant wavefield, the learning target shifts to residual error fields. These residuals are less structured, less periodic, and more spatially localized \cite{fang2024addressing}. While pure global sines risk introducing non-local ringing artifacts when fitting localized errors, $\text{ReLU}(\sin(\cdot))$ acts as a windowed, wavelet-like function. It retains the local oscillatory capacity of Sine while utilizing the compact support of ReLU to restrict updates strictly to localized residual regions, establishing a complementary strategy: ``sine for global waves, windowed $ \text{ReLU}(\sin(\cdot))$ for localized residuals.''

    \item \textbf{Optimization Stability and Convex Reformulation:} Two-layer sub-networks enjoy proven convergence guarantees under mild conditions \cite{Wu_Feng_Li_Xu_2005} and admit exact convex reformulations \cite{Ergen2021convex,fang2025computational,pilanci2020neural}. These theoretical properties ensure that each sequential refinement step reduces to a well-conditioned optimization problem, bypassing the severe stiffness and non-convex landscape pathologies of deep architectures.

    \item \textbf{Asymptotic Consistency:} Within the MGDL framework, deploying shallow sub-networks at each grade mathematically guarantees vanishing approximation error in the limit \cite{ZhangShenXu}. This theoretical property confirms the asymptotic consistency of the multi-grade expansion as the number of grades increases.

    \item \textbf{Discrete Compatibility with Finite Difference Operators:} Although the inner Sine activation maintains smooth local wave behavior, the outer ReLU introduces derivative transition boundaries. Evaluating the residual via the discrete finite-difference matrix $A_h$ rather than continuous automatic differentiation avoids derivative singularities at activation interfaces, guaranteeing non-vanishing, stable discrete Laplacian approximations across all grid nodes.
\end{itemize}

In light of the preceding discussion, we propose an adaptive MGDL algorithm for the numerical solution of the Helmholtz Equation.
In passing, we comment that the MGDL approach for solving operator equations was originally introduced in~\cite{xu2025multi}. Subsequently, an adaptive MGDL framework was developed in~\cite{jiang2025adaptive} to solve highly oscillatory second-kind Fredholm integral equations.

\begin{algorithm}[H]
\caption{Adaptive FD-MGDL for the Helmholtz Equation}
\label{adaptive FD-MGDL algorithm}
\begin{algorithmic}[1]
\Require Helmholtz equation \eqref{Helmholtz-Dirichlet}, training points $\{\mathbf{x}_{\mathbf{j}}\}$, and error tolerance $\epsilon > 0$.
\Statex {\centering\bfseries Phase I: Grade-1 Initialization \par}
\State Initialize SNN $s_1(\Theta_1;\mathbf{x})$ with parameters $\Theta_1$.\State Solve $\Theta_1^* = \arg\min_{\Theta_1} \mathcal{L}(s_1, \Theta_1)$ and set $\mathcal{L}_1^* = \mathcal{L}(s_1, \Theta_1^*)$.
\State Set $l = 1$ and $\mathcal{L}_0^* = \infty$ (or a sufficiently large value).
\Statex {\centering\bfseries Phase II: Adaptive Residual Refinement \par}
\While {$|\mathcal{L}_l^* - \mathcal{L}_{l-1}^*| > \epsilon$}
    \State $l \gets l + 1$.
    \State \textbf{Freeze} optimal parameters $\{\Theta_1^*, \dots, \Theta_{l-1}^*\}$.
    \State Initialize $l$-th grade SNN  $s_l(\Theta_l; \mathbf{x})$ with parameters $\Theta_l$.
    \State Construct cumulative model $s_l(\Theta_l; \mathbf{x})$ per \eqref{s-l}.
    \State Solve $\Theta_l^* = \arg\min_{\Theta_l} \mathcal{L}(s_l, \Theta_l)$ and set $\mathcal{L}_l^* = \mathcal{L}(s_l, \Theta_l^*)$.
\EndWhile
\State \textbf{Return} final numerical solution $s_L^*(\mathbf{x}) = \sum_{i=1}^l g_i^*(\mathbf{x})$.
\end{algorithmic}
\end{algorithm}

The stopping criterion relies on the incremental loss decay $\Delta \mathcal{L}_l = \mathcal{L}_{l-1}^* - \mathcal{L}_l^*$. By Theorem \ref{non-increasing}, the sequence of theoretical loss values $\{\mathcal{L}_l^*\}_{l=1}^L$ is monotonically non-increasing. In practice, training terminates once $\Delta \mathcal{L}_l \le \epsilon$, signaling that further grade additions yield diminishing marginal gains in residual correction. 

During the training of grade $l$, the optimized parameters of all preceding grades, $\{\Theta_1^{*},\ldots,\Theta_{l-1}^{*}\}$, are frozen, and only the parameters $\Theta_l$ are updated. Consequently, gradient computation and parameter updates are restricted to the current grade, maintaining a constant active optimization footprint ($\mathcal{O}(1)$ relative to total network depth) and completely eliminating inter-grade parameter coupling.

Below we reformulate the nonconvex optimization problem arising in the FD-MGDL framework with two-layer ReLU networks at each grade as a sequence of convex optimization subproblems, following the approach of \cite{fang2025computational}.

For a grade $l\ge 2$, we model the correction term using a two-layer ReLU network $\mathcal{N}_{2_l}:\mathbb{R}^{d_{l-1,D_{l-1}-1}}\to\mathbb{R}$ with $m_l$ hidden neurons, defined as:
\begin{equation*}    
    \mathcal{N}_{2_{l}}(\mathbf{t}_l)=\sum_{j=1}^{m_l}\alpha_{lj}\left(\mathbf{t}_l^{\top}\mathbf{w}_{lj}\right)_+,
\end{equation*}
where $\mathbf{t}_l\in\mathbb{R}^{d_{l-1,D_{l-1}-1}}$ represents the input feature vector for grade $l$, while $\mathbf{w}_{lj}\in \mathbb{R}^{d_{l-1,D_{l-1}-1}}$ and $\alpha_{lj}\in\mathbb{R}$ denote the hidden-layer weights and output layer coefficients, respectively. 

{\bf Data and Target Formulation:}
To organize the training points $\mathbf{x}_{\mathbf{j}}$ ($\mathbf{j} \in \mathbb{N}_m^d$), we fix an ordering via a bijection $\phi: \mathbb{N}_{m^d} \to \mathbb{N}_m^d$. We define the feature data matrix $\mathbf{X}_l \in \mathbb{R}^{m^d \times d_{l-1,D_{l-1}-1}}$ by collecting the feature representations from the preceding grade:
\begin{equation*}
\mathbf{X}_l := \left[ \mathbf{h}_{l-1}(\mathbf{x}_{\phi(1)}), \mathbf{h}_{l-1}(\mathbf{x}_{\phi(2)}), \ldots, \mathbf{h}_{l-1}(\mathbf{x}_{\phi(m^d)}) \right]^{\top}.
\end{equation*}
The training target for grade $l$ is the optimal residual vector $\mathbf{e}_{l-1}^*$ inherited from grade $l-1$:
\begin{equation*}
\mathbf{e}_{l-1}^* := \left[ e_{l-1}^*(\mathbf{x}_{\phi(1)}), e_{l-1}^*(\mathbf{x}_{\phi(2)}), \ldots, e_{l-1}^*(\mathbf{x}_{\phi(m^d)}) \right]^{\top} \in \mathbb{R}^{m^d}.
\end{equation*}

{\bf The Optimization Problem:}
Training the network at grade $l$ thus involves solving the following non-convex optimization problem:
\begin{equation}\label{nonconvex}
    \min_{\{\mathbf{w}_{lj},\alpha_{lj}\}_{j=1}^{m_l}} \left\| \mathbf{e}_{l-1}^*-\mathcal{A}_h\sum_{j=1}^{m_l}\left(\mathbf{X}_l\mathbf{w}_{lj}\right)_+\alpha_{lj}  \right\|_2^2,
\end{equation}
where the discrete Helmholtz operator $\mathcal{A}_h$ is applied component-wise to the resulting vector. This structure reveals that while the objective is non-convex in the weights $\mathbf{w}_{lj}$, it maintains a specific multilinear form that facilitates the convex reformulation discussed in the subsequence.

We now define the equivalent convex optimization problem. For any weight vector $\mathbf{w}\in \mathbb{R}^{d_{l-1,D_{l-1}-1}}$,  let $\mathrm{Diag}(1[\mathbf{X}_l\mathbf{w}\ge0])$ be a diagonal matrix, where $1[\cdot]$ is the element‑wise indicator function:
\begin{equation*}
    1[\mathbf{X}_l \mathbf{w} \ge \mathbf{0}] = \Bigl( 1\bigl[\mathbf{h}_{l-1}(\mathbf{x}_{\phi(1)})^\top \mathbf{w} \ge 0\bigr], \dots,1\bigl[\mathbf{h}_{l-1}(\mathbf{x}_{\phi(m^d)})^\top \mathbf{w} \ge 0\bigr] \Bigr)^\top \in \{0,1\}^{m^d}.
\end{equation*}
Let $\mathbf{D}_{l1},\dots,\mathbf{D}_{lP_l}$ be the set of all $P_l$ distinct diagonal matrices generated as $\mathbf{w}$ varies over $\mathbb{R}^{d_{l-1,D_{l-1}-1}}$. Each matrix $\mathbf{D}_{li}$ corresponds to a unique {\it convex polyhedral cone} $K_{li}$, defined as:
\begin{equation}\label{K_li}
    K_{li}:=\left\{\mathbf{w} \in \mathbb{R}^{d_{l-1,D_{l-1}-1}} :(\mathbf{X}_l\mathbf{w})_k\geq 0, \mbox{if } (\mathbf{D}_{li})_{kk}=1;(\mathbf{X}_l\mathbf{w})_k< 0, \mbox{if }(\mathbf{D}_{li})_{kk}=0,k\in\mathbb{N}_{m^d}\right\}.
\end{equation}
Geometrically,  the collection of hyperplanes $\mathbf{h}_{l-1}(\mathbf{x}_{\phi(k)})^{\top}\mathbf{w} = 0$ partition the weight space
$\mathbb{R}^{d_{l-1,D_{l-1}-1}}$ into these $P_l$ cones. For a given feasible solution $\left\{ \mathbf{w}_{lj},\alpha_{lj} \right\}_{j=1}^{m_l}$ of problem \eqref{nonconvex}, we can group the neurons according to which cone their weight vector lies in. For each $i\in\mathbb{N}_{P_l}$, we set \begin{equation*}
   S_{li}:=\left\{ j\in\mathbb{N}_{m_l}: \mathbf{w}_{lj}\in K_{li} \right\}.
\end{equation*}
The sets $S_{li}$, $i\in\mathbb{N}_{P_l}$ form a partition of the index set $\mathbb{N}_{m_l}$, that is,
\begin{equation*}
    \bigcup_{i=1}^{P_l} S_{li} = \mathbb{N}_{m_l} \quad\text{and}\quad 
S_{li} \cap S_{li'} = \emptyset \ \  \text{for}\  i \neq i'.
\end{equation*}
Therefore, the network output can be rewritten as
\begin{equation}\label{relu-sum}
\mathcal{A}_h\sum_{j=1}^{m_l}\left(\mathbf{X}_l\mathbf{w}_{lj}\right)_+\alpha_{lj} =\mathcal{A}_h\sum_{i=1}^{P_l}\sum_{j\in S_{li}}\left(\mathbf{X}_l\mathbf{w}_{lj}\right)_+\alpha_{lj}.
\end{equation}
By definition \eqref{K_li} of $K_{li}$, we have for each $j\in S_{li}$ and each $k\in\mathbb{N}_{m^d}$ that if $(\mathbf{D}_{li})_{kk}=1$, then $(\mathbf{X}_l\mathbf{w}_{lj})_k\geq 0$, 
and if $(\mathbf{D}_{li})_{kk}=0$, then $(\mathbf{X}_l\mathbf{w}_{lj})_k< 0$. These inequalities immediately yield that 
\begin{equation}\label{relu}
    (\mathbf{X}_l\mathbf{w}_{lj})_+=\mathbf{D}_{li}\mathbf{X}_l\mathbf{w}_{lj}, \ \ j\in S_{li}.
\end{equation}
Substituting equation \eqref{relu} into the right hand side of equation \eqref{relu-sum} leads to
\begin{equation}\label{relu-sum1}
\mathcal{A}_h\sum_{j=1}^{m_l}\left(\mathbf{X}_l\mathbf{w}_{lj}\right)_+\alpha_{lj} =\mathcal{A}_h\sum_{i=1}^{P_l}\mathbf{D}_{li}\mathbf{X}_l\left(\sum_{j\in S_{li}}\alpha_{lj}\mathbf{w}_{lj}\right).
\end{equation}
For each $i\in\mathbb{N}_{P_l}$, we introduce auxiliary vectors that aggregate the positive and negative contributions within each cone:
\begin{equation*}
    \mathbf{v}_{li} := \sum_{j\in S_{li},\,\alpha_{lj}\ge 0}\alpha_{lj}\mathbf{w}_{lj}, \quad
    \mathbf{u}_{li} := -\sum_{j\in S_{li},\,\alpha_{lj}< 0}\alpha_{lj}\mathbf{w}_{lj}.
\end{equation*}
Substituting 
\begin{equation*}
\sum_{j\in S_{li}}\alpha_{lj} \mathbf{w}_{lj}= \mathbf{v}_{li} - \mathbf{u}_{li}
\end{equation*}
into equation \eqref{relu-sum1} yields that 
\begin{equation*}
\mathcal{A}_h\sum_{j=1}^{m_l}\left(\mathbf{X}_l\mathbf{w}_{lj}\right)_+\alpha_{lj}=\mathcal{A}_h\sum_{i=1}^{P_l}\mathbf{D}_{li}\mathbf{X}_l(\mathbf{v}_{li}-\mathbf{u}_{li}).
\end{equation*} 
Since $\mathbf{v}_{li}$, $\mathbf{u}_{li}$ are non-negative linear combinations of weight vectors within $K_{li}$, they inherit the cone membership property. Specifically, they satisfy the linear constraints 
\begin{equation*}
    (2\mathbf{D}_{li}-\mathbf{I})\mathbf{X}_l\mathbf{v}_{li}\geq 0, \quad (2\mathbf{D}_{li}-\mathbf{I})\mathbf{X}_l\mathbf{u}_{li}\geq 0.
\end{equation*}
This transformation shows that every feasible solution of the non-convex problem \eqref{nonconvex} can be mapped to a feasible solution of a convex optimization problem with exactly the same objective value. This motivates the definition of the equivalent convex optimization problem:
\begin{equation}\label{convexprogram}
    \begin{aligned}
        &\min_{\{ \mathbf{v}_{li},\mathbf{u}_{li} \}_{i=1}^{P_l}} \left\| \mathbf{e}_{l-1}^*-\mathcal{A}_h\sum_{i=1}^{P_l}\mathbf{D}_{li}\mathbf{X}_l(\mathbf{v}_{li}-\mathbf{u}_{li})\right\|_2^2\\
        &\mathrm{subject\ to}\  (2\mathbf{D}_{li}-\mathbf{I})\mathbf{X}_l\mathbf{v}_{li}\geq0,\ (2\mathbf{D}_{li}-\mathbf{I})\mathbf{X}_l\mathbf{u}_{li}\geq0,\  i\in\mathbb{N}_{P_l}.
    \end{aligned}
\end{equation}
The linear constraints in \eqref{convexprogram} ensure that the vectors $\mathbf{v}_{li}$ and $\mathbf{u}_{li}$ remain within their respective cones $K_{li}$, effectively transforming the original non-convex search for weights into a structured search across a fixed arrangement of linear regions. Since every non-convex solution has a corresponding convex counterpart with the same objective, taking the minimum over the convex set provides a global lower bound.

The following theorem characterizes the equivalence between the non-convex optimization problem \eqref{nonconvex} and the convex reformulation \eqref{convexprogram}. Specifically, it establishes that the convex problem serves as a global lower bound for the non-convex objective. Furthermore, when the hidden layer width $m_l$ is sufficiently large, the duality gap vanishes—the optimal values of both problems coincide, and the global minimizers of the neural network can be recovered explicitly from the solution of the convex program.

\begin{theorem}[Convex–Nonconvex Equivalence]\label{nonconvex-convex}
Let $P_{\mathrm{nc}}^{*}$ and $P_{\mathrm{c}}^{*}$ denote the optimal values of the non-convex training problem \eqref{nonconvex} and the convex reformulation \eqref{convexprogram}, respectively. The following properties hold:

\begin{enumerate}
    \item \textbf{Lower Bound:} For any network width $m_l \ge 1$, the convex problem provides a global lower bound such that $P_{\mathrm{nc}}^{*} \ge P_{\mathrm{c}}^{*}$.
    
    \item \textbf{Global Optimality and Construction:} Let $\{(\mathbf{v}_{li}^*,\mathbf{u}_{li}^*)\}_{i=1}^{P_l}$ be an optimal solution to the convex program \eqref{convexprogram}. 
    If the network width satisfies 
    \begin{equation*}
        m_l \ge m_l^*:=\sum_{i=1}^{P_l}(1[\mathbf{v}_{li}^*\ne0]+1[\mathbf{u}_{li}^*\ne0]),
    \end{equation*}
    then the duality gap vanishes ($P_{\mathrm{nc}}^{*} = P_{\mathrm{c}}^{*}$). 
    
    Furthermore, a globally optimal solution $\{\mathbf{w}_{lj}^*, \alpha_{lj}^*\}_{j=1}^{m_l}$ for the non-convex problem \eqref{nonconvex} is explicitly constructed as follows:
    \begin{itemize}
        \item For each $i \in \{1, \dots, P_l\}$ where $\mathbf{v}_{li}^* \ne \mathbf{0}$:
        \begin{equation*}
        \left( \mathbf{w}_{lj_{1,i}}^*, \alpha_{lj_{1,i}}^* \right) = \left( \frac{\mathbf{v}_{li}^*}{\|\mathbf{v}_{li}^*\|}, \|\mathbf{v}_{li}^*\| \right).
        \end{equation*}
        \item For each $i \in \{1, \dots, P_l\}$ where $\mathbf{u}_{li}^* \ne \mathbf{0}$:
        \begin{equation*}
        \left( \mathbf{w}_{lj_{2,i}}^*, \alpha_{lj_{2,i}}^* \right) = \left( \frac{\mathbf{u}_{li}^*}{\|\mathbf{u}_{li}^*\|}, -\|\mathbf{u}_{li}^*\| \right).
        \end{equation*}
    \end{itemize}
    Here, $\{j_{1,i}\}$ and $\{j_{2,i}\}$ are distinct neuron indices. The remaining $m_l - m_l^*$ neurons are set to $(\mathbf{0}, 0)$.
\end{enumerate}
\end{theorem}

\begin{proof}
We first prove statement (1). Let $\left\{ \mathbf{w}_{lj},\alpha_{lj} \right\}_{j=1}^{m_l}$ be any feasible solution to the non-convex \eqref{nonconvex}. As shown previously, we can associate with this solution a corresponding feasible solution $\{\mathbf{v}_{li},\mathbf{u}_{li}\}_{i=1}^{P_l}$ to the convex problem \eqref{convexprogram} that satisfies
\begin{equation*}
    \left\| \mathbf{e}_{l-1}^*-\mathcal{A}_h\sum_{j=1}^{m_l}\left(\mathbf{X}_l\mathbf{w}_{lj}\right)_+\alpha_{lj}  \right\|_2^2=\left\| \mathbf{e}_{l-1}^*-\mathcal{A}_h\sum_{i=1}^{P_l}\mathbf{D}_{li}\mathbf{X}_l(\mathbf{v}_{li}-\mathbf{u}_{li})\right\|_2^2.
\end{equation*}
That is, the feasible solution $\{\mathbf{v}_{li},\mathbf{u}_{li}\}_{i=1}^{P_l}$ of \eqref{convexprogram} achieves exactly the same objective value as the original non-convex solution $\left\{ \mathbf{w}_{lj},\alpha_{lj} \right\}_{j=1}^{m_l}$ of \eqref{nonconvex}. Since every feasible solution of \eqref{nonconvex} can be mapped to a feasible solution of \eqref{convexprogram} with the same objective value, taking minima over both problems yields that 
$P_{\mathrm{nc}}^*\ge P_{\mathrm{c}}^*.$

We next verify statement (2). Suppose that $\{\mathbf{v}_{li}^*,\mathbf{u}_{li}^*\}_{i=1}^{P_l}$ is an optimal solution of problem 
\eqref{convexprogram} and define
\begin{equation*}
    m_l^*:=\sum_{i=1}^{P_l}(1[\mathbf{v}_{li}^*\ne0]+1[\mathbf{u}_{li}^*\ne0]).
\end{equation*}
Using this  optimal solution, we construct a feasible solution for the non‑convex problem \eqref{nonconvex} as follows. For each $i\in\mathbb{N}_{P_l}$, 
\begin{itemize}
    \item if $\mathbf{v}_{li}^*\neq 0$, we introduce a neuron with parameters 
    \begin{equation*}
      \mathbf{w}_{lj_{1,i}}^*=\frac{\mathbf{v}_{li}^*}{\|\mathbf{v}_{li}^*\|},\ \ \alpha_{lj_{1,i}}^*=\|\mathbf{v}_{li}^*\|;
    \end{equation*}
    \item if $\mathbf{u}_{li}^*\neq 0$, we introduce a neuron with parameters
     \begin{equation*}
      \mathbf{w}_{lj_{2,i}}^*=\frac{\mathbf{u}_{li}^*}{\|\mathbf{u}_{li}^*\|},\ \ \alpha_{lj_{2,i}}^*=-\|\mathbf{u}_{li}^*\|.
    \end{equation*}
\end{itemize}
The indices $j_{1,i}$ and $j_{2,i}$ are taken distinct for different pairs, and the remaining $m_l-m_l^*$ neurons are set to $(\mathbf{0}, 0)$. 
Because $\{\mathbf{v}_{li}^*,\mathbf{u}_{li}^*\}_{i=1}^{P_l}$ satisfies the constraint condition of \eqref{convexprogram}, we have for each $i\in\mathbb{N}_{P_l}$ that
\begin{equation}\label{relu-i}
    (\mathbf{X}_l\mathbf{w}_{lj_{1,i}}^*)_+=\mathbf{D}_{li}\mathbf{X}_l\mathbf{w}_{lj_{1,i}}^*,\ \ (\mathbf{X}_l\mathbf{w}_{lj_{2,i}}^*)_+=\mathbf{D}_{li}\mathbf{X}_l\mathbf{w}_{lj_{2,i}}^*.
\end{equation} 
Denote the constructed parameters by $\{(\mathbf{w}^*_{lj},\alpha^*_{lj})\}_{j=1}^{m_l}$. Then the output of the two‑layer network equals
\begin{equation*}
    \sum_{j=1}^{m_l}\left(\mathbf{X}_l\mathbf{w}_{lj}^*\right)_+\alpha_{lj}^*=\sum_{i=1}^{P_l}(\mathbf{X}_l\mathbf{w}_{lj_{1,i}}^*)_+\alpha_{lj_{1,i}}^*+(\mathbf{X}_l\mathbf{w}_{lj_{2,i}}^*)_+\alpha_{lj_{2,i}}^*.
\end{equation*} 
Substituting \eqref{relu-i} into the above equation yields that 
\begin{equation*}
    \sum_{j=1}^{m_l}\left(\mathbf{X}_l\mathbf{w}_{lj}^*\right)_+\alpha_{lj}^*=\sum_{i=1}^{P_l}\left(\mathbf{D}_{li}\mathbf{X}_l\mathbf{w}_{lj_{1,i}}^*\alpha_{lj_{1,i}}^*+\mathbf{D}_{li}\mathbf{X}_l\mathbf{w}_{lj_{2,i}}^*\alpha_{lj_{2,i}}^*\right).
\end{equation*} 
Using the definitions of $\mathbf{w}_{lj_{1,i}}^*,\alpha_{lj_{1,i}}^*$ and 
$\mathbf{w}_{lj_{2,i}}^*,\alpha_{lj_{2,i}}^*$, we obtain that
\begin{equation*}
    \sum_{j=1}^{m_l}\left(\mathbf{X}_l\mathbf{w}_{lj}^*\right)_+\alpha_{lj}^*=\sum_{i=1}^{P_l}\mathbf{D}_{li}\mathbf{X}_l(\mathbf{v}_{li}^*-\mathbf{u}_{li}^*).
\end{equation*}  
Therefore, $\{(\mathbf{w}^*_{lj},\alpha^*_{lj})\}_{j=1}^{m_l}$ is a feasible solution to \eqref{nonconvex}, whose objective value equals the optimal value $P_\mathrm{c}^*$. Consequently,  $P_{\mathrm{nc}}^*\le P_{\mathrm{c}}^*$. Combined with $P_{\mathrm{nc}}^* \ge P_{\mathrm{c}}^*$ from statement (1), we get $P_{\mathrm{nc}}^* = P_{\mathrm{c}}^*$.
\end{proof}

The established convex--nonconvex equivalence ensures that, with sufficient network width, the non-convex training of a two-layer ReLU network shares the same optimal value as a structured convex program. Within the FD-MGDL framework, this insight serves two primary roles:

\begin{itemize}
    \item \textbf{Algorithmic Flexibility:} It provides an alternative path to global optimality via convex solvers, bypassing the initialization sensitivity of non-convex optimization.
    \item \textbf{Structural Certification:} More importantly, it certifies the \textbf{structural benignity} of the non-convex landscape. As established in \cite{pilanci2020neural}, the global optima of the convex program correspond to the stationary points of the non-convex objective. This suggests that backpropagation acts as an effective heuristic for the underlying convex structure, with Stochastic Gradient Descent (SGD) consistently converging to the convex optimum regardless of initialization.
\end{itemize}

Consequently, this equivalence justifies using standard non-convex training at each grade, ensuring that the refinement of the Helmholtz solution remains a principled and stable task.

\section{Implementation Details and Structural Ablation Study}\label{details}

This section details the implementation of Algorithm \ref{adaptive FD-MGDL algorithm}, provides its justification, and presents a structural ablation study investigating the influence of network depth on the multi-grade learning process.

\subsection{Implementation Details}

{\bf Optimization and Initialization.} For each grade $l$, the optimization problem is solved using the Adam optimizer with Xavier initialization. We apply an exponential learning rate decay schedule where the learning rate at epoch $k$ is:
\begin{equation*}t_k := t_{\max} \exp(-\gamma k), \quad \gamma := \frac{\ln(t_{\max}/t_{\min})}{K},\end{equation*}
where $K$ is the total number of epochs, and $t_{\max}$ and $t_{\min}$ are the prescribed maximum and minimum learning rates. Training terminates if $K$ is reached or if the change in training loss between iterations falls below a tolerance $\epsilon$.

\textbf{Evaluation Metrics.} To assess predictive accuracy, we employ the Relative Squared Error (RSE). For a network $\mathcal{N}$ predicting $\hat{\mathbf{y}}_{\ell} := \mathcal{N}(\mathbf{x}_{\ell})$ given inputs $\mathbf{x}_\ell$ ($\ell \in \mathbb{N}_{N}$), the RSE is defined as:
\begin{equation*}
\mathrm{RSE} := \frac{\sum_{\ell \in \mathbb{N}_{N}} \| \hat{\mathbf{y}}_{\ell} - \mathbf{y}_{\ell} \|_2^2}{\sum_{\ell \in \mathbb{N}_{N}} \| \mathbf{y}_{\ell} \|_2^2}.
\end{equation*}
We denote the errors on the training and testing sets as $\mathrm{TrRSE}$ and $\mathrm{TeRSE}$, respectively. 

In our numerical experiments, an $\mathrm{RSE} \approx 1$ indicates that the model predictions have collapsed to near-zero ($\hat{\mathbf{y}}_{\ell} \approx \mathbf{0}$ for all $\ell$). In this trivial regime, the numerator $\sum_{\ell} \| \hat{\mathbf{y}}_{\ell} - \mathbf{y}_{\ell} \|_2^2$ simplifies to the signal norm $\sum_{\ell} \| \mathbf{y}_{\ell} \|_2^2$, driving the ratio to unity. Such values reflect a failure of the model to learn any meaningful underlying wave structures, effectively outputting a null field.

Computational efficiency is measured via the Accumulated Computational Time (AC time):
\begin{equation*}
    \mathrm{AC\ time} := 
    \begin{cases}
        \sum_{j=1}^i t_j, & \text{for MGDL (up to grade } i\text{)}, \\
        t_{\mathrm{full}}, & \text{for standard baseline methods},
    \end{cases}
\end{equation*}
where $t_j$ represents the training duration for the $j$-th grade in MGDL, and $t_{\mathrm{full}}$ denotes the total training time for competing baselines.

\subsection{Theoretical Motivation: Convexity and Stability}

While the FD-MGDL framework utilizes the Adam optimizer to solve non-convex loss functions for efficiency, its performance is theoretically grounded in the convex subproblem formulation described in \eqref{convexprogram}. This bridge between non-convex optimization and convex theory is critical for two primary reasons:
\begin{itemize}
    \item {\bf Computational Scalability:} Directly solving the convex subproblems via Quadratic Programming (QP) or specialized convex solvers is computationally prohibitive for large-scale 2D and 3D Helmholtz problems. This is primarily due to the extreme memory overhead and high dimensionality inherent in the activation patterns of large-scale systems. While the direct translation of Theorem \ref{nonconvex-convex} into specialized solvers remains a subject for further investigation, the theorem currently serves to justify why gradient-based methods can effectively navigate the search space without the need for prohibitive convex solvers.
    \item{\bf Theoretical Guarantee:} Theorem \ref{nonconvex-convex} establishes that with sufficient network width, the non-convex optimization landscape possesses a benign structure—it contains no ``bad'' local minima.
\end{itemize}
Consequently, gradient-based methods like Adam act as implicit solvers for the underlying convex program. This theoretical consistency ensures stable convergence even at high wavenumbers—conditions under which traditional single-grade methods (such as FD-SGDL) typically stagnate due to optimization stiffness. By decomposing the problem into multiple shallow grades, FD-MGDL leverages this benign landscape to refine the solution progressively and reliably.

\subsection{Adaptive Nature of the FD-MGDL Framework}
The adaptive nature of FD-MGDL is governed by the tolerance threshold $\epsilon$, which triggers the transition between successive grades. The robustness of this mechanism is supported by the monotonicity of the loss function (Theorem \ref{non-increasing}).

{\bf Theoretical Robustness:} Theorem \ref{non-increasing} guarantees that the training loss $\mathcal{L}$ is non-increasing as the grade level $l$ increases. This ensures that adding a new grade serves as a refinement step that either reduces the residual or maintains the current error level. Thus, the threshold $\epsilon$ acts as a criterion for identifying diminishing returns rather than a volatile hyperparameter that might risk instability.

In practice, the optimal $\epsilon$ is coupled with the wavenumber $\kappa$ and domain complexity. We adopt a monitoring strategy based on the TrRSE: if the error plateaus, $\epsilon$ is increased to terminate the grade-switching early and save computation; if a downward trend persists, $\epsilon$ is decreased to introduce new grades and capture high-frequency features more accurately.

{\bf Learning Rate Selection and Stability:} Because parameters from previous grades are frozen, each grade constitutes an independent optimization problem. We perform a grid search over learning rates $\{10^{-1}, \dots, 10^{-4}\}$ to identify optimal $t_{\max}$ and $t_{\min}$.

Interestingly, as the grade index $l$ increases, the magnitude of the residual $\|e_{l-1}^*\|_N$ typically decreases. This indicates that subsequent grades are fine-tuning the solution on a progressively "smaller signal," which provides a protective, stabilizing effect. This built-in stability ensures the final solution remains robust even with minor variations in the learning rate.

\subsection{Ablation Study: Grade-Wise Architecture and Activation Functions}

To assess the effectiveness of the architectural choices in the proposed adaptive MGDL algorithm, we conduct two controlled ablation studies concerning the grade-wise allocation of network depth and the activation-function configuration. Specifically, we consider the following two-dimensional Helmholtz problem:
\begin{equation}\label{2D-Helmholtz-equation}
    \left\{\begin{aligned}
        &\frac{\partial^2u}{\partial x_1^2}+\frac{\partial^2u}{\partial x_2^2}+\kappa^2 u(x_1,x_2)=0, &&(x_1,x_2)\in \Omega,\\
        &u(x_1,x_2) = 0, && (x_1,x_2)\in\Gamma_1:=\{0\}\times(0,1), \\
        &u(x_1,x_2) = 0, &&(x_1,x_2)\in\Gamma_2:=(0,1)\times\{0\},\\
        &u(x_1,x_2) = \sin\left(\frac{\sqrt{2}}{2}\kappa\right)\sin\left(\frac{\sqrt{2}}{2}\kappa x_2\right),  && (x_1,x_2)\in\Gamma_3:=\{1\}\times(0,1), \\
        &u(x_1,x_2) = \sin\left(\frac{\sqrt{2}}{2}\kappa x_1\right)\sin\left(\frac{\sqrt{2}}{2}\kappa\right),  &&(x_1,x_2)\in\Gamma_4:=(0,1)\times\{1\},
    \end{aligned}\right.
\end{equation}
where $\kappa=50$. 

The ablation study consists of two parts. First, we investigate how the allocation of network depth across grades affects the accuracy, convergence behavior, and computational efficiency of MGDL. Second, we evaluate the role of the Sine--ReLU activation configuration adopted in MGDL-1 by comparing it with a Sine-only counterpart. All hyperparameters and training settings other than the factor under investigation are kept fixed. The following four configurations are considered:
\begin{enumerate}
    \item {\bf MGDL-1 (Adaptive FD-MGDL):}
    \begin{equation*}
        \begin{aligned}
            &\mathrm{Grade\ 1}:\ [2] \to [256]\times2 \to [1]\\
            &\mathrm{Grade\ j}:\ [2] \to [256]_F\times j \to [256] \to [1], \ \ j=2,3, \dots, L.\\
        \end{aligned}
    \end{equation*}
    MGDL-1 follows the Adaptive FD-MGDL algorithm described in Algorithm \ref{adaptive FD-MGDL algorithm}, where the total number of grades $L$ is determined automatically during training according to a loss-based convergence criterion. For the Helmholtz problem \eqref{2D-Helmholtz-equation}, the adaptive procedure terminates at $L=6$, when the loss reduction between successive grades falls below the prescribed tolerance.
    \item {\bf MGDL-2:}
    \begin{equation*}
        \begin{aligned}
            &\mathrm{Grade\ 1}:\ [2] \to [256]\times2 \to [1]\\
            &\mathrm{Grade\ 2}:\ [2] \to [256]_F\times2 \to [256]\times2 \to [1]\\
            &\mathrm{Grade\ 3}:\ [2] \to [256]_F\times4 \to [256]\times3 \to [1].\\
        \end{aligned}
    \end{equation*}
    \item {\bf MGDL-3:}
    \begin{equation*}
        \begin{aligned}
            &\mathrm{Grade\ 1}:\ [2] \to [256]\times3 \to [1]\\
            &\mathrm{Grade\ 2}:\ [2] \to [256]_F\times3 \to [256]\times2 \to [1]\\
            &\mathrm{Grade\ 3}:\ [2] \to [256]_F\times5 \to [256]\times2 \to [1].\\
        \end{aligned}
    \end{equation*}
    MGDL-2 and MGDL-3 are constructed as controlled variants of MGDL-1 by redistributing a comparable total number of hidden layers into fewer grades with deeper per-grade architectures.
    \item {\bf MGDL-4 (Sine-only):} this configuration has exactly the same grade-wise architecture, adaptive stopping criterion, optimization parameters, and training schedule as MGDL-1, but uses the Sine activation throughout the trainable network instead of the Sine--ReLU activation configuration.
\end{enumerate}

\begin{table}[!htb]\centering
    \small
    \setlength{\tabcolsep}{10pt}
    \begin{threeparttable}
    \caption{Ablation results for different grade-wise architectures and activation-function configurations on the two-dimensional Helmholtz problem \eqref{2D-Helmholtz-equation}. MGDL-1, MGDL-2, and MGDL-3 use the Sine--ReLU activation configuration, whereas MGDL-4 (Sine-only) uses Sine activations throughout.}
    \label{NumericalResult_for_ablation}
    \centering
    \begin{tabular}{ccrccrcc} 
        \toprule
         Structure & Grade & Epochs & $t_{\max}$ & $t_{\min}$ & AC time (s)  & TrRSE & TeRSE \\
        \midrule
        \multirow{6}{*}{MGDL-1} 
        & Grade 1 & 400 & $10^{-1}$ & $10^{-2}$ & 489 & $5.62\times 10^{-1}$ & $5.64\times 10^{-1}$ \\
        & Grade 2 & 3,000 & $10^{-2}$ & $10^{-3}$ & 2,443 &$6.98\times 10^{-4}$ &$7.07\times 10^{-4}$ \\
        & Grade 3 & 3,000 & $10^{-3}$ & $10^{-4}$ & 3,736 & $5.49\times10^{-4}$ &$5.50\times 10^{-4}$ \\
        & Grade 4 & 2,500 & $10^{-3}$ & $10^{-3}$ & 4,798 &$5.41\times 10^{-4}$ &$5.43\times 10^{-4}$ \\
        & Grade 5 & 2,500 & $10^{-3}$ & $10^{-3}$ & 5,852 & $5.37\times 10^{-4}$ &$5.38\times 10^{-4}$ \\
        & Grade 6 & 1,000 & $10^{-3}$ & $10^{-3}$ & 6,272 &$\mathbf{5.37\times 10^{-4}}$ &$\mathbf{5.37\times 10^{-4}}$ \\
        \midrule
        \multirow{3}{*}{MGDL-2} 
        & Grade 1 & 2,000 & $10^{-1}$ & $10^{-2}$ & 2,456 & $3.99\times 10^{-1}$ & $4.01\times 10^{-1}$ \\
        & Grade 2 & 2,000 & $10^{-2}$ & $10^{-2}$ & 3,953 &$6.85\times 10^{-3}$ &$5.86\times 10^{-3}$ \\
        & Grade 3 & 2,000 & $10^{-3}$ & $10^{-3}$ & {\bf 6,083} & $4.61\times10^{-3}$ & $4.77\times10^{-3}$ \\
        \midrule
        \multirow{3}{*}{MGDL-3} 
        & Grade 1 & 2,000 & $10^{-1}$ & $10^{-4}$ & 3,890 & $2.03\times 10^{-2}$ & $2.07\times 10^{-2}$ \\
        & Grade 2 & 2,000 & $10^{-3}$ & $10^{-3}$ & 5,755  & $5.56\times10^{-4}$  &$5.64\times10^{-4}$  \\
        & Grade 3 & 2,000 & $10^{-3}$ & $10^{-3}$ &  7,526 & $5.43\times10^{-4}$ & $5.43\times10^{-4}$ \\
        \midrule
        & Grade 1 & 400 & $10^{-1}$ & $10^{-2}$ & 472 & $5.62\times 10^{-1}$ & $5.64\times 10^{-1}$ \\
        & Grade 2 & 3,000 & $10^{-2}$ & $10^{-3}$ & 2,801 & $1.01\times10^{-3}$ &$1.05\times 10^{-3}$ \\
        MGDL-4 & Grade 3 & 3,000 & $10^{-3}$ & $10^{-4}$ & 4,504 & $5.84\times10^{-4}$ &$5.86\times 10^{-4}$ \\
        (Sine-only) & Grade 4 & 2,500 & $10^{-3}$ & $10^{-3}$ & 5,905 &$5.83\times 10^{-4}$ &$5.86\times 10^{-4}$ \\
        & Grade 5 & 2,500 & $10^{-3}$ & $10^{-3}$ & 7,280 & $5.84\times 10^{-4}$ &$5.86\times 10^{-4}$ \\
        & Grade 6 & 1,000 & $10^{-3}$ & $10^{-3}$ &  7,830 &$5.84\times 10^{-4}$ &$5.86\times 10^{-4}$ \\
        \bottomrule
    \end{tabular}
    \end{threeparttable}
\end{table}

\begin{figure}[!htb]
\centering
\includegraphics[scale=0.83]{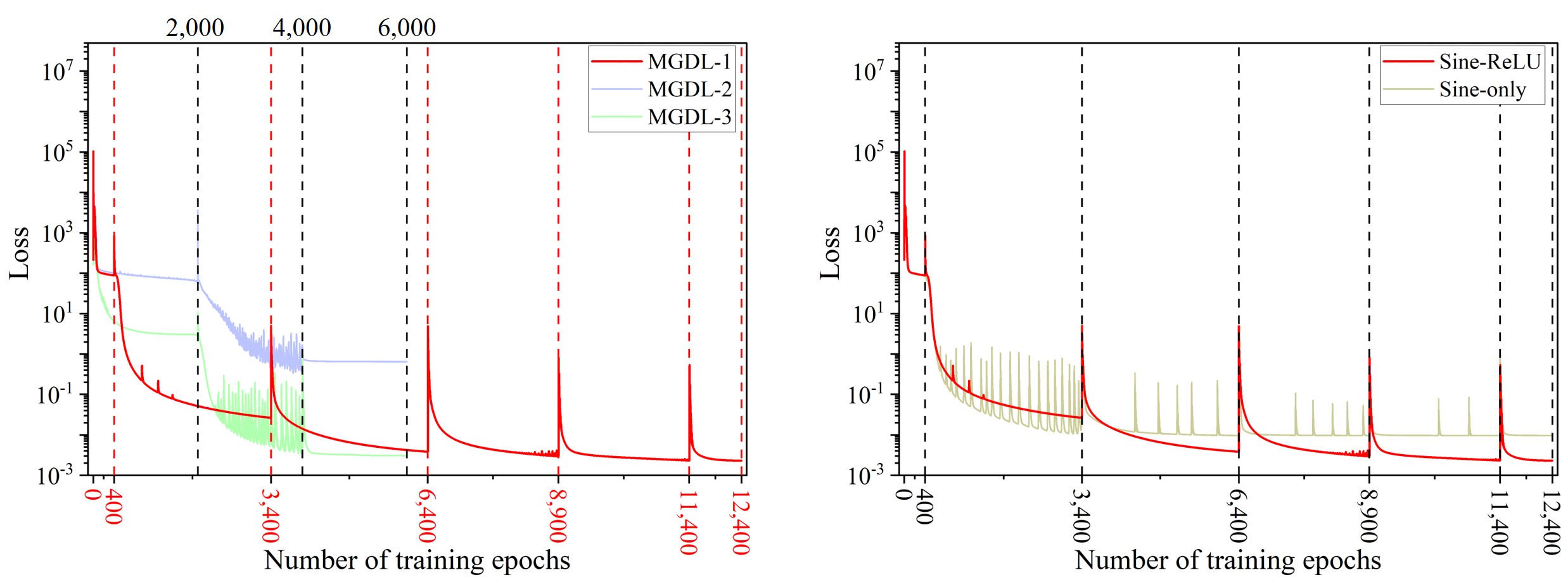}
\caption{Training-loss curves for MGDL-1, MGDL-2, MGDL-3, and MGDL-4 (Sine-only) on the two-dimensional Helmholtz problem \eqref{2D-Helmholtz-equation}.}
\label{MGDL-ablation}
\end{figure}

The results in Table \ref{NumericalResult_for_ablation} and Figure \ref{MGDL-ablation} demonstrate that both the grade-wise architectural design and the activation-function configuration influence the performance of MGDL. Under the same Sine‑ReLU activation setting, the structural ablation among MGDL‑1, MGDL‑2 and MGDL‑3 shows that the adaptive multi‑grade shallow correction strategy outperforms fixed deep designs. MGDL‑1 reaches final TrRSE and TeRSE of $5.37\times10^{-4}$, with a smooth and progressive error reduction across grades. In contrast, MGDL‑2 concentrates comparable total depth into only three grades and ends with substantially higher errors ($4.61\times10^{-3}$ and $4.77\times10^{-3}$). MGDL‑3 achieves a similar accuracy to MGDL‑1 ($5.43\times10^{-4}$) but requires a longer accumulated training time (7526 s versus 6272 s). These results indicate that distributing capacity over multiple shallow, adaptively introduced grades offers a better balance between accuracy, optimization behaviour and computational cost.

We further compare MGDL‑1 (Sine‑ReLU) with MGDL‑4 (Sine‑only). The Sine‑only variant yields final TrRSE and TeRSE about $8.7\%$ and $9.2\%$ higher than those of Sine‑ReLU, but still of the same order, indicating that pure Sine activations remain effective for the oscillatory solution. However, Sine‑only requires roughly $17\%$ more training time and exhibits noticeable loss oscillations and unstable convergence. In our framework, we use finite differences (FD) rather than automatic differentiation (AD) to compute the second‑order residuals. This avoids the gradient‑vanishing pathology that typically hinders ReLU in AD‑based PINNs. Thus, the Sine‑ReLU configuration harnesses the representational strength of Sine for oscillatory components while benefiting from ReLU’s efficiency and stability, leading to better accuracy and shorter training time than the homogeneous Sine‑only choice.

Overall, these results demonstrate that the proposed adaptive grade-wise architecture achieves a favorable balance among accuracy, optimization stability, and computational efficiency. The structural comparison confirms that distributing model capacity across multiple shallow grades is more effective than concentrating comparable depth within fewer grades. Meanwhile, the activation-function comparison shows that the Sine--ReLU configuration preserves the ability to approximate high-frequency oscillatory solutions while reducing the computational cost associated with the Sine-only model. These observations justify both the grade-wise shallow architecture and the activation-function design adopted in the proposed MGDL framework.

\section{Two-Dimensional Helmholtz Equations}\label{2d}
In this section, we consider the numerical solution of two-dimensional Helmholtz equations with Dirichlet boundary conditions at large wavenumbers:
\begin{equation}\label{2d-Helmholtz-Dirichlet}
    \left\{\begin{aligned}
        &\frac{\partial^2u}{\partial x_1^2}+\frac{\partial^2u}{\partial x_2^2}+\kappa^2u(x_1,x_2)=f(x_1,x_2), &&(x_1,x_2)\in \Omega,\\
        &u(x_1,x_2) = g(x_1,x_2),&&(x_1,x_2)\in \Gamma.
    \end{aligned}\right.
\end{equation}
We compare the performance of the proposed FD-MGDL method with several representative approaches, including AD-MGDL, FD-SGDL, Mscale \cite{liu2020multi}, FBPINN \cite{moseley2023finite}, SIREN \cite{sitzmann2020implicit}, PINN \cite{raissi2019physics}, and Pre-PINN \cite{ryck2024operator}.

Specifically, \textbf{Mscale} employs a multi-scale neural network architecture based on frequency scaling to enhance the approximation of high-frequency components; \textbf{FBPINN} adopts a domain decomposition framework in which the solution is represented as a sum of neural-network-based basis functions with compact support over overlapping subdomains; \textbf{SIREN} utilizes sinusoidal activation functions to construct implicit neural representations capable of accurately capturing oscillatory solutions and their derivatives. Following \cite{sitzmann2020implicit}, our SIREN implementation adopts standard weight initialization, a uniform distribution over $[-1/n, 1/n]$ for the first layer and $[-\sqrt{6/n}, \sqrt{6/n}]$ for subsequent layers, with zero initial biases, and applies unscaled sine activations across all hidden layers. Meanwhile, Pre-PINN \cite{ryck2024operator} reformulates PINN training from an operator preconditioning perspective to enhance convergence behavior under ill-conditioned settings.

To ensure a fair comparison, single-network baselines (such as FD-SGDL and SIREN) were configured with an overall depth and width matching the aggregated architecture across all trained grades of FD-MGDL. Consequently, these baselines possess a representational capacity equivalent to the complete multi-grade model rather than a single constituent grade. For structurally distinct architectures where direct layer-by-layer alignment is inapplicable—particularly Mscale and FBPINN—network configurations were calibrated so that their total trainable parameter counts matched those of the corresponding FD-MGDL model. This experimental design controls for model capacity, allowing the evaluation to focus directly on the efficacy of the multi-grade residual formulation and training strategy.


\subsection{Highly Oscillatory Sine Solution}\label{2d-sin}

We first consider the same two-dimensional Helmholtz problem with a highly oscillatory sine solution as in the previous section, now employed as a benchmark to evaluate and compare different numerical methods in the high-wavenumber regime. The source term is set to zero, 
\begin{equation*}
    f(x_1,x_2) = 0, \quad (x_1,x_2) \in \Omega = (0,1)^2.
\end{equation*}

Dirichlet boundary conditions are imposed on the boundary components $\Gamma=\{\Gamma_1,\Gamma_2,\Gamma_3,\Gamma_4\}$ as
\begin{equation*}
    g(x_1,x_2)=
    \begin{cases}
        0,  & \text{ if } (x_1,x_2)\in\Gamma_1:=\{0\}\times(0,1); \\
        0,  & \text{ if } (x_1,x_2)\in\Gamma_2:=(0,1)\times\{0\}; \\
        \sin\left(\frac{\sqrt{2}}{2}\kappa\right)\sin\left(\frac{\sqrt{2}}{2}\kappa x_2\right),  & \text{ if } (x_1,x_2)\in\Gamma_3:=\{1\}\times(0,1); \\
        \sin\left(\frac{\sqrt{2}}{2}\kappa x_1\right)\sin\left(\frac{\sqrt{2}}{2}\kappa\right),  & \text{ if } (x_1,x_2)\in\Gamma_4:=(0,1)\times\{1\},
    \end{cases}
\end{equation*}
where $\kappa$ is a constant wavenumber. Since the domain is the unit square $\Omega=(0,1)^2$ with characteristic length $I=1$, the dimensionless wavenumber is $\kappa I=\kappa$. The unique exact solution of problem \eqref{2d-Helmholtz-Dirichlet} is 
\begin{equation}\label{2d-sin-solution}
    u(x_1,x_2) = \sin\left(\frac{\sqrt{2}}{2}\kappa x_1\right)\sin\left(\frac{\sqrt{2}}{2}\kappa x_2\right),\quad (x_1,x_2)\in [0,1]\times[0,1],
\end{equation}
which becomes increasingly oscillatory as the wavenumber $\kappa$  grows. Consequently, larger $\kappa$ implies more wavelengths across the domain, intensifying the pollution effect and making the problem numerically more challenging.

We first solve equation \eqref{2d-Helmholtz-Dirichlet} with the exact solution \eqref{2d-sin-solution} at wavenumber $\kappa\in\{50,100,150,200\}$ by FD-MGDL and seven competing methods including AD-MGDL, FD-SGDL, Mscale, FBPINN, SIREN, PINN and Pre-PINN. For all the methods except Mscale,  we set $m=(300,500,700,700)$ for the respective wavenumbers and partition the interval $[0,1]$ into $m+1$ equally spaced subintervals with points
\begin{equation*}
    0=x_{0,i}<x_{1,i}<\cdots<x_{m,i}<x_{m+1,i}=1, \ \ i=1,2, 
\end{equation*}
where $x_{j,i}:=jh$ and $h:=1/(m+1)$. The corresponding training points are chosen as the tensor-product grid $(x_{j_1,1},x_{j_2,2})$, $j_1,j_2\in\mathbb{N}_m$. 
To ensure computational parity for \textbf{Mscale}, we account for its multi-branch architecture, where each input sample is processed across $S$ parallel scale networks. Because this structure increases forward and backward pass evaluations by a factor of $S$, evaluating $N$ points would disproportionately increase the per-iteration workload. We therefore set Mscale's batch size to $N/S$ points, randomly sampled from the domain, where $N$ represents the total grid points used by the other baselines. This normalization equalizes the total operations per gradient update across all evaluated methods, ensuring that training time comparisons isolate algorithmic efficiency under an equivalent computational budget.
%
In a similar manner, for all methods the testing points are selected as $ (\tilde{x}_{j_1,1},\tilde{x}_{j_2,2})$, $j_1,j_2\in\mathbb{N}_{\tilde{m}}$ with $\tilde{m}=(150,250,350,350)$ corresponding to the four wavenumbers. 

\begin{table}[!htb]\centering
    \small
    \setlength{\tabcolsep}{15pt}
    \begin{threeparttable}
    \caption{Performance comparison of FD-MGDL, AD-MGDL, FD-SGDL, Mscale, FBPINN, SIREN, PINN and Pre-PINN for the 2D Helmholtz problem \eqref{2d-Helmholtz-Dirichlet} with the exact solution \eqref{2d-sin-solution}.}
    \label{NumericalResult_for_2dPDE}
    \centering
    \begin{tabular}{clrrcc} 
        \toprule
         $\kappa$ & Method & Epochs & AC time (s)  & TrRSE & TeRSE \\
        \midrule
        \multirow{8}{*}{$50$} 
        & FD-MGDL & 12,400 & {\bf 6,272} & $\mathbf{5.37\times 10^{-4}}$ & $\mathbf{5.37\times 10^{-4}}$ \\
        & AD-MGDL & 12,400 & 15,152 & $6.03\times10^{-3}$ & $5.97\times10^{-3}$\\
        & FD-SGDL & 15,000 & 41,431 &$1.76\times 10^{-3}$ &$1.90\times 10^{-3}$ \\
        & Mscale & 15,000 & 9,277 & $2.14\times10^{-2}$ & $2.09\times10^{-2}$ \\
        & FBPINN & 15,000 &11,173 &$1.19\times 10^{-2}$ &$1.20\times 10^{-2}$ \\
        & SIREN & 15,000 & 32,944 & $3.06\times 10^{-2}$ & $3.08\times10^{-2}$ \\
        & PINN & 15,000 & 49,990 &$1.00\times 10^{0}$ &$1.00\times 10^{0}$ \\
        & Pre-PINN & 15,000 &61,286 &$1.00\times 10^{0}$ &$1.00\times 10^{0}$ \\  
        \midrule
        \multirow{8}{*}{$100$} & FD-MGDL & 6,500 & {\bf 9,196} & $\mathbf{5.54\times 10^{-3}}$ & $\mathbf{5.56\times 10^{-3}}$ \\
        & AD-MGDL & 6,500 & 19,044 & $8.24\times 10^{-3}$ & $8.14\times 10^{-3}$ \\
        & FD-SGDL & 8,000 & 47,792 & $1.65\times 10^{-2}$ &$1.87\times 10^{-2}$ \\
        & Mscale & 8,000 & 13,224 & $2.44\times10^{-2}$ & $2.44\times10^{-2}$ \\
        & FBPINN & 8,000 &12,965 &$4.33\times 10^{-2}$ &$4.33\times 10^{-2}$ \\
        & SIREN & 8,000 & 32,737 & $9.96\times 10^{-2}$ & $1.00\times10^{-1}$ \\
        & PINN & 8,000 & 53,636 &$1.00\times 10^{0}$ &$1.00\times 10^{0}$ \\
        & Pre-PINN & 8,000 & 70,695 &$1.00\times 10^{0}$ &$1.00\times 10^{0}$ \\
        \midrule
        \multirow{8}{*}{$150$} & FD-MGDL & 8,000 & {\bf 21,861} & $\mathbf{1.31\times 10^{-2}}$ & $\mathbf{1.34\times 10^{-2}}$ \\
        & AD-MGDL & 8,000 & 51,239 & $1.03\times 10^{-1}$ & $1.02\times 10^{-1}$ \\
        & FD-SGDL & 10,000 & 108,747  & $5.22\times10^{-1}$  &$5.23\times10^{-1}$  \\
        & Mscale & 10,000 & 33,191 & $6.02\times10^{-2}$ & $6.01\times10^{-2}$ \\
        & FBPINN & 10,000 &26,827 &$6.28\times 10^{-2}$ &$6.29\times 10^{-2}$ \\  
        & SIREN & 10,000 & 98,435 & $1.06\times10^{-1}$ & $1.06\times10^{-1}$ \\
        & PINN & 10,000 & 131,216 & $1.00\times10^{0}$ & $1.00\times10^{0}$ \\
        & Pre-PINN & 10,000 & 160,863 & $1.00\times10^{0}$ & $1.00\times10^{0}$ \\
        \midrule
        \multirow{8}{*}{$200$} & FD-MGDL & 6,500 & {\bf 17,776} & $\mathbf{3.92\times10^{-1}}$ & $\mathbf{3.92\times10^{-1}}$ \\
        & AD-MGDL & 6,500 & 37,321 & $8.22\times10^{-1}$ & $8.22\times10^{-1}$ \\
        & FD-SGDL & 8,000 & 88,110 & $8.79\times10^{-1}$ & $8.79\times10^{-1}$ \\
        & Mscale & 8,000 & 23,479 & $4.75\times10^{-1}$ & $4.75\times10^{-1}$ \\
        & FBPINN & 8,000 & 26,827 & $7.27\times10^{-1}$ & $7.28\times10^{-1}$ \\
        & SIREN & 8,000 & 63,123 & $8.95\times10^{-1}$ & $8.98\times10^{-1}$ \\
        & PINN & 8,000 & 106,135 & $1.00\times10^{0}$ & $1.00\times10^{0}$ \\
        & Pre-PINN & 8,000 & 139,890 & $1.00\times10^{0}$ & $1.00\times10^{0}$ \\
        \bottomrule
    \end{tabular}
    \end{threeparttable}
\end{table}

\begin{figure}[!htb]
\centering
\includegraphics[scale=0.85]{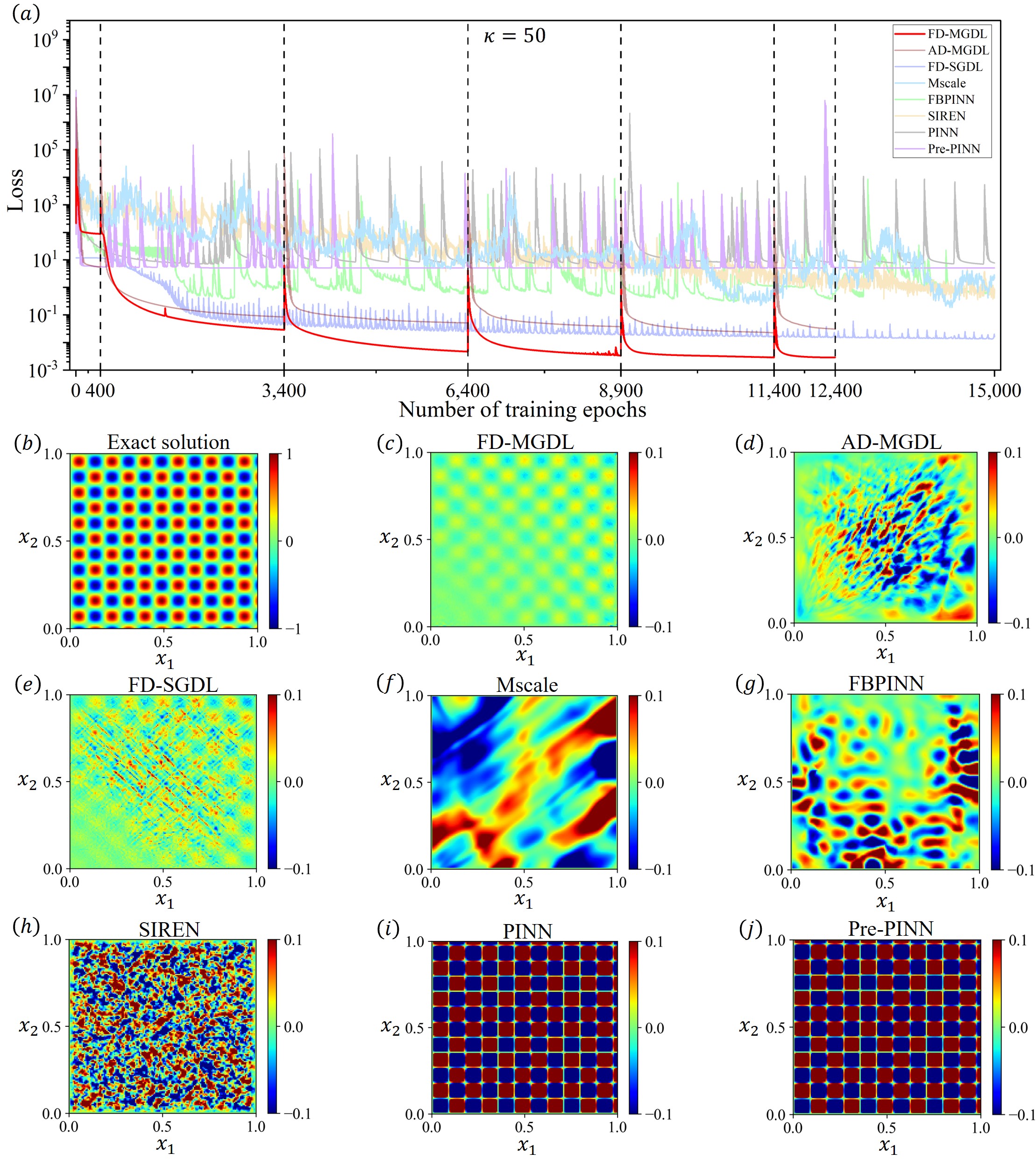}
\caption{Performance comparison of FD-MGDL and seven baseline methods (AD-MGDL, FD-SGDL, Mscale, FBPINN, SIREN, PINN and Pre-PINN) for the 2D Helmholtz problem \eqref{2d-Helmholtz-Dirichlet} with the exact solution \eqref{2d-sin-solution} at $\kappa=50$: $(a)$ training loss curves; $(b)$ exact solution; $(c)-(j)$ error visualizations of the numerical solutions.}
\label{2dsin-50}
\end{figure}

\begin{figure}[!htb]
\centering
\includegraphics[scale=0.85]{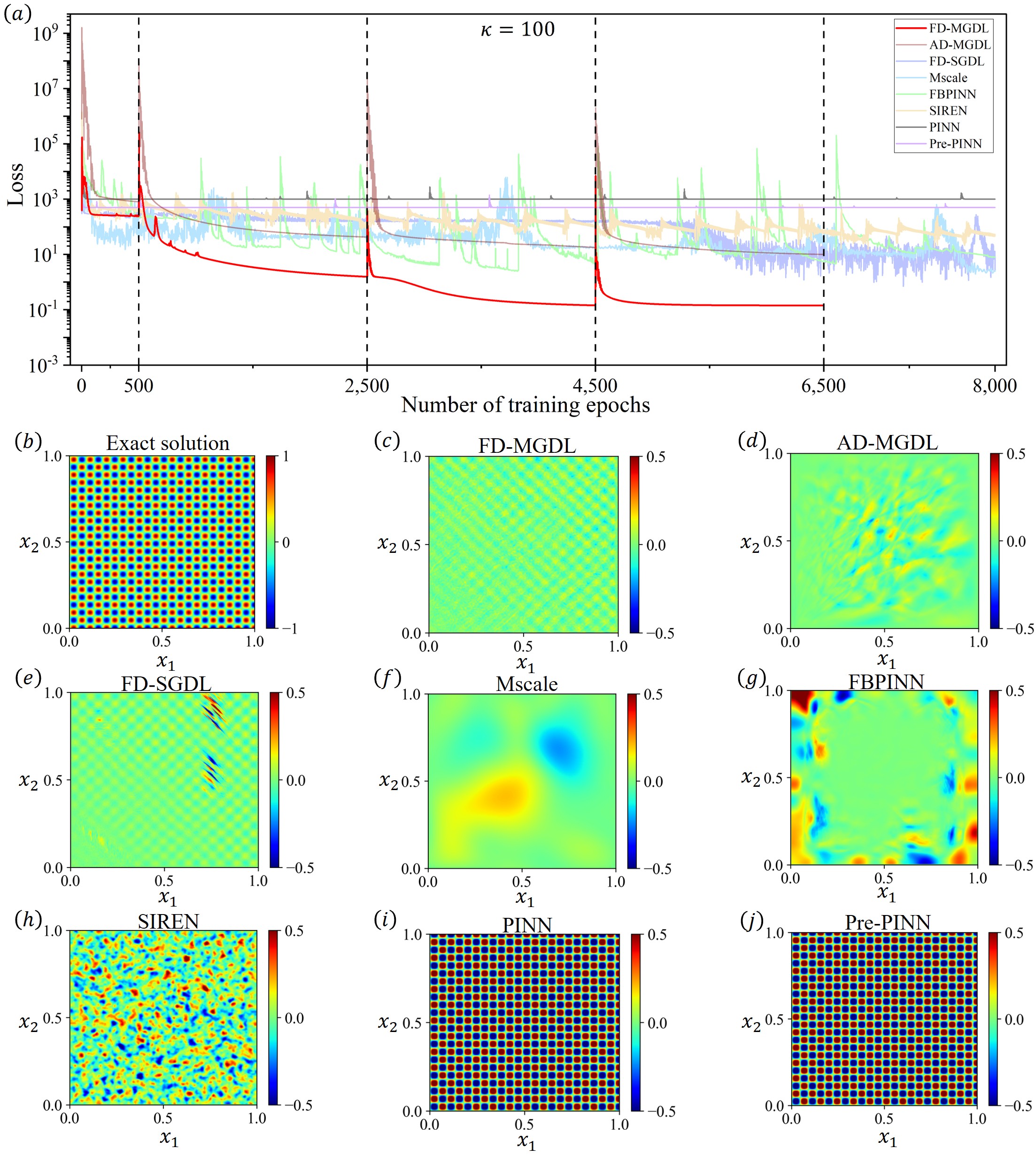}
\caption{Performance comparison of FD-MGDL and seven baseline methods (AD-MGDL, FD-SGDL, Mscale, FBPINN, SIREN, PINN and Pre-PINN) for the 2D Helmholtz problem \eqref{2d-Helmholtz-Dirichlet} with the exact solution \eqref{2d-sin-solution} at $\kappa=100$: $(a)$ training loss curves; $(b)$ exact solution; $(c)-(j)$ error visualizations of the numerical solutions.}
\label{2dsin-100}
\end{figure}

\begin{figure}[!htb]
\centering
\includegraphics[scale=0.85]{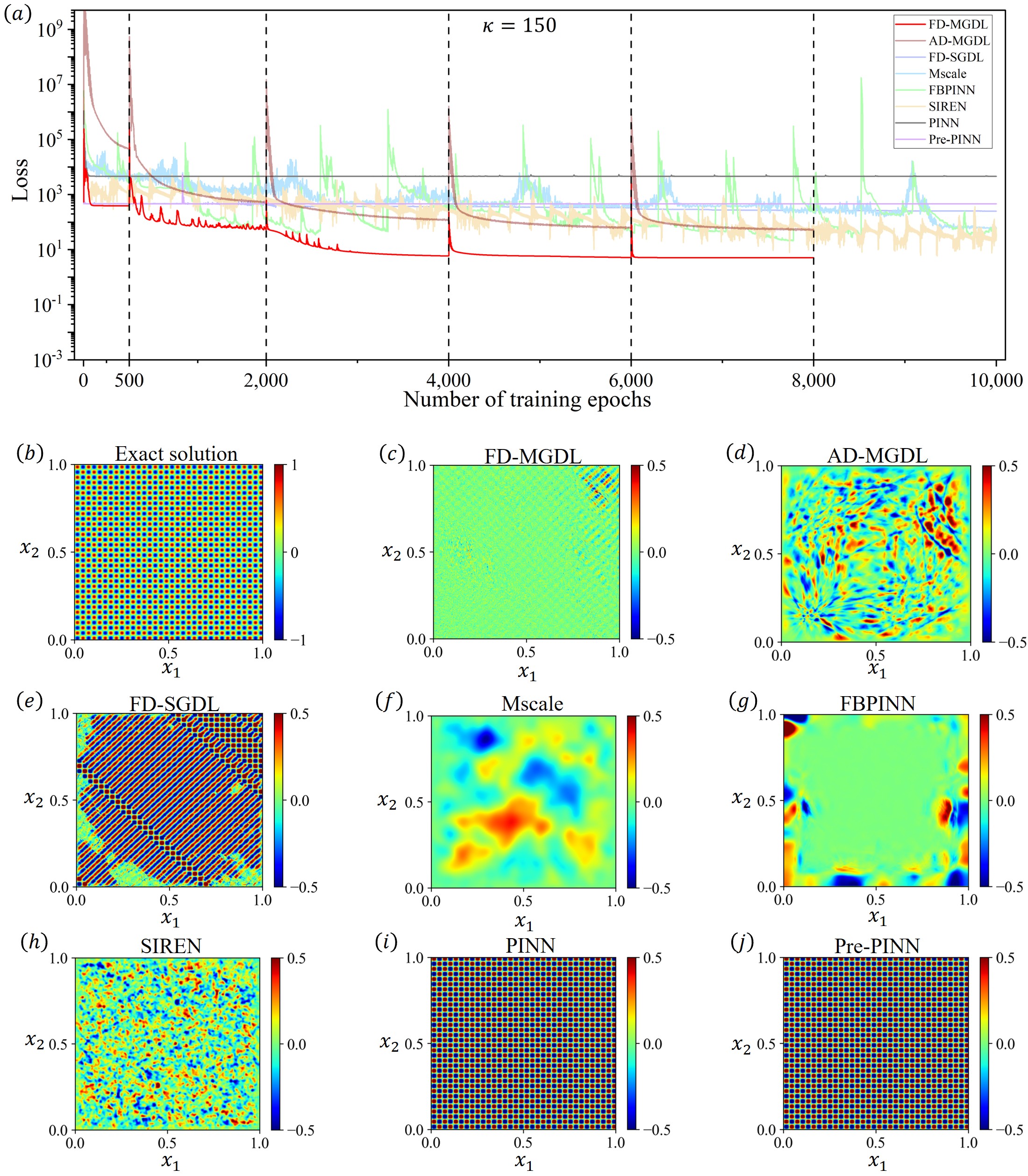}
\caption{Performance comparison of FD-MGDL and seven baseline methods (AD-MGDL, FD-SGDL, Mscale, FBPINN, SIREN, PINN and Pre-PINN) for the 2D Helmholtz problem \eqref{2d-Helmholtz-Dirichlet} with the exact solution \eqref{2d-sin-solution} at $\kappa=150$: $(a)$ training loss curves; $(b)$ exact solution; $(c)-(j)$ error visualizations of the numerical solutions.}
\label{2dsin-150}
\end{figure}

\begin{figure}[!htb]
\centering
\includegraphics[scale=0.81]{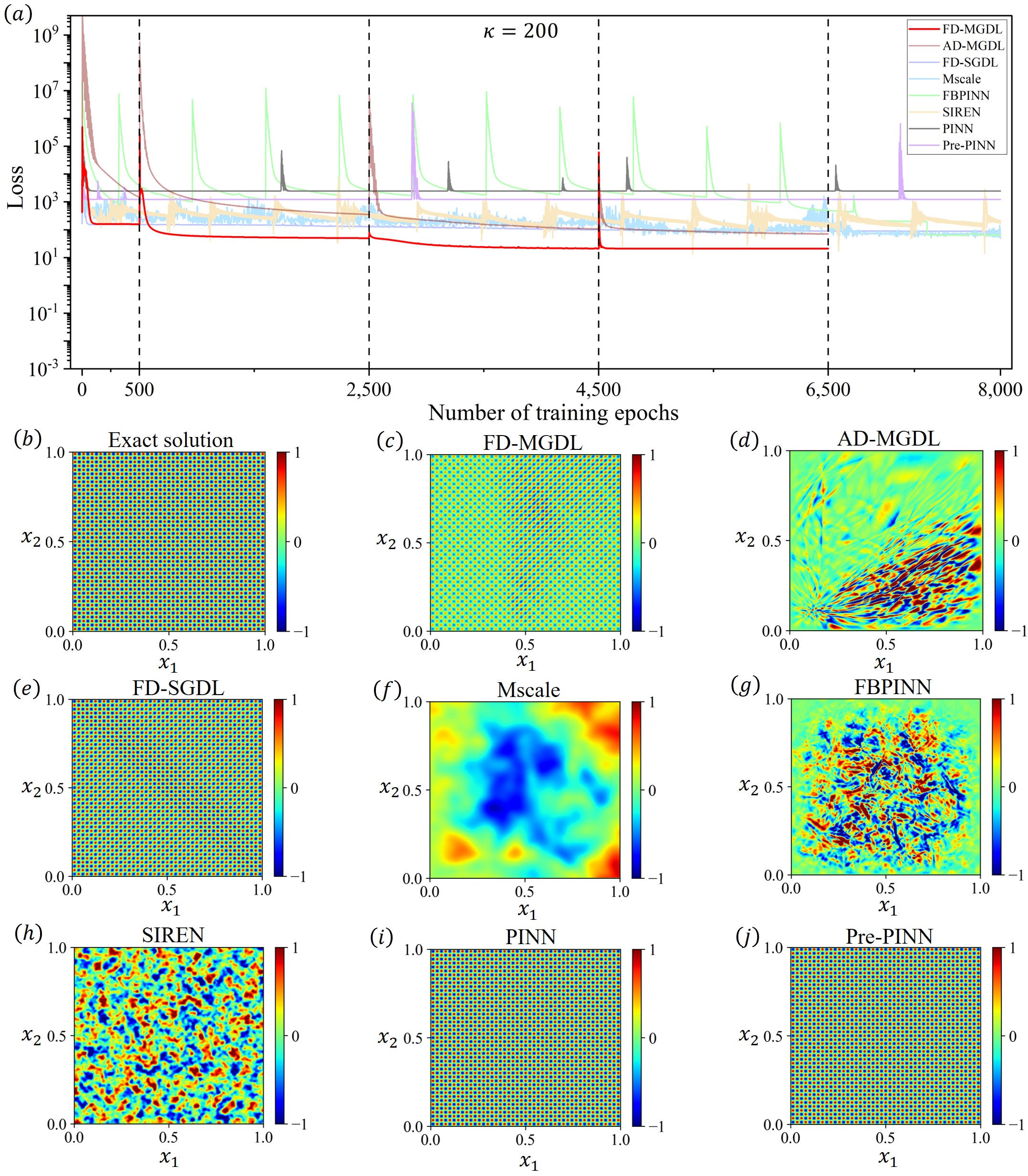}
\caption{Performance comparison of FD-MGDL and seven baseline methods (AD-MGDL, FD-SGDL, Mscale, FBPINN, SIREN, PINN and Pre-PINN) for the 2D Helmholtz problem \eqref{2d-Helmholtz-Dirichlet} with the exact solution \eqref{2d-sin-solution} at $\kappa=200$: $(a)$ training loss curves; $(b)$ exact solution; $(c)-(j)$ error visualizations of the numerical solutions.}
\label{2dsin-200}
\end{figure}

\begin{figure}[!htb]
\centering
\includegraphics[scale=0.83]{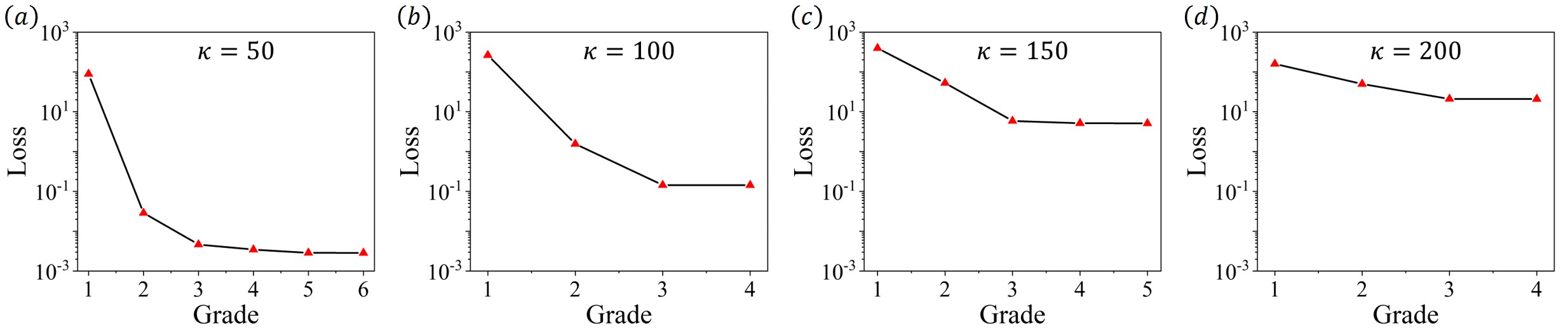}
\caption{Non-increasing grade-end training loss of FD-MGDL.}
\label{Optimal-loss}
\end{figure}

Table \ref{NumericalResult_for_2dPDE} presents a quantitative comparison between FD-MGDL and seven baseline methods for the Helmholtz problem \eqref{2d-Helmholtz-Dirichlet} with the exact solution \eqref{2d-sin-solution} at wavenumber $\kappa=50, 100, 150, 200$. The reported metrics include the number of training epochs, AC time, TrRSE and TeRSE. These indicators enable a comprehensive evaluation of each method in terms of accuracy, generalization, and computational efficiency. Table \ref{NumericalResult_for_2dPDE} clearly shows that FD-MGDL outperforms all competing methods in solution accuracy.

In addition to evaluating FD-MGDL against established baselines (Mscale, FBPINN, SIREN, PINN, and Pre-PINN), we analyze four internal variants, FD-MGDL, AD-MGDL, FD-SGDL, and AD-SGDL, which form a controlled $2 \times 2$ factorial comparison. This setup isolates two core methodological choices: the differentiation scheme (FD versus AD) and the training strategy (MGDL versus SGDL), where standard PINN corresponds to AD-SGDL.

Across the wavenumber range $\kappa \in [50, 200]$, comparing FD-MGDL, AD-MGDL, FD-SGDL, and AD-SGDL reveals the distinct and complementary benefits of the multi-grade training strategy and the finite-difference approximation of differentiation. The testing errors ($\mathrm{TeRSE}$) for standard AD-SGDL remain at $1.00$ across all tested wavenumbers, indicating that conventional single-grade automatic differentiation completely fails to resolve highly oscillatory wave fields. By contrast, introducing progressive grade-wise residual learning under automatic differentiation (AD-MGDL) consistently outperforms AD-SGDL. This isolated comparison confirms the intrinsic capability of the multi-grade architecture to alleviate spectral bias and overcome optimization bottlenecks in high-frequency regimes.

At $\kappa=50$, FD-SGDL outperforms AD-MGDL, demonstrating that the finite difference approximation of differentiation alone significantly improves trainability in moderate-wavenumber regimes. However, this advantage diminishes as $\kappa$ increases because SGDL lacks the multi-scale error decomposition inherent to MGDL. Notably, AD-MGDL surpasses FD-SGDL by factors of approximately $2.3$ and $5.1$ at $\kappa=100$ and $150$, respectively. This crossover indicates that the single-grade training paradigm becomes the primary bottleneck at higher wavenumbers, compounding the inherent limitations of automatic differentiation. Mechanistically, FD refines the construction and evaluation of the Helmholtz residual, whereas MGDL provides the structural framework needed to decompose and learn high-frequency error components. Consequently, FD-MGDL achieves the lowest error across all tested wavenumbers. These results confirm that neither the multi-grade architecture nor the FD formulation alone is sufficient over the entire wavenumber spectrum; rather, the strength of FD-MGDL lies in their synergy: MGDL partitions the target solution into progressive residual-learning tasks, while FD supplies a reliable residual signal to drive the correction networks.

The marked advantage of FD within MGDL aligns directly with the choice of ReLU correction networks in subsequent grades. Because a ReLU network is piecewise linear, its second partial spatial derivatives obtained via pointwise automatic differentiation vanish almost everywhere. Consequently, an AD-derived Laplacian fails to provide a meaningful gradient signal for local variations in ReLU corrections during collocation training. In contrast, the finite difference operator evaluates spatial function values across neighboring grid points, defining a well-conditioned discrete Helmholtz residual without requiring classical second-order differentiability. Thus, FD is not merely a computational convenience, but a mathematical formulation naturally compatible with grade-wise ReLU refinement.

This compatibility perspective is further corroborated by our baseline benchmarks: at higher wavenumbers ($\kappa = 150$ and $200$), AD-MGDL is outperformed by Mscale and FBPINN. Unlike the ReLU-activated correction grades in AD-MGDL, Mscale and FBPINN employ $C^2$-smooth network mappings, bypassing the AD--ReLU derivative vanishing issue.

Beyond mathematical compatibility, FD-MGDL yields a $2.1$--$2.4$-fold speedup over AD-MGDL across $\kappa \in [50, 200]$. While FD evaluates the discrete operator directly from forward-pass network predictions using central difference stencils, AD requires constructing and traversing a higher-order computational graph to evaluate second derivatives. Bypassing this second-order autodiff overhead dramatically reduces both memory consumption and wall-clock execution time, a benefit that becomes increasingly pronounced as network depth and collocation density grow.

A broader benchmarking against Mscale, FBPINN, SIREN, standard PINN, and Pre-PINN highlights the superior overall performance of FD-MGDL relative to existing neural solvers. Across $\kappa \in [50, 150]$, FD-MGDL substantially outperforms all competing baselines. Notably, standard PINN and Pre-PINN yield $\mathrm{TeRSE} = 1.00$ throughout this range, indicating a failure to capture non-trivial solution dynamics. While Mscale and FBPINN improve upon conventional PINN formulations, neither matches the accuracy of FD-MGDL. Similarly, SIREN yields inferior precision despite its periodic activation prior. At $\kappa = 200$, the performance margin narrows as numerical difficulties escalate: Mscale achieves a $\mathrm{TeRSE}$ of $4.75 \times 10^{-1}$, compared to $3.92 \times 10^{-1}$ for FD-MGDL. Although FD-MGDL retains leading accuracy, this smaller margin reflects the formidable challenge of resolving extremely high-frequency wave fields.

The loss convergence trajectories and error spatial maps across Figs. \ref{2dsin-50}--\ref{2dsin-200} strongly mirror our quantitative findings. At $\kappa = 50$, FD-MGDL exhibits rapid loss decay and produces the smallest, most spatially uniform error field. FD-SGDL demonstrates competitive performance in this moderate regime, whereas AD-MGDL plateaus at a higher error floor. As the wavenumber scales to $\kappa = 100$ and $150$, AD-MGDL surpasses FD-SGDL, visually confirming the distinct benefit of the multi-grade architecture. Across this intermediate frequency range, FD-MGDL leverages both progressive grade-wise convergence and well-conditioned finite-difference residuals, achieving minimal amplitude and phase errors among all evaluated methods. At $\kappa = 200$, the loss profiles and spatial distributions in Fig. \ref{2dsin-200} reflect severe performance degradation across all solvers. Although FD-MGDL still yields the lowest testing error ($\mathrm{TeRSE} = 3.92 \times 10^{-1}$) and the cleanest error profile, this outcome represents a relative rather than absolute success: FD-MGDL substantially delays high-frequency breakdown, but does not eradicate it entirely. Notably, the close tracking between $\mathrm{TrRSE}$ and $\mathrm{TeRSE}$ across convergent runs confirms that this degradation stems from fundamental representation and discretization challenges rather than overfitting.

To examine whether finite-step optimization empirically mirrors the monotonic behavior predicted by Theorem \ref{non-increasing} for theoretical grade-wise minimizers, we record the training loss at the conclusion of each grade. Figure \ref{Optimal-loss} displays the resulting grade-end loss trajectories across $\kappa \in \{50, 100, 150, 200\}$. For all tested wavenumbers, the grade-end loss remains monotonically non-increasing with respect to the grade index. Specifically, each newly trained correction network yields a final loss bounded above by that inherited from the preceding grade. Thus, while finite-step Adam optimization does not strictly guarantee exact grade-wise minimizers, the computed numerical iterates empirically satisfy the fundamental descent property predicted by Theorem \ref{non-increasing}.

Overall, Table \ref{NumericalResult_for_2dPDE} and Figs. \ref{2dsin-50}--\ref{2dsin-200} establish MGDL and FD as distinct yet highly complementary pillars of the proposed framework. Architecturally, MGDL serves as the primary mechanism for mitigating spectral bias, decoupling multi-scale errors, and systematically learning progressive oscillatory corrections. Numerically, FD supplies a mathematically sound discrete Helmholtz operator for non-smooth ReLU corrections while bypassing the heavy computational burden of second-order automatic differentiation. The performance edge of AD-MGDL over FD-SGDL at $\kappa \in \{100, 150\}$ proves that the impact of MGDL extends far beyond simply substituting AD with FD. Conversely, the consistent superiority of FD-MGDL over AD-MGDL demonstrates that MGDL reaches its full potential when paired with an operator formulation compatible with piecewise linear corrections. Their integration yields optimal performance across the entire tested spectrum, though the findings at $\kappa = 200$ also delineate the operational limits of the framework in extreme high-frequency regimes.

\subsection{Plane Wave Propagation Problem}\label{2d-wave}

We next evaluate a two-dimensional Helmholtz problem featuring a plane wave solution to further assess solver performance in high-wavenumber regimes. Given the poor performance of the baseline PINN, SIREN, and Pre-PINN models observed in the last subsection, we restrict all subsequent comparative evaluations throughout the remainder of this paper to FD-MGDL, FD-SGDL, Mscale, and FBPINN.

In this test case, we consider a homogeneous source term, $f(\mathbf{x})=0$, on the unit square $\Omega:=(0,1)^2$. Dirichlet boundary conditions are prescribed on $\partial\Omega := \bigcup_{j=1}^4 \Gamma_j$ such that the boundary data $g$ is consistent with the exact monochromatic plane-wave solution:
\begin{equation}\label{2d-wave-solution}u(x_1,x_2) = \exp\left[i(\kappa_1 x_1 + \kappa_2 x_2)\right], \quad \mathbf{x} \in \bar{\Omega},
\end{equation}
where the wave vector components are $(\kappa_1, \kappa_2) := \kappa(\cos\theta, \sin\theta)$, with $\kappa$ denoting the wavenumber and $\theta \in [0, 2\pi)$ the propagation angle.

Specifically, the boundary function $g$ is defined piecewise:
\begin{equation*}
    g(x_1,x_2)=
    \begin{cases}
        \exp(i\kappa_2 x_2),  & \text{on } \Gamma_1 := \{0\}\times(0,1); \\ 
        \exp(i\kappa_1 x_1),  & \text{on } \Gamma_2 := (0,1)\times \{0\}; \\ 
        \exp[i(\kappa_1+\kappa_2 x_2)],  & \text{on } \Gamma_3 := \{1\}\times(0,1); \\ 
        \exp[i(\kappa_1 x_1+\kappa_2)],  & \text{on } \Gamma_4 := (0,1)\times\{1\}.
    \end{cases}
\end{equation*}
As the wavenumber $\kappa$ increases, the solution becomes highly oscillatory, posing a significant challenge for numerical approximation methods due to the stringent sampling requirements needed to resolve the wave phases.

We evaluate the performance of FD-MGDL, FD-SGDL, Mscale, and FBPINN for  $\kappa \in \{50, 100, 150, 200\}$ and a fixed direction $\theta = \pi/4$. To resolve the increasing oscillations, we scale the number of grid points $m$ with $\kappa$, setting $m \in \{300, 500, 700, 700\}$ for training and $\tilde{m} \in \{150, 250, 350, 350\}$ for testing. Training points for most methods are arranged on a tensor-product grid; however, for Mscale, points are randomly sampled to maintain a total count equivalent to the other methods divided by the number of scales.


\begin{table}[!htb]\centering
    \small
    \setlength{\tabcolsep}{15pt}
    \begin{threeparttable}
    \caption{Performance comparison of FD-MGDL, FD-SGDL, Mscale and FBPINN for the 2D Helmholtz problem \eqref{2d-Helmholtz-Dirichlet} with the exact solution \eqref{2d-wave-solution}.}
    \label{NumericalResult_for_2dplanewave}
    \centering
    \begin{tabular}{clrrcc} 
        \toprule
         $\kappa$ & Method & Epoch & AC time (s)  & TrRSE & TeRSE \\
        \midrule
        \multirow{4}{*}{$50$} & FD-MGDL & 4,500 & {\bf 2,540} & $\mathbf{3.40\times10^{-4}}$ & $\mathbf{3.42\times10^{-4}}$ \\
        & FD-SGDL & 6,000 & 14,079 & $3.54\times10^{-3}$ & $3.57\times10^{-3}$  \\
        & Mscale & 6,000 & 2,823 & $3.90\times10^{-2}$ & $3.82\times10^{-2}$ \\
        & FBPINN  & 6,000 & 4,544 & $1.02\times10^{-2}$ & $1.02\times10^{-2}$ \\
        \midrule
        \multirow{4}{*}{$100$} & FD-MGDL & 6,500 & {\bf 9,493} & $\mathbf{2.82\times10^{-3}}$ & $\mathbf{2.83\times10^{-3}}$ \\
        & FD-SGDL & 8,000 & 57,681 & $4.03\times10^{-1}$ & $4.03\times10^{-1}$\\
        & Mscale & 8,000 & 11,548 & $1.89\times10^{-2}$ & $1.81\times10^{-2}$ \\
        & FBPINN  & 8,000 & 16,066 & $5.82\times10^{-2}$ & $5.82\times10^{-2}$\\
        \midrule
        \multirow{4}{*}{$150$} & FD-MGDL & 2,500 & {\bf 6,894} & $\mathbf{8.07\times10^{-3}}$ & $\mathbf{8.07\times10^{-3}}$ \\
        & FD-SGDL & 5,000 & 40,342 & $8.16\times10^{-2}$ & $8.19\times10^{-2}$ \\
        & Mscale & 5,000 & 13,799 & $2.12\times10^{-1}$ & $2.09\times10^{-1}$ \\
        & FBPINN  & 5,000 & 21,367 & $1.37\times10^{-1}$ & $1.37\times10^{-1}$ \\
        \midrule
        \multirow{4}{*}{$200$} & FD-MGDL & 3,500 & {\bf 10,462} & $\mathbf{5.20\times10^{-2}}$ & $\mathbf{5.21\times10^{-2}}$ \\
        & FD-SGDL & 5,000 & 70,964 & $7.37\times10^{-1}$ & $7.38\times10^{-1}$ \\
        & Mscale & 5,000 & 14,292 & $2.79\times10^{-1}$ & $2.74\times10^{-1}$ \\
        & FBPINN  & 5,000 & 20,103 & $2.65\times10^{-1}$ & $2.65\times10^{-1}$ \\
        \bottomrule
    \end{tabular}
    \end{threeparttable}
\end{table}

\begin{figure}[!htb]
\centering
\includegraphics[scale=0.83]{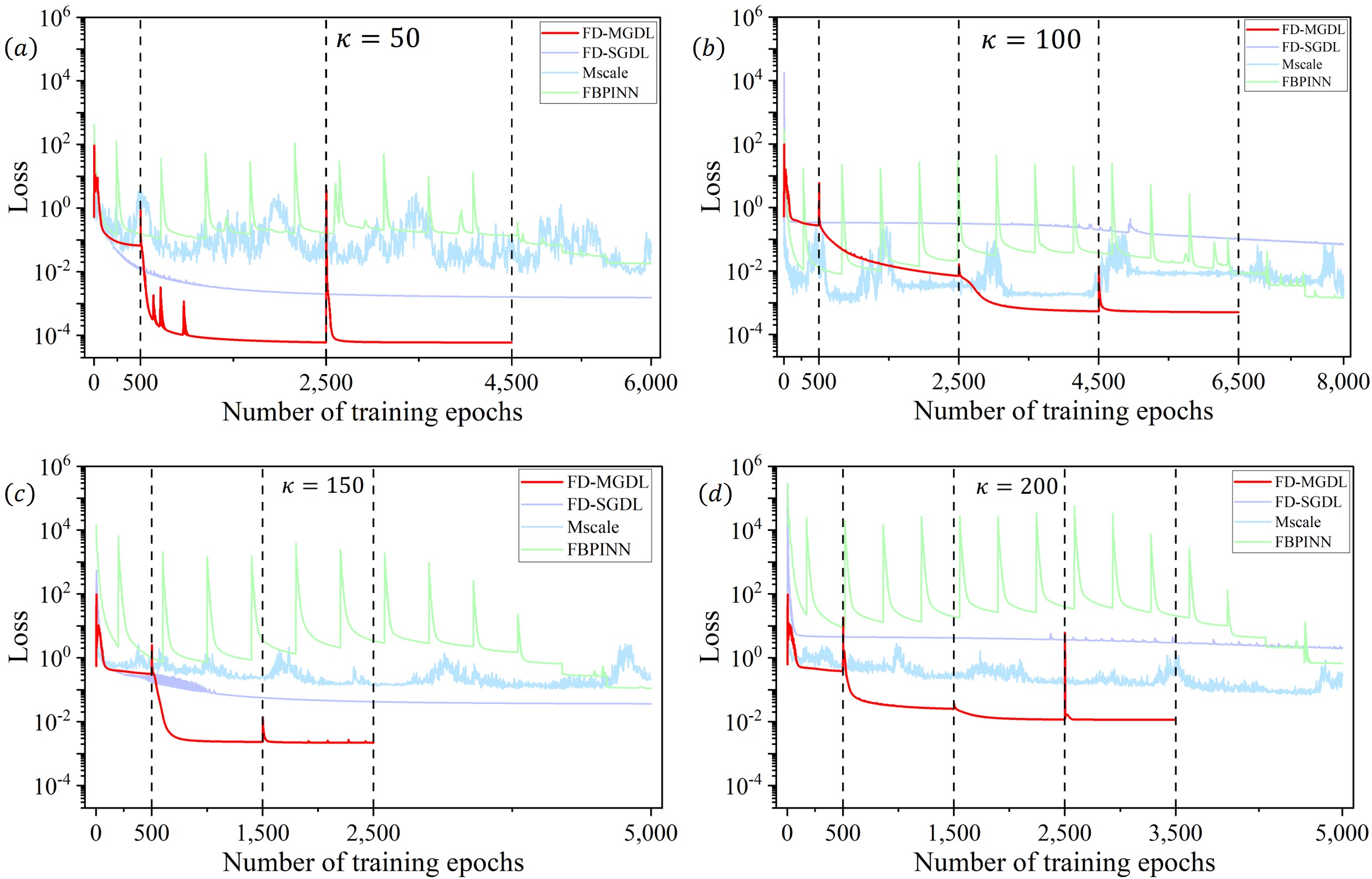}
\caption{Performance comparison of training loss curves of FD-MGDL and three baseline methods (FD-SGDL, Mscale and FBPINN) for the 2D Helmholtz problem \eqref{2d-Helmholtz-Dirichlet} with the exact solution \eqref{2d-wave-solution}: $(a)$ $\kappa=50$; $(b)$ $\kappa=100$; $(c)$ $\kappa=150$; $(d)$ $\kappa=200$.}
\label{2dwave-loss}
\end{figure}

Table \ref{NumericalResult_for_2dplanewave} summarizes the performance of FD-MGDL alongside three baselines — FD-SGDL, Mscale, and FBPINN — for wavenumbers $\kappa \in \{50, 100, 150, 200\}$. Across all cases, FD-MGDL demonstrates superior accuracy and efficiency. At $\kappa=50$, it achieves training (TrRSE) and testing (TeRSE) errors on the order of $10^{-4}$ with significantly fewer epochs and lower computational cost than FD-SGDL. In contrast, Mscale and FBPINN yield errors one to two orders of magnitude higher, struggling to resolve the plane wave structure.

As $\kappa$ increases, the intensified oscillations exacerbate numerical difficulty. While all baseline methods experience rapid accuracy degradation, FD-MGDL consistently maintains the lowest error levels. For $\kappa=100$ and $150$, its relative errors remain roughly an order of magnitude smaller than those of Mscale and FBPINN, and several orders smaller than FD-SGDL. Even at the extreme $\kappa=200$, FD-MGDL remains the most robust, whereas baselines suffer from pronounced error growth typically associated with pollution effects and optimization stagnation.

Convergence dynamics are further illustrated in Figure \ref{2dwave-loss}. FD-MGDL exhibits faster and more stable loss decay across the entire wavenumber spectrum. Conversely, FD-SGDL stagnates prematurely, and both Mscale and FBPINN display increasingly erratic trajectories as the frequency rises.

In summary, these experiments confirm that FD-MGDL achieves high precision at moderate wavenumbers and exceptional robustness in high-frequency regimes. By effectively mitigating the convergence hurdles and accuracy loss observed in existing PINN and multiscale solvers, the multigrade learning strategy establishes FD-MGDL as a highly reliable approach for oscillatory Helmholtz problems.

\section{Three-Dimensional Helmholtz Equations}\label{3d}

In this section, we consider the three-dimensional Helmholtz equation with Dirichlet boundary conditions:
\begin{equation}\label{3d-Helmholtz-Dirichlet}
    \left\{\begin{aligned}
        &\frac{\partial^2u}{\partial x_1^2}+\frac{\partial^2u}{\partial x_2^2}+\frac{\partial^2u}{\partial x_3^2}+\kappa^2u(x_1,x_2,x_3)=f(x_1,x_2,x_3),&&(x_1,x_2,x_3)\in \Omega,\\
        &u(x_1,x_2,x_3) = g(x_1,x_2,x_3),&&(x_1,x_2,x_3)\in \Gamma.
    \end{aligned}\right.
\end{equation} 
We compare the performance of FD-MGDL, FD-SGDL, Mscale and FBPINN for solving equation \eqref{3d-Helmholtz-Dirichlet} in the large-wavenumber regime.

\subsection{Highly Oscillatory Sine Solution}\label{3d-sin}
We evaluate the performance of the proposed method against several baselines using a test case of \eqref{3d-Helmholtz-Dirichlet} with a highly oscillatory sinusoidal solution; this facilitates a comparative analysis of the methods' efficacy in the high-wavenumber regime.

In this test case, we consider a homogeneous source term, $f(\mathbf{x}) = 0$, over the unit cubic domain $\Omega = (0,1)^3$. Dirichlet boundary conditions are prescribed on $\partial\Omega := \bigcup_{i=1}^6 \Gamma_i$ as follows:
\begin{equation*}
    g(x_1,x_2,x_3)=
    \begin{cases}
        0,  & \text{ if } (x_1,x_2,x_3)\in\Gamma_1:=\{0\}\times(0,1)\times(0,1); \\
        0,  & \text{ if } (x_1,x_2,x_3)\in\Gamma_2:=(0,1)\times\{0\}\times(0,1); \\
        0,  & \text{ if } (x_1,x_2,x_3)\in\Gamma_3:=(0,1)\times(0,1)\times\{0\}; \\
        \sin\left(\frac{\sqrt{3}}{3}\kappa\right)\sin\left(\frac{\sqrt{3}}{3}\kappa x_2\right)\sin\left(\frac{\sqrt{3}}{3}\kappa x_3\right),  & \text{ if } (x_1,x_2,x_3)\in\Gamma_4:=\{1\}\times(0,1)\times(0,1); \\
        \sin\left(\frac{\sqrt{3}}{3}\kappa x_1\right)\sin\left(\frac{\sqrt{3}}{3}\kappa\right)\sin\left(\frac{\sqrt{3}}{3}\kappa x_3\right),  & \text{ if } (x_1,x_2,x_3)\in\Gamma_5:=(0,1)\times\{1\}\times(0,1);\\
        \sin\left(\frac{\sqrt{3}}{3}\kappa x_1\right)\sin\left(\frac{\sqrt{3}}{3}\kappa x_2\right)\sin\left(\frac{\sqrt{3}}{3}\kappa\right),  & \text{ if } (x_1,x_2,x_3)\in\Gamma_6:=(0,1)\times(0,1)\times\{1\},
    \end{cases}
\end{equation*}
where $\kappa$ is a constant wavenumber. Under these conditions, the analytical solution to \eqref{3d-Helmholtz-Dirichlet} is:
\begin{equation}\label{3d-sin-solution}
    u(x_1,x_2,x_3) = \sin\left(\frac{\sqrt{3}}{3}\kappa x_1\right)\sin\left(\frac{\sqrt{3}}{3}\kappa x_2\right)\sin\left(\frac{\sqrt{3}}{3}\kappa x_3\right),\quad (x_1,x_2,x_3)\in [0,1]\times[0,1]\times[0,1],
\end{equation}
which exhibits increasingly rapid oscillations as $\kappa$ increases. This test case is specifically designed to ensure that the exact solution possesses a uniform degree of oscillation across all three coordinate directions, directly proportional to $\kappa$.

We solve the equation \eqref{3d-Helmholtz-Dirichlet} with the exact solution \eqref{3d-sin-solution} for wavenumbers $\kappa \in \{20, 30, 40, 50\}$ utilizing FD-MGDL, FD-SGDL, Mscale, and FBPINN. For all methods excluding Mscale, we set $m=60$ and discretize the unit interval $[0,1]$ into $m+1$ uniform subintervals. The resulting grid points are defined as $x_{j,i} := jh$ for $j \in \{0, \dots, m+1\}$ and $i \in \{1, 2, 3\}$, with a mesh size of $h := 1/(m+1)$. The training set is constructed using a tensor-product grid of these points: $\mathcal{X}_{train} = \{(x_{j_1,1}, x_{j_2,2}, x_{j_3,3}) : j_1, j_2, j_3 \in \mathbb{N}_m\}$. For the Mscale method, training points are randomly sampled from the computational domain $\Omega$; the total number of samples is kept consistent with the other methods by dividing the tensor-grid cardinality by the number of scales. For performance evaluation, testing points $\{(\tilde{x}_{j_1,1}, \tilde{x}_{j_2,2}, \tilde{x}_{j_3,3}) : j_1, j_2, j_3 \in \mathbb{N}_{\tilde{m}}\}$ are generated similarly using $\tilde{m}=30$ for all models.


\begin{table}[!htb]\centering
    \small
    \setlength{\tabcolsep}{15pt}
    \begin{threeparttable}
    \caption{Performance comparison of FD-MGDL, FD-SGDL, Mscale and FBPINN for the 3D Helmholtz problem \eqref{3d-Helmholtz-Dirichlet} with the exact solution \eqref{3d-sin-solution}.}
    \label{NumericalResult_for_3dPDE}
    \centering
    \begin{tabular}{clrrcc} 
        \toprule
         $\kappa$ & Method & Epoch & AC time (s)  & TrRSE & TeRSE \\
        \midrule
        \multirow{4}{*}{$20$} & FD-MGDL & 8,500 & {\bf 10,828} & $\mathbf{2.74\times10^{-2}}$ & $\mathbf{3.31\times10^{-2}}$ \\
        & FD-SGDL & 10,000 & 77,071 & $2.93\times10^{-2}$ & $3.55\times10^{-2}$ \\
        & Mscale & 10,000 & 20,117 & $1.89\times10^{-1}$ & $1.71\times10^{-1}$ \\
        & FBPINN  & 10,000 & 23,167 & $1.19\times10^{-1}$ & $1.19\times10^{-1}$ \\
        \midrule
        \multirow{4}{*}{$30$} & FD-MGDL & 6,000 & {\bf 8,038} & $\mathbf{1.69\times10^{-2}}$ & $\mathbf{1.81\times10^{-2}}$ \\
        & FD-SGDL & 8,000  & 33,722 & $1.77\times10^{-1}$ & $2.85\times10^{-1}$ \\
        & Mscale & 8,000 & 16,119 & $5.08\times10^{-1}$ & $4.13\times10^{-1}$ \\
        & FBPINN  & 8,000  & 21,090 & $2.82\times10^{-1}$ & $2.83\times10^{-1}$ \\
        \midrule
        \multirow{4}{*}{$40$} & FD-MGDL & 12,500 & {\bf 12,810} & $\mathbf{3.13\times10^{-1}}$ & $\mathbf{3.28\times10^{-1}}$ \\
        & FD-SGDL & 15,000 & 145,043 & $4.30\times10^{-1}$ & $4.58\times10^{-1}$ \\
        & Mscale & 15,000 & 29,865 & $5.58\times10^{-1}$ & $4.57\times10^{-1}$ \\
        & FBPINN  & 15,000 & 38,613 & $7.37\times10^{-1}$ & $7.37\times10^{-1}$ \\
        \midrule
        \multirow{4}{*}{$50$} & FD-MGDL & 6,500 & {\bf 8,646} & $\mathbf{3.48\times10^{-1}}$ & $\mathbf{3.91\times10^{-1}}$ \\
        & FD-SGDL & 8,000 & 64,206 & $9.85\times10^{-1}$ & $9.86\times10^{-1}$ \\
        & Mscale & 8,000 & 14,597 & $6.00\times10^{-1}$ & $5.46\times10^{-1}$ \\
        & FBPINN  & 8,000 & 24,292 & $9.25\times10^{-1}$ & $9.25\times10^{-1}$ \\
        \bottomrule
    \end{tabular}
    \end{threeparttable}
\end{table}

\begin{figure}[!htb]
\centering
\includegraphics[scale=0.83]{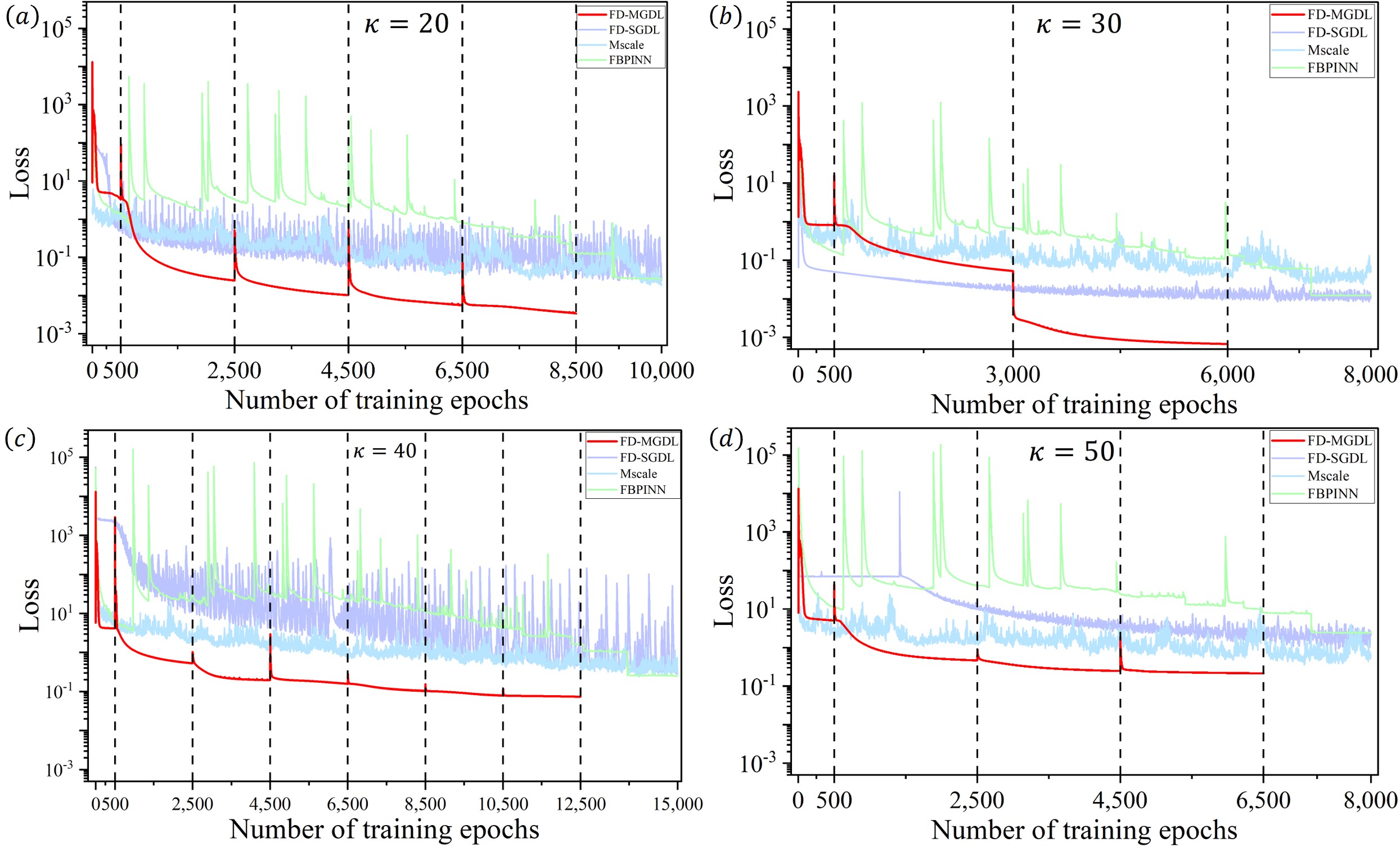}
\caption{Performance comparison of training loss curves of FD-MGDL and three baseline methods (FD-SGDL, Mscale and FBPINN) for the 3D Helmholtz problem \eqref{3d-Helmholtz-Dirichlet} with the exact solution \eqref{3d-sin-solution}: $(a)$ $\kappa=20$; $(b)$ $\kappa=30$; $(c)$ $\kappa=40$; $(d)$ $\kappa=50$.}
\label{3dsin-loss}
\end{figure}

Table \ref{NumericalResult_for_3dPDE} presents a comparative analysis of FD-MGDL against three baselines — FD-SGDL, Mscale, and FBPINN — for wavenumbers $\kappa \in \{20, 30, 40, 50\}$. Moving from two to three dimensions significantly increases computational complexity and exacerbates the difficulties inherent in capturing oscillatory wave propagation.

At $\kappa=20$, FD-MGDL achieves the superior performance, yielding the lowest training (TrRSE) and testing (TeRSE) errors. Notably, while FD-SGDL achieves comparable accuracy, its computational cost is nearly an order of magnitude higher, highlighting the poor scalability of single-grade networks in 3D. Meanwhile, Mscale and FBPINN exhibit markedly higher errors, struggling to resolve the 3D oscillatory structure.

As $\kappa$ increases to $30$ and $40$, the problem's numerical difficulty rises sharply, leading to a general deterioration in accuracy across all solvers. Nevertheless, FD-MGDL consistently maintains the lowest error levels with reduced computational overhead. For $\kappa=30$, FD-MGDL preserves a testing error on the order of $10^{-2}$, whereas baseline errors escalate to $10^{-1}$ or even $10^0$. This suggests that the multigrade strategy effectively mitigates optimization challenges in moderately high-frequency regimes.

At the highest wavenumber ($\kappa=50$), the severe oscillations cause substantial error growth in all methods. However, FD-MGDL remains the most robust, avoiding the near-complete loss of accuracy observed in FD-SGDL and FBPINN. Its relative performance advantage in this regime underscores the stability conferred by the multi-grade approach.

Convergence dynamics are illustrated via the training loss curves in Fig. \ref{3dsin-loss}. FD-MGDL displays faster and more stable decay across all wavenumbers. In contrast, FD-SGDL suffers from stagnation as $\kappa$ increases, while Mscale and FBPINN exhibit increasingly irregular training trajectories. These qualitative observations align closely with the quantitative results in Table \ref{NumericalResult_for_3dPDE}.

In summary, these 3D experiments demonstrate that FD-MGDL effectively scales to higher-dimensional Helmholtz problems. It offers a superior trade-off between accuracy and efficiency, delivering more robust solutions than existing SGDL, Multiscale, and PINN-based approaches in high-frequency settings.

\subsection{Plane Wave Propagation Problem}\label{3d-wave}
We extend the plane wave benchmark from subsection \ref{2d-wave} to the three-dimensional unit cube $\Omega = (0,1)^3$, assessing the numerical methods in higher-dimensional, high-wavenumber regimes.

For this test case, we consider the homogeneous Helmholtz equation ($f=0$). Dirichlet boundary conditions are prescribed on $\partial\Omega = \bigcup_{j=1}^6 \Gamma_j$, with the boundary data $g$ consistent with the exact solution:
\begin{equation}\label{3d-wave-solution}
u(x_1,x_2,x_3) = \exp\left[i(\kappa_1 x_1 + \kappa_2 x_2 + \kappa_3 x_3)\right], \quad \mathbf{x} \in \bar{\Omega}.
\end{equation}
The wave vector components are defined as $(\kappa_1, \kappa_2, \kappa_3) := \kappa(\cos\phi\cos\theta, \cos\phi\sin\theta, \sin\phi)$, where $\kappa$ is the wavenumber and $(\phi, \theta)$ are the angular parameters determining the propagation direction. Specifically, the boundary function $g$ is given by:
\begin{equation*}
    g(x_1,x_2,x_3)=
    \begin{cases}
        \mathrm{exp}\left[i(\kappa_2x_2+\kappa_3x_3)\right],  & \text{ if } (x_1,x_2,x_3)\in\Gamma_1:=\{0\}\times(0,1)^2; \\
        \mathrm{exp}\left[i(\kappa_1x_1+\kappa_3x_3)\right],  & \text{ if } (x_1,x_2,x_3)\in\Gamma_2:=(0,1)\times\{0\}\times(0,1); \\
        \mathrm{exp}\left[i(\kappa_1x_1+\kappa_2x_2)\right],  & \text{ if } (x_1,x_2,x_3)\in\Gamma_3:=(0,1)^2\times\{0\}; \\
        \mathrm{exp}\left[i(\kappa_1+\kappa_2x_2+\kappa_3x_3)\right],  & \text{ if } (x_1,x_2,x_3)\in\Gamma_4:=\{1\}\times(0,1)^2; \\
        \mathrm{exp}\left[i(\kappa_1x_1+\kappa_2+\kappa_3x_3)\right],  & \text{ if } (x_1,x_2,x_3)\in\Gamma_5:=(0,1)\times\{1\}\times(0,1);\\
        \mathrm{exp}\left[i(\kappa_1x_1+\kappa_2x_2+\kappa_3)\right],  & \text{ if } (x_1,x_2,x_3)\in\Gamma_6:=(0,1)^2\times\{1\}.
    \end{cases}
\end{equation*}

This setup represents a monochromatic plane wave whose oscillations intensify as $\kappa$ increases. Resolving such high-frequency wave fields in 3D remains a significant numerical challenge, particularly regarding the trade-off between discretization density and computational efficiency.

We solve the homogeneous Helmholtz equation to approximate the exact plane-wave solution \eqref{3d-wave-solution} for wavenumbers $\kappa \in \{20, 30\}$ with propagation angles $\phi=\pi/3$ and $\theta=\pi/8$. We compare the performance of FD-MGDL, FD-SGDL, Mscale, and FBPINN. Following the setup in the previous experiment, training points for most methods are defined on a tensor-product grid with $m=60$, while testing points are defined similarly with $\tilde{m}=30$. For the Mscale method, training points are randomly sampled from the domain $\Omega$, maintaining a total sample count equivalent to the other methods divided by the number of scales.


\begin{table}[!htb]\centering
    \small
    \setlength{\tabcolsep}{15pt}
    \begin{threeparttable}
    \caption{Performance comparison of FD-MGDL, FD-SGDL, Mscale and FBPINN for the 3D Helmholtz problem \eqref{3d-Helmholtz-Dirichlet} with the exact solution \eqref{3d-wave-solution}.}
    \label{NumericalResult_for_3dplanewave}
    \centering
    \begin{tabular}{clrrcc} 
        \toprule
         $\kappa$ & Method & Epoch & AC time (s)  & TrRSE & TeRSE \\
        \midrule
        \multirow{4}{*}{$20$} & FD-MGDL & 3,500 & {\bf 5,200} & $\mathbf{4.16\times10^{-2}}$ & $\mathbf{4.29\times10^{-2}}$ \\
        & FD-SGDL & 5,000 & 41,358 & $7.69\times10^{-1}$ & $7.79\times10^{-1}$ \\
        & Mscale & 5,000 & 9,142 & $3.18\times10^{-1}$ & $2.70\times10^{-1}$ \\
        & FBPINN  & 5,000 & 12,950 & $3.04\times10^{-1}$ & $3.04\times10^{-1}$ \\
        \midrule
        \multirow{4}{*}{$30$} & FD-MGDL & 3,500 & {\bf 4,998} & $\mathbf{2.51\times10^{-1}}$ & $\mathbf{2.59\times10^{-1}}$ \\
        & FD-SGDL & 5,000 & 40,978 & $9.19\times10^{-1}$ & $9.28\times10^{-1}$ \\
        & Mscale & 5,000 & 9,051 & $6.66\times10^{-1}$ & $5.91\times10^{-1}$ \\
        & FBPINN  & 5,000 & 12,881 & $7.28\times10^{-1}$ & $7.28\times10^{-1}$ \\
        \bottomrule
    \end{tabular}
    \end{threeparttable}
\end{table}

\begin{figure}[!htb]
\centering
\includegraphics[scale=0.83]{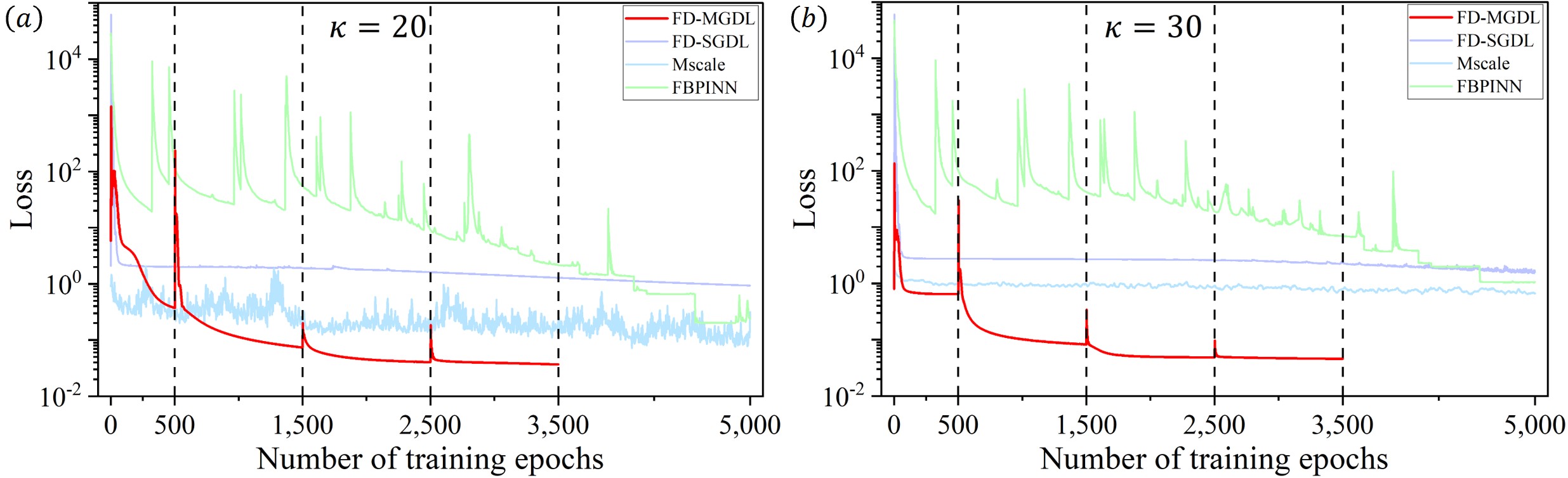}
\caption{Performance comparison of training loss curves of FD-MGDL and three baseline methods (FD-SGDL, Mscale and FBPINN) for the 3D Helmholtz problem \eqref{3d-Helmholtz-Dirichlet} with the exact solution \eqref{3d-wave-solution}: $(a)$ $\kappa=20$; $(b)$ $\kappa=30$.}
\label{3dwave-loss}
\end{figure}

Table \ref{NumericalResult_for_3dplanewave} presents a quantitative comparison of FD-MGDL with FD-SGDL, Mscale, and FBPINN for wavenumbers $\kappa=20$ and $30$. Compared with the three-dimensional sinusoidal solution, the plane wave configuration places stronger emphasis on accurate phase propagation and directional consistency, thereby posing additional challenges for neural-network-based solvers.

For $\kappa=20$, FD-MGDL achieves the highest accuracy among all tested methods, with TrRSE and TeRSE values on the order of $10^{-2}$, while requiring substantially fewer training epochs and significantly lower accumulated computational time. In contrast, FD-SGDL exhibits severe accuracy degradation, with relative errors approaching unity despite its markedly higher computational cost. Mscale and FBPINN improve upon FD-SGDL in terms of efficiency, yet their error levels remain an order of magnitude larger than those of FD-MGDL, indicating insufficient phase accuracy even at this moderate wavenumber.

As the wavenumber increases to $\kappa=30$, the limitations of the baseline methods become even more pronounced. Table \ref{NumericalResult_for_3dplanewave} shows that FD-SGDL, Mscale, and FBPINN all suffer from substantial error growth, with TrRSE and TeRSE values ranging from $\mathcal{O}(10^{-1})$ to $\mathcal{O}(10^{0})$. In contrast, FD-MGDL maintains a clear accuracy advantage, yielding the smallest errors and converging with only 4,998 seconds. Although the absolute error level increases compared with the lower-wavenumber case, FD-MGDL remains the most reliable solver among the tested approaches.

The convergence characteristics are further illustrated in Fig. \ref{3dwave-loss}, which displays the training loss curves for $\kappa=20$ and $30$. As shown in Fig. \ref{3dwave-loss} $(a)$–$(b)$, FD-MGDL consistently demonstrates faster and smoother loss decay, whereas the baseline methods exhibit slow convergence and early stagnation. These trends are consistent with the quantitative results in Table \ref{NumericalResult_for_3dplanewave} and highlight the effectiveness of the multigrade learning strategy in alleviating optimization difficulties associated with three-dimensional oscillatory wave propagation.

In summary, the three-dimensional plane wave experiments confirm that FD-MGDL extends robustly to high-dimensional, phase-sensitive Helmholtz problems. By achieving a favorable balance between accuracy and computational efficiency, FD-MGDL substantially outperforms existing SGDL, Multiscale and PINN-based solvers, further demonstrating its potential for practical three-dimensional wave simulations.

\section{Comparison with Finite Difference Method}\label{FDM}

In this section, we evaluate FD-MGDL against the traditional Finite Difference Method (FDM), a widely established numerical baseline for Helmholtz equations on structured domains \cite{chen2012dispersion,wu2017dispersion,wu2021new}.

\subsection{Methodology and Interpolation}
The FDM solution is computed on the structured grid defined by the training collocation nodes:
\begin{equation*}
    \mathcal{X}_{\mathrm{train}} = \left\{(x_{j_1,1},\dots,x_{j_d,d}) : j_i \in \mathbb{N}_m, \, i \in \mathbb{N}_d\right\}.
\end{equation*}
We employ a standard second-order central difference stencil to discretize the spatial Laplacian operator, solving the resulting global sparse linear system to obtain nodal solution values.

Because the off-grid testing points
\begin{equation*}
    \mathcal{X}_{\mathrm{test}} = \left\{(\tilde{x}_{j_1,1},\dots,\tilde{x}_{j_d,d}) : j_i \in \mathbb{N}_{\tilde{m}}, \, i \in \mathbb{N}_d\right\}
\end{equation*}
do not generally coincide with the discrete grid nodes, off-grid values must be estimated via spatial interpolation. To rigorously evaluate FDM performance under varying post-processing accuracies, we implement both multilinear and multiquadratic interpolation schemes.

Consider an arbitrary testing point $\tilde{\mathbf{x}}=(\tilde{x}_{1},\tilde{x}_{2},\dots,\tilde{x}_d)$ located within the grid cell
\begin{equation*}
\Omega_\mathbf{j} = \prod_{i=1}^{d} [x_{j_i,i}, \, x_{j_i+1,i}], \quad \mathbf{j}=(j_i : i\in\mathbb{N}_d).
\end{equation*}
We define local normalized coordinates $\xi_i = (\tilde{x}_{i} - x_{j_i,i})/h$ for $i \in \mathbb{N}_d$, such that $\boldsymbol{\xi} = (\xi_1,\xi_2,\dots,\xi_d) \in [0,1]^d$. The interpolation schemes are formulated as follows:

\begin{enumerate}
    \item {\bf Multilinear Interpolation:} The reconstructed value $\tilde{u}_{\mathrm{test}}$ at $\tilde{\mathbf{x}}$ is evaluated as a multilinear combination of the $2^d$ cell vertex values:
    \begin{equation*}
        \tilde{u}_{\mathrm{test}} = \sum_{\boldsymbol{\alpha}\in\{0,1\}^d} \left( \prod_{i=1}^{d} \left[ \alpha_i \xi_i + (1-\alpha_i)(1-\xi_i) \right] \right) \tilde{u}_{\mathbf{j}+\boldsymbol{\alpha}},
    \end{equation*}
    where $\tilde{u}_{\mathbf{j}+\boldsymbol{\alpha}}$ denotes the nodal FDM solution at $(x_{j_1+\alpha_1,1},\dots,x_{j_d+\alpha_d,d})$.

    \item {\bf Multiquadratic Interpolation:} To achieve higher-order post-processing reconstruction, $\tilde{u}_{\mathrm{test}}$ is evaluated via the tensor product of one-dimensional quadratic Lagrange polynomials over $3^d$ local stencil nodes:
    \begin{equation*}
        \tilde{u}_{\mathrm{test}} = \sum_{\mathbf{p}\in\{0,1,2\}^d} \left( \prod_{i=1}^{d} \mathcal{L}_{p_i}(\xi_i) \right) \tilde{u}_{\mathbf{j}+\mathbf{p}},
    \end{equation*}
    where $\mathbf{p}=(p_1,\dots,p_{d})$ indexes the local node topology. The 1D basis functions $\mathcal{L}_k(\xi)$ defined on reference coordinates $\{\xi_0=0, \, \xi_1=1/2, \, \xi_2=1\}$ take the form:
    \begin{equation*}
        \mathcal{L}_k(\xi) = \prod_{\substack{r=0 \\ r\neq k}}^2 \frac{\xi-\xi_r}{\xi_k-\xi_r}. 
    \end{equation*}
\end{enumerate}

\subsection{Numerical Results and Analysis}

Tables \ref{FDM_for_2dPDE}--\ref{FDM_for_3dwave} summarize the comparative performance between classical FDM and FD-MGDL across the 2D and 3D benchmark problems. For each configuration, we report the training relative structural error ($\mathrm{TrRSE}$) evaluated strictly at grid nodes alongside the off-grid testing error ($\mathrm{TeRSE}$) reconstructed via bilinear/trilinear and bi-/triquadratic interpolations.

While FDM completes linear system solves rapidly on fixed low-resolution grids, the results highlight critical numerical limitations of traditional discrete solvers when handling high-frequency wave fields. The key advantages of FD-MGDL are detailed below:

\begin{enumerate}
    \item {\bf Solution Fidelity vs. Offline Computational Investment:}
    While FDM computes direct linear solves in seconds, its computational efficiency is compromised by severe accuracy degradation as the wavenumber increases. For example, in Table \ref{FDM_for_2dPDE} at $\kappa=100$, FDM completes in $0.87$~s but yields a physically corrupted prediction with a $27.1\%$ testing error ($\mathrm{TeRSE} = 2.71\times 10^{-1}$). In contrast, FD-MGDL achieves high-precision resolution with an error of $0.55\%$ ($\mathrm{TeRSE} = 5.56\times 10^{-3}$). A fast computational solve offers little practical value when the solution is dominated by numerical dispersion; FD-MGDL purposefully invests offline training time to deliver physically accurate predictions where classical discretizations fail.

    \item {\bf Robustness to Resonant Ill-Conditioning:}
    The 2D homogeneous Dirichlet problem admits resonant wavenumbers $\kappa_{m,n} := \pi\sqrt{m^2+n^2}$ ($m,n \in \mathbb{Z}^+$). As $\kappa$ approaches a resonant frequency, the global discrete Helmholtz matrix $A$ in FDM becomes severely ill-conditioned ($\kappa(A) \to \infty$), causing dramatic solution breakdown. In Table \ref{FDM_for_2dPDE}, $\kappa=100 \approx \kappa_{22,23}$ and $\kappa=200 \approx \kappa_{45,45}$ represent near-resonant states where FDM error surges to $26.9\%$ and $68.8\%$, respectively. Conversely, FD-MGDL avoids global matrix inversions entirely, optimizing network parameters through local grade-wise residual minimization. Consequently, FD-MGDL suppresses resonant ill-conditioning, achieving up to a \textbf{50-fold error reduction} at $\kappa=100$ and demonstrating robust stability across resonant frequencies.

    \item {\bf Continuous Functional Representation vs. Two-Stage Interpolation Errors:}
    FDM outputs discrete pointwise values restricted to grid nodes, relying on spatial interpolation to estimate off-grid locations. As demonstrated in Table \ref{FDM_for_2dPDE} at $\kappa=50$, even when FDM attains a small nodal error ($\mathrm{TrRSE} = 5.48\times 10^{-4}$), off-grid interpolation degrades the testing error by over an order of magnitude ($\mathrm{TeRSE} = 8.84\times 10^{-3}$), with higher-order quadratic schemes yielding negligible relief. In contrast, FD-MGDL parametrizes a mesh-free, continuously differentiable solution manifold $u_\theta(\mathbf{x})$. Off-grid evaluation requires no post-processing interpolation, ensuring that testing errors strictly align with training errors ($\mathrm{TrRSE} \approx \mathrm{TeRSE}$) across all test cases.

\end{enumerate}

\begin{table}[!htb]\centering
    \small
    \setlength{\tabcolsep}{15pt}
    \begin{threeparttable}
    \caption{Performance comparison of FDM and FD-MGDL for the 2D Helmholtz problem \eqref{2d-Helmholtz-Dirichlet} with the exact solution \eqref{2d-sin-solution}.}
    \label{FDM_for_2dPDE}
    \centering
    \begin{tabular}{clcccc} 
        \toprule
         $\kappa$  &  Method & Time (s) & TrRSE &  \multicolumn{2}{c}{TeRSE} \\
          & &  &  & Bilinear & Biquadratic\\
        \midrule
        \multirow{2}{*}{$50$} 
        & FDM & 0.22 & $5.48\times 10^{-4}$ & $8.84\times 10^{-3}$ &$8.80\times 10^{-3}$ \\
        & FD-MGDL & 6,272 & $\mathbf{5.37\times 10^{-4}}$ & \multicolumn{2}{c}{$\mathbf{5.37\times 10^{-4}}$} \\
        \midrule
        \multirow{2}{*}{$100$} 
        & FDM & 0.87 & $2.69\times 10^{-1}$ & $2.71\times 10^{-1}$ &$2.71\times 10^{-1}$ \\
        & FD-MGDL & 9,196 & $\mathbf{5.54\times 10^{-3}}$ & \multicolumn{2}{c}{$\mathbf{5.56\times 10^{-3}}$} \\
        \midrule
        \multirow{2}{*}{$150$} 
        & FDM & 2.27 & $2.03\times 10^{-2}$ & $2.54\times 10^{-2}$ &$2.63\times 10^{-2}$ \\
        & FD-MGDL & 21,861 & $\mathbf{1.31\times 10^{-2}}$ & \multicolumn{2}{c}{$\mathbf{1.34\times 10^{-2}}$} \\
        \midrule
        \multirow{2}{*}{$200$} 
        & FDM & 2.26 & $6.88\times 10^{-1}$ & $6.88\times 10^{-1}$ & $6.86\times 10^{-1}$ \\
        & FD-MGDL & 17,776 & $\mathbf{3.92\times 10^{-1}}$ & \multicolumn{2}{c}{$\mathbf{3.92\times 10^{-1}}$} \\
        \bottomrule
    \end{tabular}
    \end{threeparttable}
\end{table}

\begin{table}[!htb]\centering
    \small
    \setlength{\tabcolsep}{15pt}
    \begin{threeparttable}
    \caption{Performance comparison of FDM and FD-MGDL for the 2D Helmholtz problem \eqref{2d-Helmholtz-Dirichlet} with the exact solution \eqref{2d-wave-solution}.}
    \label{FDM_for_2dWave}
    \centering
    \begin{tabular}{clcccc} 
        \toprule
         $\kappa$  &  Method & Time (s) & TrRSE &  \multicolumn{2}{c}{TeRSE} \\
         & &  &  & Bilinear & Biquadratic\\
        \midrule
        \multirow{2}{*}{$50$} 
        & FDM & 0.86 & $4.25\times 10^{-4}$ & $4.24\times 10^{-4}$ &$4.22\times 10^{-4}$ \\
        & FD-MGDL & 2,540 & $\mathbf{3.40\times 10^{-4}}$ & \multicolumn{2}{c}{$\mathbf{3.42\times 10^{-4}}$} \\
        \midrule
        \multirow{2}{*}{$100$} 
        & FDM & 2.86 & $2.61\times 10^{-1}$ & $2.60\times 10^{-1}$ &$2.60\times 10^{-1}$ \\
        & FD-MGDL & 9,493 & $\mathbf{2.82\times 10^{-3}}$ & \multicolumn{2}{c}{$\mathbf{2.83\times 10^{-3}}$} \\
        \midrule
        \multirow{2}{*}{$150$} 
        & FDM & 6.89 & $1.26\times 10^{-2}$ & $1.81\times 10^{-2}$ &$1.83\times 10^{-2}$ \\
        & FD-MGDL & 6,894 & $\mathbf{8.07\times 10^{-3}}$ & \multicolumn{2}{c}{$\mathbf{8.07\times 10^{-3}}$} \\
        \midrule
        \multirow{2}{*}{$200$} 
        & FDM & 6.79 & $4.47\times 10^{-1}$ & $4.49\times 10^{-1}$ & $4.52\times 10^{-1}$ \\
        & FD-MGDL & 10,462 & $\mathbf{5.20\times 10^{-2}}$ & \multicolumn{2}{c}{$\mathbf{5.21\times 10^{-2}}$} \\
        \bottomrule
    \end{tabular}
    \end{threeparttable}
\end{table}

\begin{table}[!htb]\centering
    \small
    \setlength{\tabcolsep}{15pt}
    \begin{threeparttable}
    \caption{Performance comparison of FDM and FD-MGDL for the 3D Helmholtz problem \eqref{3d-Helmholtz-Dirichlet} with the exact solution \eqref{3d-sin-solution}.}
    \label{FDM_for_3dsin}
    \centering
    \begin{tabular}{clcccc} 
        \toprule
         $\kappa$ & Method & Time (s)  & TrRSE &  \multicolumn{2}{c}{TeRSE} \\
         & &  &  & Trilinear & Triquadratic\\
        \midrule
        \multirow{2}{*}{$20$} 
        & FDM & 236.60 & $3.18\times 10^{-2}$ & $7.77\times 10^{-2}$ &$8.05\times 10^{-2}$ \\
        & FD-MGDL & 10,828 & $\mathbf{2.74\times 10^{-2}}$ & \multicolumn{2}{c}{$\mathbf{3.31\times 10^{-2}}$} \\
        \midrule
        \multirow{2}{*}{$30$} 
        & FDM & 235.58 & $8.17\times 10^{-3}$ & $8.15\times 10^{-2}$ &$8.30\times 10^{-2}$ \\
        & FD-MGDL & 8,038 & $\mathbf{1.69\times 10^{-2}}$ & \multicolumn{2}{c}{$\mathbf{1.81\times 10^{-2}}$} \\
        \midrule
        \multirow{2}{*}{$40$} 
        & FDM & 234.49 & $5.85\times 10^{-1}$ & $5.88\times 10^{-1}$ &$6.08\times 10^{-1}$ \\
        & FD-MGDL & 12,810 & $\mathbf{3.13\times 10^{-1}}$ & \multicolumn{2}{c}{$\mathbf{3.28\times 10^{-1}}$} \\
        \midrule
        \multirow{2}{*}{$50$} 
        & FDM & 237.44 & $3.79\times 10^{-1}$ & $4.83\times 10^{-1}$ & $5.26\times 10^{-1}$ \\
        & FD-MGDL & 8,646 & $\mathbf{3.48\times 10^{-1}}$ & \multicolumn{2}{c}{$\mathbf{3.91\times 10^{-1}}$} \\
        \bottomrule
    \end{tabular}
    \end{threeparttable}
\end{table}

\begin{table}[!htb]\centering
    \small
    \setlength{\tabcolsep}{15pt}
    \begin{threeparttable}
    \caption{Performance comparison of FDM and FD-MGDL for the 3D Helmholtz problem \eqref{3d-Helmholtz-Dirichlet} with the exact solution \eqref{3d-wave-solution}.}
    \label{FDM_for_3dwave}
    \centering
    \begin{tabular}{clcccc} 
        \toprule
         $\kappa$  &  Method & Time (s) & TrRSE &  \multicolumn{2}{c}{TeRSE} \\
         & &  &  & Trilinear & Triquadratic\\
        \midrule
        \multirow{2}{*}{$20$} 
        & FDM & 466.54 & $ 6.10\times10^{-2} $ & $ 1.01\times10^{-1} $ &$ 1.03\times10^{-1} $ \\
        & FD-MGDL & 5,200 & $\mathbf{4.16\times 10^{-2}}$ & \multicolumn{2}{c}{$\mathbf{4.29\times 10^{-2}}$} \\
        \midrule
        \multirow{2}{*}{$30$} 
        & FDM & 471.72 & $ 3.45\times10^{-1} $ & $ 4.02\times10^{-1} $ &$ 4.23\times10^{-1} $ \\
        & FD-MGDL & 4,998 & $\mathbf{2.51\times 10^{-1}}$ & \multicolumn{2}{c}{$\mathbf{2.59\times 10^{-1}}$} \\
        \bottomrule
    \end{tabular}
    \end{threeparttable}
\end{table}

\subsection{The Role of the Grid in FD-MGDL}
It is essential to clarify that in FD-MGDL, the spatial grid discretizes the \emph{loss function}, rather than constraining the \emph{solution space}:
\begin{itemize}
\item {\bf Training Phase:} The grid functions as a structured collocation sampler to construct discrete PDE residuals. Finite difference stencils efficiently approximate the Laplacian operator, bypassing the computational overhead and higher-order derivative vanishing associated with automatic differentiation.
\item {\bf Inference Phase:} Once trained, the solution is represented as a continuous function $s_L^*(\mathbf{x})$ parametrized by the multi-grade neural architecture. Forward evaluation predicts values at arbitrary spatial coordinates instantaneously without requiring spatial grids or post-processing interpolation.
\end{itemize}

While structured rectangular domains provide a standardized computational testbed for establishing algorithmic efficiency, the FD-MGDL framework is readily extensible to irregular geometries via two primary methodologies:
\begin{enumerate}
    \item \textbf{Coordinate Transformation and Curvilinear Grids:} Complex physical domains can be mapped onto standard rectangular computational domains via smooth coordinate transformations. Applying the chain rule to the PDE operators allows high-order structured stencils to be evaluated directly on the uniform computational grid.
    \item \textbf{Generalized Finite Difference (GFD) Methods:} For arbitrary or highly complex geometries where global grid mappings are difficult to construct, the framework can adopt unstructured point clouds. GFD schemes construct local Taylor expansions over neighboring point clusters, generating discrete derivative weights on arbitrary domains without requiring a structured grid mesh.
\end{enumerate}

\section{Higher-Order Accuracy Discretizations of Second-Order Derivatives}\label{FD_extensions}

A key limitation of the FD-MGDL implementation discussed thus far lies in the operator discretization error: the overall solution accuracy remains bounded by the second-order truncation error of the discrete differential stencils. To fully leverage the expressive capacity of the multi-grade neural architecture, this section incorporates finite difference schemes with higher-order accuracy to approximate second partial derivatives with significantly enhanced spatial precision.

Specifically, the second derivative can be discretized with fourth-order accuracy using the five-point central stencil:
\begin{equation*}
    u''(x) = \frac{1}{12h^2}\left[-u(x-2h) + 16u(x-h) - 30u(x) + 16u(x+h) - u(x+2h)\right] + \mathcal{O}(h^4).
\end{equation*}
We replace the baseline second-order stencil with this fourth-order accuracy formulation for the high-wavenumber test case ($\kappa=100$) in Section~\ref{2d-sin}, keeping all network architectures and training hyperparameters fixed.

\begin{table}[!htb]
    \centering
    \small
    \setlength{\tabcolsep}{14pt}
    \begin{threeparttable}
    \caption{Comparison for  FD loss formulations with second- and fourth-order accuracy ($\kappa=100$).}
    \label{TeRSE_FD_Order}
    \centering
    \begin{tabular}{l|cc|cc} 
        \toprule
        \multirow{2}{*}{Method} & \multicolumn{2}{c|}{2nd-Order Accuracy FD Loss} & \multicolumn{2}{c}{4th-Order Accuracy FD Loss} \\
        \cmidrule(lr){2-3} \cmidrule(lr){4-5}
         & TeRSE & AC time (s) & TeRSE & AC time (s)\\
        \midrule
        FD-MGDL  & $5.56\times10^{-3}$ & 9,196 & $\mathbf{3.20\times10^{-4}}$ & 9,217 \\
        FD-SGDL  & $1.87\times10^{-2}$ & 47,792 & $5.37\times10^{-3}$ & 47,992 \\
        \bottomrule
    \end{tabular}
    \end{threeparttable}
\end{table}

As reported in Table~\ref{TeRSE_FD_Order}, adopting the fourth-order accuracy formulation substantially improves solution precision, reducing the testing error ($\mathrm{TeRSE}$) of FD-MGDL from $5.56\times 10^{-3}$ to $3.20\times 10^{-4}$, a $17.4$-fold gain. Although FD-SGDL also benefits from the higher-order accuracy formulation (yielding a $3.5$-fold error reduction), the improvement is significantly more pronounced within the multi-grade framework. Crucially, for both FD-MGDL and FD-SGDL, this heightened accuracy is achieved with virtually no additional computational cost.

It is worth noting the distinct practical feasibility of incorporating higher-order accuracy discretizations in MGDL compared to traditional finite difference solvers. In classical FD methods, extending stencils to a higher order accuracy introduces new implementation and computational challenges: wider stencils expand matrix bandwidth, substantially inflating linear system complexity, and complicate boundary treatments, often requiring specialized ghost nodes or lower-order boundary stencils that impair global accuracy. In contrast, higher-order accuracy stencils are seamlessly integrated into MGDL. Because MGDL optimizes residual loss functions rather than performing global matrix inversions, wider stencils do not inflate solver computational graph complexity. Furthermore, boundary conditions in MGDL are enforced via decoupled loss penalties rather than structural stencil modifications, bypassing the boundary-induced instabilities characteristic of traditional higher-order accuracy FD schemes.

These findings confirm that higher-order accuracy derivative approximations can be seamlessly integrated into the loss formulation to achieve superior precision in high-frequency wave modeling, while further highlighting the rich expressive capacity and flexibility of the MGDL framework. A comprehensive and systematic investigation into integrating MGDL with a broader spectrum of higher-order accuracy derivative approximations is deferred to future work, where we will thoroughly evaluate the trade-offs between operator fidelity and computational efficiency.


\section{Concave Model}\label{concave}

In this section, we apply FD-MGDL to the concave velocity model governed by the two-dimensional Helmholtz equation \cite{ren2009seismic}, which involves variable wavenumbers in $\mathbb{R}^2$. Testing a concave velocity profile is scientifically significant because it introduces caustics and multipathing effects, where wave energy is focused into localized regions. This creates a high-stress scenario for numerical solvers, as they must accurately capture sharp phase transitions and amplitude spikes without introducing artificial oscillations.

Since analytical solutions for such heterogeneous media are generally unavailable, numerical methods like finite difference and finite element methods are essential for validating the robustness of the solver. To truncate the unbounded domain while preventing non-physical reflections from contaminating the focused wavefield, absorbing boundary conditions are imposed; we adopt the perfectly matched layer (PML) introduced by B\'{e}renger \cite{berenger1994perfectly}, leading to the Helmholtz equation with PML.

For this benchmark, our FD-MGDL method is built on the classical 5-point finite difference scheme, which is known to suffer from significant numerical dispersion. Since the exact solution is unavailable for this concave velocity model, we use the optimal 9-point scheme of Chen et al. \cite{chen2013optimal} as a high-fidelity reference. This 9-point operator is designed to reduce dispersion and improve accuracy, serving as a reliable surrogate for the exact wavefield.

We compare the wavefields produced by the standard 5-point scheme, FD-MGDL, and the optimal 9-point reference to assess how effectively FD-MGDL reduces the dispersion errors of the 5-point baseline.

Specifically, we consider time-harmonic wave propagation in a heterogeneous medium governed by the 2D Helmholtz equation. For $(x,y) \in \mathbb{R}^2$ with spatially varying velocity $v(x,y)$ and source term $g(x,y)$, the wavefield $u(x,y)$ satisfies:
\begin{equation*}(\Delta + \kappa^2)u(x,y) = g(x,y),
\end{equation*}
where $\kappa(x,y) = \omega/v(x,y)$ is the spatially dependent wavenumber and $\omega$ is the angular frequency.

The concave velocity model, illustrated in Figure \ref{velocity-structure}, is defined on the square domain
\begin{equation*}
    \Omega := [0, 2000]\text{ m } \times [0, 2000]\text{ m}.
\end{equation*}
The model features three piecewise constant velocity layers: $1500\text{ m/s}$, $2000\text{ m/s}$, and $2500\text{ m/s}$, ordered from top to bottom. A point source is positioned at $(x_s, y_s) = (1000, 800)$, characterized by
\begin{equation*}
    g(x,y) := \delta(x-x_s, y-y_s)\widehat{R}(f, f_0),
\end{equation*}
where $\widehat{R}(f, f_0)$ is the frequency-domain Ricker wavelet with a dominant frequency $f_0 = 25\text{ Hz}$:
\begin{equation*}
    \widehat{R}(f,f_0) = \int_{-\infty}^{+\infty}(1-2\pi^2f_0^2t^2)\exp({-\pi^2f_0^2t^2})\exp({-i2\pi ft})\mathrm{d}t.
\end{equation*}
To determine the wavefield for this configuration, we solve the following frequency-domain equation:
\begin{equation*}\label{Concave_Model}
    \Delta u + \kappa^2 u = \delta(x-x_s, y-y_s)\widehat{R}(f, f_0), \quad \text{in } \mathbb{R}^2.
\end{equation*}

\begin{figure}[!htb]
\centering
\includegraphics[scale=0.35]{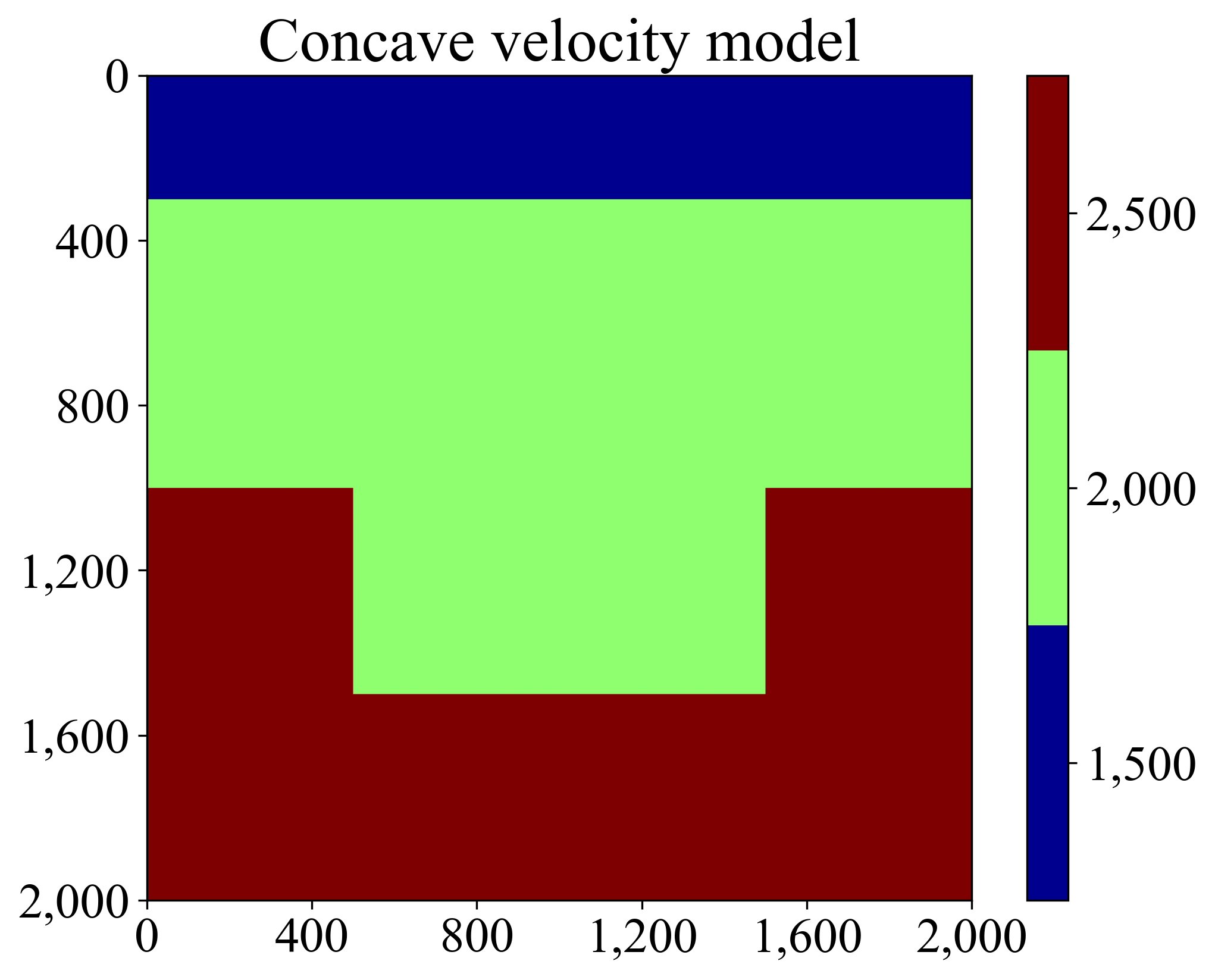}
\caption{Three-layer concave velocity structure.}
\label{velocity-structure}
\end{figure}

To truncate the computational domain and suppress artificial reflections, we apply the PML technique. The resulting Helmholtz equation with PML is given by:
\begin{equation*}\label{Helm_pml}
    \frac{\partial}{\partial x}\left(A\frac{\partial u}{\partial x}\right)+\frac{\partial}{\partial y}\left(B\frac{\partial u}{\partial y}\right)+C\kappa^2u = \delta(x-x_s,y-y_s)\widehat{R}(f,f_0),
\end{equation*}
where the complex-valued coefficients $A$, $B$, and $C$ are defined in terms of the stretching functions $e_x$ and $e_y$:
\begin{equation*}
    A := \frac{e_y}{e_x}, \quad B := \frac{e_x}{e_y}, \quad C := e_x e_y.
\end{equation*}
The stretching functions $e_x$ and $e_y$ account for wave attenuation and are defined as:
\begin{equation*}
    e_x := 1 + i\frac{\sigma_x}{\omega}, \quad e_y := 1 + i\frac{\sigma_y}{\omega},
\end{equation*}
where $\omega := 2\pi f$ is the angular frequency. To minimize numerical reflections at the interface, the damping profiles $\sigma_x$ and $\sigma_y$ are chosen as differentiable functions. Specifically, for the $x$-direction:
\begin{equation*}
    \sigma_x := 
    \begin{cases}
        2\pi a_0 f_0 \left( \frac{l_x}{L_{PML}} \right)^2, & \text{inside the PML},\\
        0, & \text{outside the PML},
    \end{cases}
\end{equation*}
where $f_0$ is the dominant source frequency, $L_{PML}$ is the layer thickness, and $l_x$ represents the distance from the point $(x,y)$ to the interior-PML interface. The scaling parameter $a_0$ is set to $1.79$ \cite{zeng2001application}. The function $\sigma_y$ is defined analogously for the $y$-direction.

\begin{figure}[!htb]
\centering
\includegraphics[scale=0.9]{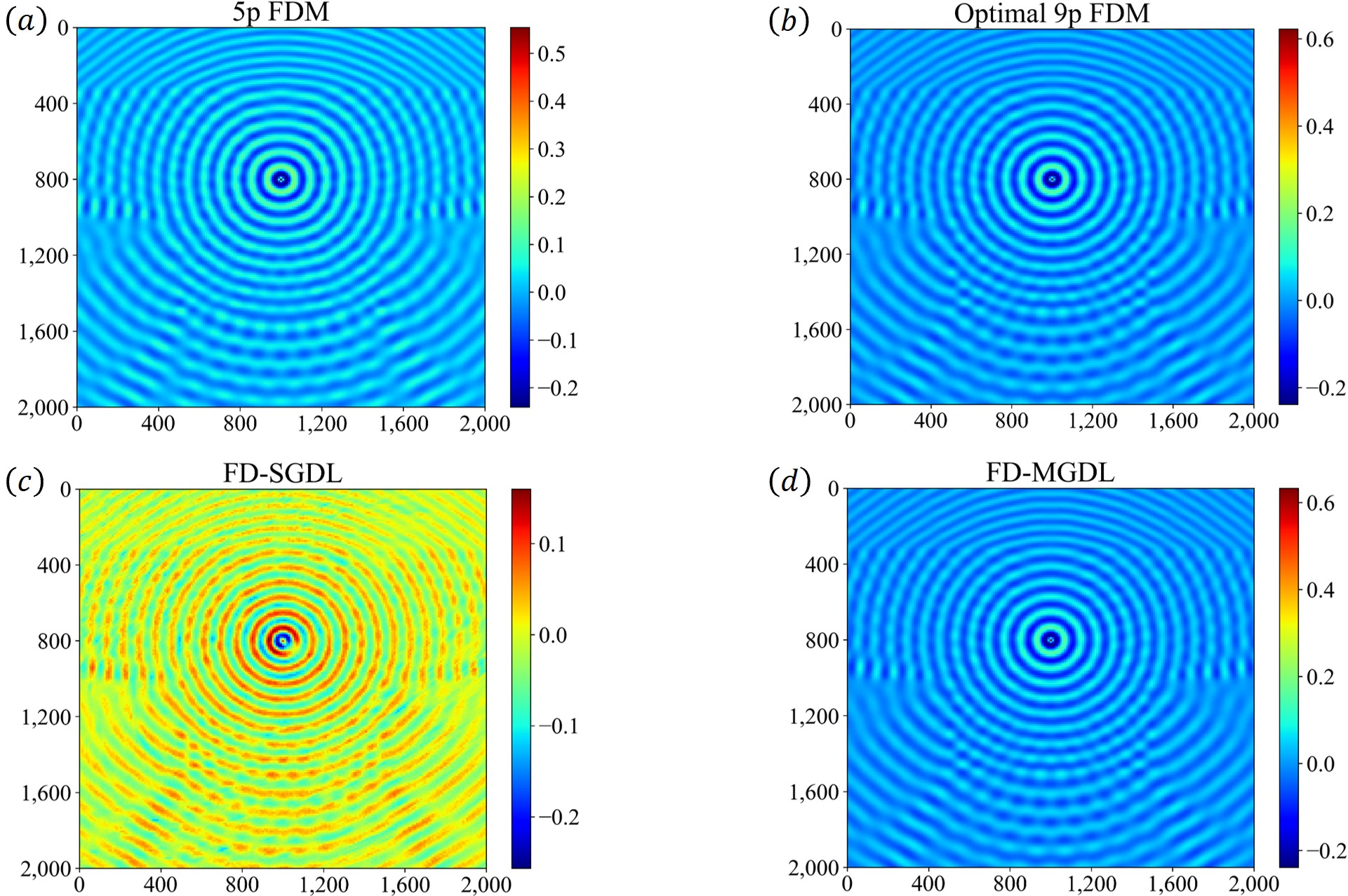}
\caption{Wavefield comparison for the concave velocity model. \textnormal{Reference wavefields computed using the 5-point $(a)$ and optimal 9-point $(b)$ finite-difference schemes, and wavefields predicted by FD-SGDL $(c)$ and FD-MGDL $(d)$.}}
\label{concave-visualization}
\end{figure}

\begin{figure}[!htb]
\centering
\includegraphics[scale=0.34]{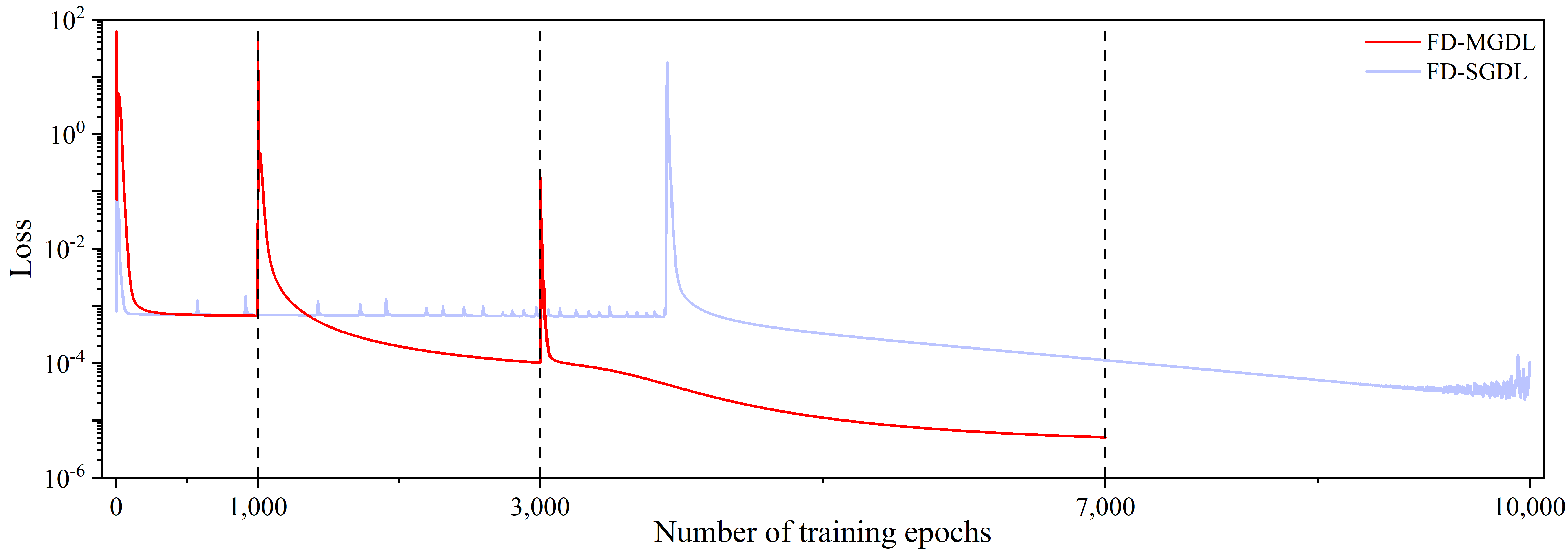}
\caption{Training loss curves of FD-SGDL and FD-MGDL for the concave velocity model.}
\label{concave-loss}
\end{figure}

\begin{figure}[!htb]
\centering
\includegraphics[scale=0.83]{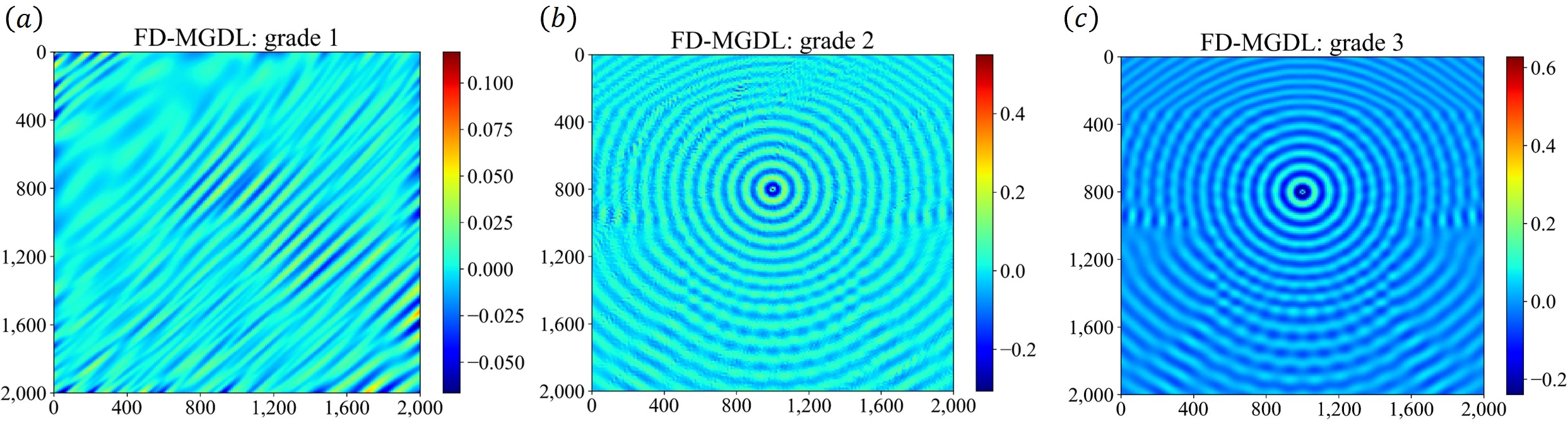}
\caption{Wavefields predicted by FD-MGDL for the concave velocity model at successive grades: grade 1 $(a)$, grade 2 $(b)$, and grade 3 $(c)$.}
\label{concave-MGDL}
\end{figure}

We evaluate the FD-MGDL-predicted wavefield  by comparing it with solutions computed using two finite-difference frequency-domain (FDFD) methods: the classical 5-point scheme and the optimal 9-point scheme of \cite{chen2013optimal}.

On a $201\times201$ grid, the real part of the monofrequency wavefield at $f = 25\ \mathrm{Hz}$ obtained with the 5-point and optimal 9-point schemes is shown in Figures \ref{concave-visualization} $(a)$ and $(b)$, respectively.
Figures \ref{concave-visualization} $(c)$ and $(d)$ display the wavefields produced by FD-SGDL and FD-MGDL. Figure \ref{concave-loss} shows the training loss versus epochs, confirming the convergence of FD-MGDL. Figures \ref{concave-MGDL} (a)-(c) present the FD-MGDL predictions at grades 1–3, illustrating the progressive refinement across grades.

Because the optimal 9-point scheme achieves high accuracy and effectively reduces numerical dispersion, especially at large wavenumbers, we use it as the reference solution. It clearly captures upward and downward incident waves, transmitted waves, and reflections within the middle velocity layer consistent with Snell’s law. As seen in Figure \ref{concave-visualization}, FD-MGDL closely matches the 9-point reference, whereas the 5-point scheme and FD-SGDL show noticeable discrepancies. These results demonstrate the improved accuracy of FD-MGDL.

We also compare computational efficiency. The total training time of FD-MGDL is 2,399 s, significantly shorter than FD-SGDL (10,150 s). This substantial reduction, together with improved accuracy, highlights the effectiveness and efficiency of FD-MGDL for wavefield simulation.

We briefly note the potential for singularities at the ``brown corners" in Figure \ref{velocity-structure}, where the three constant-velocity layers intersect. For the Helmholtz equation with piecewise constant coefficients, the solution must maintain continuity of both the wavefield and its normal flux across all interfaces. Consequently, despite the velocity discontinuity at these junctions, the wavefield remains continuous and bounded. Furthermore, because the source is a smooth Ricker wavelet located away from these corners, no additional singularities are introduced.

In summary, for the concave velocity benchmark, FD-MGDL significantly enhances wavefield accuracy compared with the 5-point scheme and FD-SGDL, closely matching the 9-point reference while effectively mitigating numerical dispersion. The grade-wise results show progressive refinement, and the method achieves higher accuracy with significantly shorter training time, demonstrating both effectiveness and efficiency for wavefield simulation in heterogeneous media.

\section{Conclusion}\label{conclusion}

This paper introduced Multi-Grade Deep Learning leveraged by Finite Difference (FD-MGDL), a novel adaptive framework designed to overcome the persistent challenges of solving high-frequency Helmholtz equations. By synergizing the well-conditioned algebraic structure of finite difference operators with the adaptive representation capacity of multi-grade neural architectures, the proposed approach effectively mitigates the spectral bias and optimization stiffness characteristic of standard Physics-Informed Neural Networks (PINNs).

The core contributions and findings of this study are summarized as follows:
\begin{itemize}
    \item \textbf{Hybrid Finite-Difference and Neural-Basis Formulation:} 
    FD-MGDL establishes a hybrid approximation paradigm wherein finite difference discretizations and neural networks play complementary roles. Discrete finite difference stencils approximate differential operators in the loss formulation, completely bypassing the computational overhead and higher-order derivative attenuation associated with automatic differentiation. Concurrently, the unknown solution is parameterized by a neural network, acting as a global, non-linear continuous basis function. This offers a significantly richer and more flexible approximation space than the locally supported polynomial representations underlying conventional finite difference methods, combining discrete stencil stability with continuous functional representation.

    \item \textbf{Adaptive Grade Refinement and Decoupled Optimization:} 
    The MGDL strategy systematically enhances solution accuracy by introducing shallow neural sub-networks grade-by-grade. Unlike Single-Grade Deep Learning (SGDL), which simultaneously optimizes all parameters of a single deep network and suffers from severe parameter coupling, MGDL decomposes the global learning task into a sequence of localized residual-correction subproblems. At each grade, only the newly added shallow network is trained while previously optimized components remain frozen, substantially reducing optimization complexity and stabilizing gradient descent. This progressive residual-learning mechanism inherently counteracts spectral bias: early grades capture dominant, low-frequency wave features, while subsequent grades sequentially isolate and learn fine, high-frequency residual details. Furthermore, our theoretical analysis establishes a rigorous monotonicity property guaranteeing that the grade-end training loss is non-increasing, providing a solid theoretical foundation for progressive model refinement and adaptive stopping criteria.

    \item \textbf{Superior Performance in High-Frequency Regimes:} 
    Comprehensive numerical benchmarks in 2D and 3D domains across wavenumbers up to $\kappa=200$ demonstrate that FD-MGDL consistently outperforms state-of-the-art neural PDE solvers (including Mscale, FBPINN, and SIREN) as well as classical finite difference schemes. Notably, FD-MGDL maintains stable convergence and spatially controlled error distributions in high-wavenumber and near-resonant regimes where traditional discrete solvers suffer from severe dispersion pollution and matrix ill-conditioning.

    \item \textbf{Robustness in Inhomogeneous Media:} 
    Applying FD-MGDL to heterogeneous velocity benchmarks (such as a concave wave-speed profile) confirms its capacity to resolve complex physical wave dynamics, including caustics and wave focusing. The framework significantly surpasses standard finite difference schemes in numerical stability and wavefield resolution across complex media.
\end{itemize}

Crucially, FD-MGDL preserves the mesh-free evaluation advantage characteristic of neural solvers. While finite difference stencils are employed on a structured grid to formulate derivative residuals during training, the resulting multi-grade architecture parametrizes a continuous function $u_\theta(\mathbf{x})$ over the entire physical domain. Consequently, the trained model can be evaluated instantaneously at any arbitrary coordinate without post-processing spatial interpolation.





\section*{Acknowledgments}
Rui Wang is supported in part by the Natural Science Foundation of China under grants 12571562 and 12171202. Tingting Wu is supported by the Natural Science Foundation of Shandong Province of China under grant ZR2021MA 049. Yuesheng Xu is supported in part by the US National Science Foundation under grant DMS-2208386.

\appendix
{\section{Resource Analysis and Statistical Stability}}
In this appendix, we provide supplementary details on the efficiency analysis and the statistical robustness with respect to random initialization.


\subsection{Computational Complexity and Memory Efficiency}
To further evaluate the computational efficiency of the MGDL framework, we analyze the computational complexity and memory consumption during the training process. The computational benefits of multi-grade architecture have also been investigated in related implicit-representation problems, including training-data-free image restoration \cite{liao2026mgspair}. The peak memory usage (Peak Mem.) is recorded at each optimization stage, including the memory occupied by network parameters, intermediate activations, gradient computation, and temporary tensors generated during backpropagation. In addition, we report the parameter counts (Para.) and total optimization work (Total Opt.), defined as the product of the number of trainable parameters and training epochs, to provide a hardware-independent measurement of the training complexity. 
As a representative benchmark, we consider the experiment in Section~\ref{2d-sin} at wavenumber $\kappa=100$, where FD-MGDL and FD-SGDL are evaluated under identical computational environments. This wavenumber is specifically chosen for its proximity to a resonant frequency, which introduces severe challenges for numerical solvers.

\begin{table}[!htb]
    \centering
    \setlength{\tabcolsep}{4pt}
    \begin{threeparttable}
    \caption{Comparison of computational complexity and memory consumption between FD-MGDL and FD-SGDL.}
    \label{memory_usage}
    \centering
    \begin{tabular}{lccccccc}
        \toprule    
        \multirow{2}{*}{Method} & \multirow{2}{*}{Grade} & Network Structure & \multirow{2}{*}{Para.} & \multirow{2}{*}{Total Opt.} & \multirow{2}{*}{Peak Mem. (MiB)} & \multirow{2}{*}{AC time (s)} & \multirow{2}{*}{TeRSE} \\
        & & (Width $\times$ Depth) & & & & & \\
        \midrule
        \multirow{4}{*}{FD-MGDL}
        & 1 & $256\times2$ & 66,817 & $3.34\times10^{7}$ & $3,640$ & $1,669$ & $9.30\times10^{-1}$\\
        & 2 & $256\times1$ & 66,049 & $1.32\times10^{8}$ & $2,659$ & $4,470$ & $1.06\times10^{-2}$\\
        & 3 & $256\times1$ & 66,049 & $1.32\times10^{8}$ & $2,417$ & $6,933$& $5.56\times10^{-3}$ \\
        & 4 & $256\times1$ & 66,049 & $1.32\times10^{8}$ & $2,418$ & $\mathbf{9,196}$ & $\mathbf{5.56\times10^{-3}}$ \\ 
        \midrule
        FD-SGDL & -- & $256\times5$ & 264,193 & $2.11\times10^{9}$ & $8,794$ & $47,792$ &$1.87\times 10^{-2}$ \\
        \bottomrule
    \end{tabular}
    \end{threeparttable}
\end{table}

The results in Table \ref{memory_usage} demonstrate the computational advantages of the MGDL strategy in terms of both optimization workload and peak memory consumption. The optimization work of an individual grade of FD-MGDL ranges from $3.34\times10^{7}$ to $1.32\times10^{8}$, while the accumulated optimization work over all four grades is approximately $4.30\times10^{8}$. By comparison, FD-SGDL requires approximately $2.11\times10^{9}$ optimization-work units. Consequently, the cumulative optimization workload of FD-MGDL is only about $20.3\%$ of that of FD-SGDL. This substantial difference arises from the distinct optimization strategies of the two methods. FD-SGDL optimizes all parameters of a single deep network simultaneously throughout the training process, resulting in a considerably larger optimization burden. In contrast, FD-MGDL decomposes the original approximation task into a sequence of residual-correction subproblems, each of which is solved by optimizing a relatively small shallow network while the previously trained grades remain fixed. The grade-wise training strategy therefore distributes the overall optimization task across several smaller parameter-optimization subproblems and substantially reduces the cumulative optimization workload.

Regarding memory consumption, the largest memory requirement of FD-MGDL occurs in the first grade, with a peak memory usage of approximately $3,640$ MiB. The subsequent grades require only $2,417$--$2,659$ MiB of peak memory, maintaining a stable memory footprint throughout the residual refinement process. In comparison, FD-SGDL reaches a peak memory consumption of $8,794$ MiB, which is approximately $58.6\%$ higher than that of the maximum memory requirement of FD-MGDL.

The improved computational efficiency of MGDL can be attributed to its hierarchical optimization strategy. In SGDL, the complete solution is approximated by a single deep network, and all intermediate activations of the entire architecture must be retained during backpropagation for gradient evaluation. In contrast, MGDL decomposes the learning process into multiple grades, where each grade employs an independent shallow network to approximate the residual correction. During the optimization of each grade, only the current network requires gradient storage, while the previously trained components remain fixed and do not participate in backpropagation. This sequential residual learning mechanism substantially reduces both optimization complexity and peak memory consumption.

\subsection{Statistical Robustness across Random Seeds}

Because neural network optimization depends on random parameter initialization, we evaluate the statistical sensitivity of the results across different random seeds. While all preceding experiments used a baseline random seed ($\text{seed}=1$), this robustness study supplements the baseline with two independent runs using $\text{seed}=2$ and $\text{seed}=12$. Specifically, we consider the 2D Helmholtz problem with the sinusoidal solution (Section~\ref{2d-sin}) at wavenumber $\kappa=100$, a near-resonant, high-wavenumber regime characterized by pronounced high-frequency oscillations. For each method, all hyperparameters, including network architecture, learning-rate schedule, and total epoch count, are held strictly identical to the baseline setup without any seed-specific retuning. Table~\ref{Tab-seed} reports the sample mean and standard deviation of the TeRSE metric alongside the cumulative computational time across all three seeds.

\begin{table}[!htb]
    \centering
    \setlength{\tabcolsep}{8pt}
    \begin{threeparttable}
    \caption{Statistical Summary over Three Independent Random Seeds.}
    \label{Tab-seed}
    \centering
    \begin{tabular}{lccccc}
        \toprule 
        Method & Random Seed & AC time (s) & Mean time (s) & TeRSE & Mean $\pm$ Std. Dev. \\
        \midrule
        \multirow{3}{*}{FD-MGDL} & 1\ \  & 9,196 & \multirow{3}{*}{9,297} & $5.56\times10^{-3}$ & \multirow{3}{*}{$6.88\times10^{-3}\pm 1.22\times10^{-3}$}\\
        & 2\ \ & 9,450 & & $7.13\times10^{-3}$ &  \\
        & 12  & 9,246 & & $7.96\times10^{-3}$ &  \\
        \midrule
        \multirow{3}{*}{FD-SGDL} & 1\ \  & 47,792 & \multirow{3}{*}{47,715} & $1.87\times10^{-2}$ & \multirow{3}{*}{$3.43\times10^{-2}\pm 2.40\times10^{-2}$}\\
        & 2\ \ & 47,739 &  & $6.19\times10^{-2}$ &  \\
        & 12 & 47,614 &  & $2.23\times10^{-2}$ & \\
        \bottomrule
    \end{tabular}
    \end{threeparttable}
\end{table}

As shown in Table \ref{Tab-seed}, FD-MGDL consistently achieves lower TeRSE values than FD-SGDL across all three random seeds. In particular, FD-MGDL obtains a mean TeRSE of $6.88\times10^{-3}$ with a standard deviation of $1.22\times10^{-3}$, whereas FD-SGDL yields $3.43\times10^{-2}\pm2.40\times10^{-2}$. The substantially smaller standard deviation of FD-MGDL indicates reduced sensitivity to random parameter initialization. These results provide additional empirical evidence that the multi-grade residual-correction strategy improves both optimization stability and computational efficiency.

\section{Detailed Experimental Configurations}

This appendix provides a summary of the experimental configurations used throughout the paper. 

All numerical experiments were conducted on a 64-bit workstation running Windows 11, equipped with an Intel(R) Core(TM) Ultra 9 285K processor operating at 3.70 GHz and 64 GB of system memory. All computations were performed exclusively on the CPU without GPU acceleration. The algorithms were implemented in Python 3.12 using PyTorch 2.12.0. All network parameters, input data, intermediate quantities, and loss evaluations were represented in double precision using \texttt{float64}. All methods were trained in full-batch mode. At each optimization iteration, the complete training set was used to evaluate the loss and update the network parameters. For example, for the two-dimensional tensor-product grids with $m=300$, $500$ and $700$, the corresponding full-batch sizes were $m^2=90,000$; $250,000$ and $490,000$ respectively. For Mscale, the full batch consisted of all randomly sampled training points specified in the corresponding experimental configuration. 

The eight tables below compile the detailed hyperparameter settings of all compared methods, including network architectures, activation functions, learning rate schedules, and training epochs.
FD-MGDL employs the adaptive FD-MGDL algorithm, in which the number of grades is automatically determined according to the problem difficulty. In all experiments, the first grade uses a shallow network with two hidden layers and $\sin$ activation, while each subsequent grade consists of a single hidden layer with ReLU activation. The width of all hidden layers in FD-MGDL is fixed at 256. The hyperparameter settings for FD-MGDL in the four experiments across different wavenumbers $\kappa$ are summarized in Tables \ref{hyperparameters-2dsin-MGDL}, \ref{hyperparameters-2dwave-MGDL}, \ref{hyperparameters-3dsin-MGDL}, \ref{hyperparameters-3dwave-MGDL}, which detail the grade-wise configurations used in the adaptive training process. The corresponding settings for the remaining comparison methods are provided in Tables \ref{hyperparameters-2dsin}, \ref{hyperparameters-2dwave}, \ref{hyperparameters-3dsin}, \ref{hyperparameters-3dwave}.

In our experiments, FD-MGDL and all baselines underwent empirical hyperparameter tuning, with the final configurations reported in Tables \ref{hyperparameters-2dsin-MGDL}--\ref{hyperparameters-3dwave}. We acknowledge that FD-MGDL involves grade-wise learning rates and epoch allocations, which results in a larger number of hyperparameter combinations than the baselines. However, we stress that this larger search space does not render FD-MGDL computationally more expensive during tuning. This is because FD-MGDL’s tuning is performed per grade: each trial only trains the newly added shallow subnetwork (single hidden layer), whose per-run cost is much lower than the full deep model. In contrast, evaluating a single candidate configuration for a baseline method requires retraining the entire deep network from scratch, which is considerably more expensive. Consequently, even though FD-MGDL may evaluate more configurations, its cumulative tuning time is  lower than that of the baseline methods.

\begin{table}[!htb]\centering
    \small
    \setlength{\tabcolsep}{2.5pt}
    \begin{threeparttable}
    \caption{Hyperparameter settings of FD-MGDL for the 2D Helmholtz problem \eqref{2d-Helmholtz-Dirichlet} with the exact solution \eqref{2d-sin-solution}.}
    \label{hyperparameters-2dsin-MGDL}
    \centering
    \begin{tabular}{cclll}
        \toprule
        $\kappa$  & Grade  & $t_{\max}$  & $t_{\min}$  & Epochs \\
        \midrule
        $50$  &6  & $\{10^{-1},10^{-2},10^{-3},10^{-3},10^{-3},10^{-3}\}$ & $\{10^{-2},10^{-3},10^{-4},10^{-3},10^{-3},10^{-3}\}$  & [400,3000,3000,2500,2500,1000]\\
        $100$  &4  & $\{10^{-1},10^{-1},10^{-2},10^{-2}\}$ & $\{10^{-1},10^{-2},10^{-3},10^{-3}\}$  & [500,2000,2000,2000]\\
        $150$  &5  & $\{10^{-1},10^{-1},10^{-2},10^{-3},10^{-2}\}$ & $\{10^{-2},10^{-1},10^{-3},10^{-3},10^{-3}\}$  & [500,1500,2000,2000,2000]\\
        $200$  &4  & $\{10^{-1},10^{-1},10^{-3},10^{-2}\}$ & $\{10^{-1},10^{-2},10^{-3},10^{-3}\}$  & [500,2000,2000,2000]\\
        \bottomrule
    \end{tabular}
    \end{threeparttable}
\end{table}

\begin{table}[!htb]\centering
    \small
    \setlength{\tabcolsep}{4pt}
    \centering
    \caption{Hyperparameter settings of FD-MGDL for the 2D Helmholtz problem \eqref{2d-Helmholtz-Dirichlet} with the exact solution \eqref{2d-wave-solution}.}
    \label{hyperparameters-2dwave-MGDL}
    \begin{tabular}{cclll}
        \toprule
        $\kappa$  & Grade  & $t_{\max}$  & $t_{\min}$  & Epochs \\
        \midrule
        $50$  &3  & $\{10^{-1},10^{-3},10^{-1}\}$ & $\{10^{-1},10^{-4},10^{-3}\}$  & [500,2000,2000]\\
        $100$  &4  & $\{10^{-1},10^{-2},10^{-3},10^{-3}\}$ & $\{10^{-2},10^{-3},10^{-4},10^{-4}\}$  & [500,2000,2000,2000]\\
        $150$  &3  & $\{10^{-1},10^{-2},10^{-3}\}$ & $\{10^{-2},10^{-3},10^{-3}\}$  & [500,1000,1000]\\
        $200$  &4  & $\{10^{-1},10^{-1},10^{-3},10^{-2}\}$ & $\{10^{-2},10^{-2},10^{-4},10^{-2}\}$  & [500,1000,1000,1000]\\
        \bottomrule
    \end{tabular}
\end{table}

\begin{table}[!htb]\centering
    \small
    \setlength{\tabcolsep}{4pt}
    \caption{Hyperparameter settings of FD-MGDL for the 3D Helmholtz problem \eqref{3d-Helmholtz-Dirichlet} with the exact solution \eqref{3d-sin-solution}.}
    \label{hyperparameters-3dsin-MGDL}
    \centering
    \begin{tabular}{cclll}
        \toprule
        $\kappa$  & Grade  & $t_{\max}$  & $t_{\min}$  & Epochs \\
        \midrule
        $20$  &5  & $\{10^{-1},10^{-2},10^{-3},10^{-3},10^{-3}\}$ & $\{10^{-1},10^{-3},10^{-4},10^{-3},10^{-3}\}$  & [500,2000,2000,2000,2000]\\
        $30$  &3  & $\{10^{-1},10^{-2},10^{-3}\}$ & $\{10^{-1},10^{-4},10^{-4}\}$  & [500,2500,3000]\\
        $40$ & 7  & $\{10^{-1},10^{-1},10^{-3}(\times5)\}$  & $\{10^{-1},10^{-1},10^{-3}(\times5)\}$  & $[500,2000,2000(\times5)]$\\
        $50$  &4  & $\{10^{-1},10^{-2},10^{-3},10^{-3}\}$ & $\{10^{-1},10^{-3},10^{-4},10^{-4}\}$  & [500,2000,2000,2000]\\
        \bottomrule
    \end{tabular}
\end{table}

\begin{table}[!htb]\centering
    \small
    \setlength{\tabcolsep}{4pt}
    \caption{Hyperparameter settings of FD-MGDL for the 3D Helmholtz problem \eqref{3d-Helmholtz-Dirichlet} with the exact solution \eqref{3d-wave-solution}.}
    \label{hyperparameters-3dwave-MGDL}
    \centering
    \begin{tabular}{cclll}
        \toprule
        $\kappa$  & Grade  & $t_{\max}$  & $t_{\min}$  & Epochs \\
        \midrule
        $20$  &4  & $\{10^{-1},10^{-1},10^{-3},10^{-3}\}$ & $\{10^{-2},10^{-1},10^{-3},10^{-3}\}$  & [500,1000,1000,1000]\\
        $30$  &4  & $\{10^{-1},10^{-1},10^{-2},10^{-3}\}$ & $\{10^{-4},10^{-1},10^{-3},10^{-4}\}$  & [500,1000,1000,1000]\\
        \bottomrule
    \end{tabular}
\end{table}

\begin{table}[!htb]\centering
    \small
    \setlength{\tabcolsep}{2.5pt}
    \begin{threeparttable}
    \caption{Hyperparameter settings of FD-SGDL, Mscale, FBPINN, SIREN, PINN and Pre-PINN for the 2D Helmholtz problem \eqref{2d-Helmholtz-Dirichlet} with the exact solution \eqref{2d-sin-solution}.}
    \label{hyperparameters-2dsin}
    \centering
    \begin{tabular}{cllll}
        \toprule
        $\kappa$  & Method  & Architecture (width $\times$ depth)  & Activation  & $\{t_{\max},t_{\min}\}$ / Epochs \\
        \midrule
        \multirow{4}{*}{$50$}
        & FD-SGDL, PINN, Pre-PINN  & $256\times7$   & $\sin\times2+$ReLU$\times5$   & $\{10^{-3},10^{-4}\}$ / 15,000 \\
        & SIREN  & $256\times7$  & $\sin\times7$  & $\{10^{-3},10^{-4}\}$ / 15,000 \\
        & Mscale  & $180\times3$ with scales $\{1,2,4,8,16,32\}$  & $\sin\times3$  & $\{10^{-3},-\}$ / 15,000 \\
        & FBPINN  & $40\times2$ with subdomains $15\times15$  & $\tanh\times2$  & $\{10^{-3},-\}$ / 15,000 \\
        \midrule
        \multirow{4}{*}{$100$}
        & FD-SGDL, PINN, Pre-PINN   & $256\times5$   & $\sin\times2+$ReLU$\times3$   & $\{10^{-2},-\}$ / 8,000 \\
        & SIREN  & $256\times5$   & $\sin\times5$   & $\{10^{-2},-\}$ / 8,000 \\
        & Mscale  & $160\times3$ with scales $\{1,2,4,8,16\}$  & $\sin\times3$  & $\{10^{-3},-\}$ / 8,000 \\
        & FBPINN  & $40\times2$ with subdomains $12\times12$  & $\tanh\times2$  & $\{10^{-3},-\}$ / 8,000 \\
        \midrule
        \multirow{4}{*}{$150$}
        & FD-SGDL, PINN, Pre-PINN  & $256\times6$   & $\sin\times2+$ReLU$\times4$   & $\{10^{-2},10^{-4}\}$ / 10,000 \\
        & SIREN  & $256\times6$   & $\sin\times6$   & $\{10^{-2},10^{-4}\}$ / 10,000 \\
        & Mscale  & $150\times3$ with scales $\{1,2,4,8,16,32,64\}$  & $\sin\times3$  & $\{5\times10^{-4},-\}$ / 10,000 \\
        & FBPINN  & $38\times2$ with subdomains $14\times14$  & $\tanh\times2$  & $\{10^{-3},-\}$ / 10,000 \\
        \midrule
        \multirow{4}{*}{$200$}
        & FD-SGDL, PINN, Pre-PINN  & $256\times5$   & $\sin\times2+$ReLU$\times3$   & $\{10^{-2},10^{-4}\}$ / 8,000 \\
        & SIREN  & $256\times5$   & $\sin\times5$   & $\{10^{-2},10^{-4}\}$ / 8,000 \\
        & Mscale  & $135\times3$ with scales $\{1,2,4,8,16,32,64\}$  & $\sin\times3$  & $\{5\times10^{-4},-\}$ / 8,000 \\
        & FBPINN  & $38\times2$ with subdomains $13\times13$  & $\tanh\times2$  & $\{10^{-3},-\}$ / 8,000 \\
        \bottomrule
    \end{tabular}
    \end{threeparttable}
\end{table}

\begin{table}[!htb]\centering
    \small
    \setlength{\tabcolsep}{4pt}
    \centering
    \caption{Hyperparameter settings of FD-SGDL, Mscale and FBPINN for the 2D Helmholtz problem \eqref{2d-Helmholtz-Dirichlet} with the exact solution \eqref{2d-wave-solution}.}
    \label{hyperparameters-2dwave}
    \centering
    \begin{tabular}{cllll}
        \toprule
        $\kappa$  & Method   & Architecture (width $\times$ depth)  & Activation  & $\{t_{\max},t_{\min}\}$ / Epochs \\
        \midrule
        \multirow{3}{*}{$50$}
        & FD-SGDL   & $256\times4$   & $\sin\times2+$ReLU$\times2$   & $\{10^{-1},10^{-4}\}$ / 6,000 \\
        & Mscale  & $140\times3$ with scales $\{1,2,4,8,16\}$  & $\sin\times3$  & $\{5\times10^{-4},-\}$ / 6,000 \\
        & FBPINN  & $32\times2$ with subdomains $13\times13$  & $\tanh\times2$  & $\{10^{-3},-\}$ / 6,000 \\
        \midrule
        \multirow{3}{*}{$100$}
        & FD-SGDL  & $256\times5$   & $\sin\times2+$ReLU$\times3$   & $\{10^{-1},10^{-2}\}$ / 8,000 \\
        & Mscale  & $160\times3$ with scales $\{1,2,4,8,16\}$  & $\sin\times3$  & $\{5\times10^{-4},-\}$ / 8,000 \\
        & FBPINN  & $32\times2$ with subdomains $15\times15$  & $\tanh\times2$  & $\{10^{-3},-\}$ / 8,000 \\
        \midrule
        \multirow{3}{*}{$150$}
        & FD-SGDL  & $256\times4$   & $\sin\times2+$ReLU$\times2$   & $\{10^{-1},10^{-4}\}$ / 5,000 \\
        & Mscale  & $120\times3$ with scales $\{1,2,4,8,16,32,64\}$  & $\sin\times3$  & $\{5\times10^{-4},-\}$ / 5,000 \\
        & FBPINN  & $32\times2$ with subdomains $13\times13$  & $\tanh\times2$  & $\{10^{-3},-\}$ / 5,000 \\
        \midrule
        \multirow{3}{*}{$200$}
        & FD-SGDL  & $256\times5$   & $\sin\times2+$ReLU$\times3$   & $\{10^{-1},10^{-2}\}$ / 5,000 \\
        & Mscale  & $164\times3$ with scales $\{1,2,4,8,16\}$  & $\sin\times3$  & $\{5\times10^{-4},-\}$ / 5,000 \\
        & FBPINN  & $32\times2$ with subdomains $15\times15$  & $\tanh\times2$  & $\{10^{-3},-\}$ / 5,000 \\
        \bottomrule
    \end{tabular}
\end{table}

\begin{table}[!htb]\centering
    \small
    \setlength{\tabcolsep}{4pt}
    \caption{Hyperparameter settings of FD-SGDL, Mscale and FBPINN for the 3D Helmholtz problem \eqref{3d-Helmholtz-Dirichlet} with the exact solution \eqref{3d-sin-solution}.}
    \label{hyperparameters-3dsin}
    \centering
    \begin{tabular}{cllll}
        \toprule
        $\kappa$  & Method  & Architecture (width $\times$ depth)  & Activation  & $\{t_{\max},t_{\min}\}$ / Epochs \\
        \midrule
        \multirow{3}{*}{$20$}
        & FD-SGDL  & $256\times6$   & $\sin\times2+$ReLU$\times4$   & $\{10^{-3},-\}$ / 10,000 \\
        & Mscale  & $164\times3$ with scales $\{1,2,4,8,16,32\}$  & $\sin\times3$  & $\{3\times10^{-4},-\}$ / 10,000 \\
        & FBPINN  & $40\times2$ with subdomains $5\times6\times6$  & $\tanh\times2$  & $\{10^{-3},-\}$ / 10,000 \\
        \midrule
        \multirow{3}{*}{$30$}
        & FD-SGDL   & $256\times4$   & $\sin\times2+$ReLU$\times2$   & $\{10^{-1},10^{-2}\}$ / 8,000 \\
        & Mscale    & $127\times3$ with scales $\{1,2,4,8,16,32\}$  & $\sin\times3$  & $\{3\times10^{-4},-\}$ / 8,000 \\
        & FBPINN    & $37\times2$ with subdomains $5\times5\times5$  & $\tanh\times2$  & $\{10^{-3},-\}$ / 8,000 \\
        \midrule
        \multirow{3}{*}{$40$}
        & FD-SGDL  & $256\times8$   & $\sin\times2+$ReLU$\times6$   & $\{10^{-2},10^{-3}\}$ / 15,000 \\
        & Mscale   & $168\times3$ with scales $\{1,2,4,8,16,32,64,128\}$  & $\sin\times3$  & $\{2\times10^{-4},-\}$ / 15,000 \\
        & FBPINN   & $[43,44]$ with subdomains $6\times6\times6$  & $\tanh\times2$  & $\{10^{-3},-\}$ / 15,000 \\
        \midrule
        \multirow{3}{*}{$50$}
        & FD-SGDL   & $256\times5$   & $\sin\times2+$ReLU$\times3$   & $\{10^{-2},-\}$ / 8,000 \\
        & Mscale    & $147\times3$ with scales $\{1,2,4,8,16,32\}$  & $\sin\times3$  & $\{2\times10^{-4},-\}$ / 8,000 \\
        & FBPINN    & $43\times2$ with subdomains $5\times5\times5$  & $\tanh\times2$  & $\{10^{-3},-\}$ / 8,000 \\
        \bottomrule
    \end{tabular}
\end{table}

\begin{table}[!htb]\centering
    \small
    \setlength{\tabcolsep}{4pt}
    \begin{threeparttable}
    \caption{Hyperparameter settings of FD-SGDL, Mscale and FBPINN for the 3D Helmholtz problem \eqref{3d-Helmholtz-Dirichlet} with the exact solution \eqref{3d-wave-solution}. {}}
    \label{hyperparameters-3dwave}
    \centering
    \begin{tabular}{cllll}
        \toprule
        $\kappa$  & Method  & Architecture (width $\times$ depth)  & Activation  & $\{t_{\max},t_{\min}\}$ / Epochs \\
        \midrule
        \multirow{3}{*}{$20$}
        & FD-SGDL  & $256\times5$   & $\sin\times2+$ReLU$\times3$   & $\{10^{-1},10^{-3}\}$ / 5,000 \\
        & Mscale   & $147\times3$ with scales $\{1,2,4,8,16,32\}$  & $\sin\times3$  & $\{2\times10^{-4},-\}$ / 5,000 \\
        & FBPINN   & $32\times2$ with subdomains $6\times6\times6$  & $\tanh\times2$  & $\{10^{-3},-\}$ / 5,000 \\
        \midrule
        \multirow{3}{*}{$30$}
        & FD-SGDL   & $256\times5$   & $\sin\times2+$ReLU$\times3$   & $\{10^{-1},10^{-2}\}$ / 5,000 \\
        & Mscale    & $147\times3$ with scales $\{32,32,32,32,32,32\}$  & $\sin\times3$  & $\{10^{-4},-\}$ / 5,000 \\
        & FBPINN    & $32\times2$ with subdomains $6\times6\times6$  & $\tanh\times2$  & $\{10^{-3},-\}$ / 5,000 \\
        \bottomrule
    \end{tabular}
    \begin{tablenotes}
        \footnotesize
        \item The constant-scale configuration of Mscale is adopted directly from the original MscaleDNN paper \cite{liu2020multi} (see Section 5.4, configuration ``MscaleDNN-2(1)''). This configuration is used solely for faithful reproduction of the original MscaleDNN methodology.
    \end{tablenotes}
    \end{threeparttable}
\end{table}

\bibliographystyle{unsrt}
\bibliography{refs}

\end{document}